%% file: manuscript_article.tex
\documentclass{article}
\input{head.tex}

\title{A Robust Hessian-based Trust Region Algorithm for Spherical Conformal Parameterizations}

\author{Zhong-Heng Tan\thanks{School of Mathematics and Shing-Tung Yau Center, Southeast University,
Nanjing 211189; Nanjing Center for Applied Mathematics, Nanjing 211135, People’s Republic of China.}
\and Tiexiang Li\thanks{Corresponding author. School of Mathematics and Shing-Tung Yau Center, Southeast University,
Nanjing 211189; Nanjing Center for Applied Mathematics, Nanjing 211135, People’s Republic of China.}
\and Wen-Wei Lin\thanks{Nanjing Center for Applied Mathematics, Nanjing 211135, People's Republic of China; Department of Applied Mathematics, Yang Ming Chiao Tung University, Hsinchu 300, Taiwan.}
\and Shing-Tung Yau\thanks{Yau Mathematical Sciences Center, Tsinghua University, Beijing 100084, China.}}

\date{}

\begin{document}
\maketitle

\abstract{Surface parameterizations are widely applied in computer graphics, medical imaging and transformation optics. In this paper, we rigorously derive the gradient vector and Hessian matrix of the discrete conformal energy for spherical conformal parameterizations of simply connected closed surfaces of genus-$0$. In addition, we give the sparsity structure of the Hessian matrix, which leads to a robust Hessian-based trust region algorithm for the computation of spherical conformal maps. Numerical experiments demonstrate the local quadratic convergence of the proposed algorithm with low conformal distortions. We subsequently propose an application of our method to surface registrations that still maintains local quadratic convergence.}



\maketitle

\section{Introduction}

The Poincar\'e-Klein-Koebe theorem is a fundamental theorem in Riemann geometry that states that a simply connected Riemann surface $\mathcal{M}$ is conformally equivalent to either a unit sphere $\mathbb{S}^2$, a complex plane $\mathbb{C}$ or a unit disk $\mathbb{D}$. From a numerical point of view, for a given discrete triangular mesh of $\mathcal{M}$, we require that the conformal parameterization between $\mathcal{M}$ and the canonical shape is a bijective map from $\mathcal{M}$ to $\mathbb{S}^2$ (Riemann surface of genus-zero) or $\mathbb{D}$ (Riemann surface with a single boundary) while minimizing the total angle distortion induced by the Dirichlet energy \cite{LBPS02,MDMM02,MHWW17}. Conformal parameterizations, also known as angle-preserving parameterizations, preserve the intersection angle of two arbitrary intersecting curves on the surface up to the map. In other words, conformal parameterizations preserve the local shapes of the surfaces. The development of digital 3D object technologies has enabled the representation of smooth surfaces in the real world through high-resolution meshes on computers.
High-resolution meshes have the ability to characterize the intricate geometrical structure of surfaces. Nevertheless, high resolution also makes computations and processes more challenging on complicated meshes.
Recently, surface parameterizations have received increasing attention due to their ability to transform intricate geometries into simply shaped regions. It is clear that the parameterization of a surface is not necessarily unique, even when the target region is given. Conformal parameterization is one of the most commonly used parameterization methods.
As a result, it is used in several fields, such as medical imaging \cite{XGYW04,JNTL07,KCRS13,RSWZ17}, texture mapping \cite{SHSA00,LBPS02}, and transformation optics \cite{LXHC14,Hoss22}.

A detailed overview of surface conformal parameterizations can be found in classical surveys \cite{MFKH05,ASEP06,KHBL07,PAGU08,XDWZ11}.
Numerous methods for conformal parameterizations have been proposed, including boundary first flattening \cite{RSKC17}, discrete conformal equivalence \cite{MGBS21}, harmonic energy minimization \cite{RLZW14} and solving the Laplacian-Beltrami equation with a particular Dirac delta function as the right-hand term \cite{SHSA00}.
In earlier years, based on the time-flow technique, Jin et al. \cite{MJJK08} and Yang et al. \cite{YLRG09} proposed the discrete Ricci flow with conformal circle packing metric, which is a negative gradient flow of some convex energy and can be accelerated by the Newton method. However, the evaluations of the related coefficients at each iterative step are somewhat complicated due to the circle packing metric. Huang et al. \cite{WQXD14} proposed the quasi-implicit Euler method (QIEM) in view of the nonlinear heat diffusion process with normalization on $\mathbb{S}^2$. However, the convergence can be very slow, and there is thus far no theory to support its convergence.
With the great development of GPU-based computation in recent years, parallelizable algorithms were subsequently proposed \cite{GPYL20}.

In 2015, Choi et al. \cite{PTKC15} applied the quasi-conformal approach to spherical parameterization and proposed the FLASH algorithm, which produces the composition of the quasi-conformal maps with the same Beltrami coefficient to obtain the conformal map.
In 2019, Yueh et al. \cite{MHTL19} proposed a north-south hemisphere alternating iteration, called spherical conformal energy minimization (SCEM), which alternatively maps the north and south poles to infinity and fixes the corresponding hemisphere while updating the other hemisphere. From numerical experiments, both FLASH and SCEM are satisfactory with high accuracy and effectiveness compared to those of the other previously mentioned methods.
FLASH is a direct method that solves two double-sized linear systems compared to SCEM, which utilizes stereographic projection to transform the unit sphere $\mathbb{S}^2$ to the extended complex plane $\overline{\mathbb{C}}$. The resulting map, which is composed of the inverse stereographic projection and the computed conformal map, is the ideal conformal parameterization from the given Riemann surface of genus zero to $\mathbb{S}^2$. In addition, it simplifies the spherical constraint problem in $\mathbb{R}^3$ into an unconstrained problem in $\overline{\mathbb{C}}$, making it highly efficient. However, the stereographic projection maps the north pole of $\mathbb{S}^2$ to infinity and others to $\mathbb{C}$. Therefore, the computational error near the north pole becomes relatively large in practical applications.
To mitigate this error, the north-south hemisphere alternatingly iterative SCEM reduces the conformal distortion near the poles and has numerically sublinear convergence and asymptotically $R$-linear convergence. However, larger computational errors in conformal quantities are transferred to the junction of the equator. From a theoretical point of view, both FLASH and SCEM consider minimizing the Dirichlet energy on the extended complex plane $\overline{\mathbb{C}}$. From a numerical point of view, both FLASH and SCEM may get larger angle distortions near the junction region of the computational conformal map because they are connected by two submaps via stereographic projection.

Different from FLASH and SCEM, in this paper, we solve the spherical conformal energy optimization problem to obtain the conformal map by directly employing the spherical coordinates for the representation of the conformal energy. Unlike the stereographic projection, the spherical coordinate representation is not a conformal map. We derive the related gradient vector and Hessian matrix and develop a robust Hessian-based trust region (HBTR) algorithm that has local quadratic convergence. The main contributions of this paper can be divided into three folds.

\begin{itemize}
\item We use the spherical coordinates to represent the conformal energy on $\mathbb{S}^2$ and develop the HBTR algorithm to directly minimize the discrete conformal energy, which is described as the difference between the discrete Dirichlet energy on $\mathbb{S}^2$ and the image area of the conformal map. In \Cref{subsec:convdiscrete}, numerical experiments demonstrate that the discrete conformal energy, mean and standard deviation (SD) are reduced to $1/4$ and $1/2$, respectively, when the mesh size is refined by a half.

\item We give the explicit derivation of the gradient vector and Hessian matrix and present the significant sparsity of the Hessian matrix. This benefits the development of fast computations for calculating the Newton iterations. Combined with the trust region technique, we thus propose the robust HBTR algorithm.

\item Numerical experiments and comparisons with existing algorithms demonstrate the advantages of our method in terms of conformality and robustness. The computational cost of the HBTR algorithm is of the same order of magnitude as that of FLASH and SCEM but slightly more expensive for most models due to spending a high percentage of time selecting the convergence region. Furthermore, the bijectivity can almost be guaranteed because the conformal energy is expressed by the unity of spherical coordinates. For few meshes resulting in nonbijective maps, folding can be easily removed by performing a postprocessing algorithm \cite{Floa03}.
\end{itemize}

This paper is organized as follows. The \Cref{sec:conformal} provides a brief review of the conformal map and conformal energy. Then, in \Cref{sec:Hessian} we present the theoretical derivation of the gradient vector and Hessian matrix. In \Cref{sec:algorithm}, we describe the proposed algorithm for spherical conformal parameterization that uses the sparsity property of the Hessian matrix. The numerical performance and comparison with other methods are presented in \Cref{sec:experiments}. In \Cref{sec:registration}, we demonstrate an application of the algorithm to surface registration. A concluding mark is given in \Cref{sec:conclusion}.

The frequently used notations in this paper are listed here.
Bold letters, e.g., $\mathbf{a},\mathbf{s}$ denote vectors.
$\mathbf{a}_i$ denotes the $i$-th entry of $\mathbf{a}$.
$\mathbf{1}_{m\times n}$ denotes the $m\times n$ matrix of all ones. The notation without subscript $\mathbf{1}$ denotes the vector of all ones with proper dimension.
$\mathbf{e}_i$ denotes the $i$-th column of the identity matrix with the proper dimension.
$\big[ A \big]_{ij}$ denotes the $(i,j)$-th entry of matrix $A$.
$\diag(\mathbf{a})$ denotes the diagonal matrix with the $(i,i)$-th entry being $\mathbf{a}_i$.
$[v_i,v_j]$ denotes the edge formed by $v_i$ and $v_j$.
$[v_i,v_j,v_k]$ denotes the triangle formed by $v_i$, $v_j$ and $v_k$, and
$\big|[v_i,v_j,v_k]\big|$ denotes the area of that triangle.
The other notations are defined wherever they appear.

\section{Conformal map and Conformal Energy} \label{sec:conformal}

In this section, we briefly review the continuous conformal map and the conformal energy. Readers can refer to \cite{XGST07book,XGST20book,Hutc91,UPKP93,SHSA00} for more details.

Let $\mathcal{M} = \mathbf{r}(u,v)$ be a surface and $f$ be a continuous and bijective vector-valued map on $\mathcal{M}$, which maps $\mathcal{M}$ to another surface $\widetilde{\mathcal{M}} = \tilde{\mathbf{r}}(u,v)$ with $(u,v)\in \Omega \subset \mathbb{R}^2$. Let $\gamma_1,\gamma_2 \subset \mathcal{M}$ be two arbitrary curves intersecting at a point. Then, we call $f$ conformal if the intersecting angle of $\gamma_1$ and $\gamma_2$ equals that of $f(\gamma_1)$ and $f(\gamma_2)$ in $\widetilde{\mathcal{M}}$. Equivalently, $f$ is conformal if and only if the first fundamental forms $I(u,v)$ and $\tilde{I}(u,v)$ of $\mathcal{M}$ and $\widetilde{\mathcal{M}}$ with respect to $(u,v)$, i.e.,
\begin{align*}
    I(u,v) = \begin{bmatrix}
        \langle \mathbf{r}_u, \mathbf{r}_u \rangle &
        \langle \mathbf{r}_u, \mathbf{r}_v \rangle \\
        \langle \mathbf{r}_v, \mathbf{r}_u \rangle &
        \langle \mathbf{r}_v, \mathbf{r}_v \rangle \\
    \end{bmatrix},\quad
    \tilde{I}(u,v) = \begin{bmatrix}
        \langle \tilde{\mathbf{r}}_u, \tilde{\mathbf{r}}_u \rangle &
        \langle \tilde{\mathbf{r}}_u, \tilde{\mathbf{r}}_v \rangle \\
        \langle \tilde{\mathbf{r}}_v, \tilde{\mathbf{r}}_u \rangle &
        \langle \tilde{\mathbf{r}}_v, \tilde{\mathbf{r}}_v \rangle \\
    \end{bmatrix},
\end{align*}
satisfy
\begin{align*}
    \tilde{I}(u,v) = \eta(u,v) I(u,v),
\end{align*}
where $\mathbf{r}_u:= \frac{\partial \mathbf{r}}{\partial u}$, $\mathbf{r}_v:= \frac{\partial \mathbf{r}}{\partial v}$ and $\eta(u,v)$ is a positive scalar function on $\Omega$. Then, the conformal energy functional \cite{Hutc91} of $f$ is defined as
\begin{align} \label{def: continuous conformal energy}
    E_C(f) = \frac{1}{2}\int_{\mathcal{M}} \|\nabla_\mathcal{M} f\|_F^2 \mathrm{d}\sigma - \mathcal{A}(f),
\end{align}
where $\nabla_\mathcal{M}$ is the tangential gradient, $\mathrm{d}\sigma$ is the area element on $\mathcal{M}$ and $\mathcal{A}(f)$ is the area of the image surface $\widetilde{\mathcal{M}} = f(\mathcal{M})$. Let us note that the first term in \eqref{def: continuous conformal energy} is the Dirichlet energy functional of $f$.
It has been proven that \cite{Hutc91,UPKP93}
\begin{itemize}
\item $E_C(f)\geq 0$,
\item $E_C(f) = 0$ if and only if $f$ is conformal.
\end{itemize}

From these properties, it is reasonable to adopt conformal energy minimization (CEM) to obtain the conformal map. In the continuous scheme, if the target surface $\widetilde{\mathcal{M}}$ is given, its area remains constant. Thus, a method to solve \eqref{def: continuous conformal energy} is to minimize the Dirichlet energy, which is a quadratic functional. In the discrete scheme, the area is actually not a constant. Introducing the discrete area term typically yields a better conformal parameterization. It is easy to verify that the optimization problem \eqref{def: continuous conformal energy} has a trivial solution $f = constant$. In this case, all vertices shrink into a point, and the area of the formed region becomes $0$, which violates the requirement of $f$ being bijective. We refer to the trivial solution as 'degeneration'. Directly minimizing the Dirichlet energy without any additional constraints often leads to degeneration in practical computations. The area term is a natural penalty for the parameterization and hence can greatly weaken the degeneration during iteration \cite{Hutc91}. On the other hand, introducing the area term can further decrease the conformal distortion. For disk parameterizations, Yueh et al.\cite{MHWW17} proposed a disk conformal parameterization algorithm through Dirichlet energy minimization, while a novel algorithm minimizing the disk conformal energy was later proposed by Kuo et al. \cite{YCWW21}, in which they derived a particular and simple area expression of a discrete disk represented by polar coordinates to design a fast algorithm. As a consequence, conformal energy minimization \cite{YCWW21} gives significantly less conformal energy and angle distortion compared  to those of the Dirichlet energy minimization \cite{MHWW17}.

For the spherical parameterization, it is clearly feasible to introduce the area term. However, it is a significant challenge that the area of the discrete sphere, which is the sum of areas of all triangles formed by vertices, cannot be simply expressed as in the disk case, even if both vertices are of unit length. Simultaneously, the gradient vector and Hessian matrix of the conformal energy also have complicated expressions. In the next section, we carefully derive the expression of the conformal energy for the spherical parameterization with spherical coordinates, along with its gradient vector and Hessian matrix. The Hessian matrix has a special sparsity structure as the Laplacian matrix, which guides the fast construction and computation associated with it.

\section{Discrete Conformal Energy on Closed Surfaces} \label{sec:Hessian}

Let $M$ be a discrete closed surface of genus-$0$ composed of triangles, called a triangulation. Given  that $M$ has $n$ vertices, we denote the vertex set, edge set and triangle face set of $M$ by $\mathcal{V}(M) = \{v_i,i = 1,2,\cdots,n\}$,
$\mathcal{E}(M) = \{[v_i,v_j]\}$ and
$\mathcal{F}(M) = \{T_{ijk} := [v_i,v_j,v_k]\},$
respectively.

In this paper, we aim to find a piecewise linear map $f$ that maps surface $M$ to a discrete unit sphere $\mathbb{S}^2$ conformally, in which all of the vertices are unit vectors.
Every point in a triangle can be represented as the convex combination of its vertices via barycentric coordinates,
\begin{align} \label{eq:bary}
    f(\hat{v}) = \lambda_i(\hat{v}) f(v_i) + \lambda_j(\hat{v}) f(v_j) + \lambda_k(\hat{v}) f(v_k), \text{ for } \hat{v} \in T_{ijk},
\end{align}
where
\begin{align*}
    \lambda_i(\hat{v}) = \frac{|[\hat{v},v_j,v_k]|}{|T_{ijk}|}, \quad
    \lambda_j(\hat{v}) = \frac{|[v_i,\hat{v},v_k]|}{|T_{ijk}|}, \quad
    \lambda_k(\hat{v}) = \frac{|[v_i,v_j,\hat{v}]|}{|T_{ijk}|}.
\end{align*}
After $f(v_i), i = 1,2,\cdots,n$ is obtained, the discrete sphere is also formed. Hence, we compute $\mathbf{f} = [\mathbf{f}_1^\top, \mathbf{f}_2^\top, \cdots, \mathbf{f}_n^\top]^\top \in \mathbb{R}^{n\times 3}$ with $\mathbf{f}_i = f(v_i)  \in \mathbb{R}^{1\times 3}$ for the given triangulation of $M$.

From the perspective of the conformal energy \eqref{def: continuous conformal energy}, the conformal map $f$ is obtained by solving the optimization problem
\begin{align}
    \mathbf{f}^* = \argmin_{\mathbf{f}: \|\mathbf{f}_i\| = 1,i = 1,2,\cdots,n} E_C(f), \label{opt:conformal}
\end{align}
where the discrete conformal energy $E_C(f)$ of $f$ on $M$ is given by
\begin{align} \label{def:CE}
    E_C(f) = \frac{1}{2} \langle L\mathbf{f}, \mathbf{f} \rangle - A(f),
\end{align}
in which $L$ is the Laplacian matrix defined as
\begin{align} \label{def:L}
    \big[L\big]_{ij} = \begin{cases}
        -w_{ij} & \text{if}~i\neq j,~[v_i,v_j] \in \mathcal{E}(M),\\
        \sum_{k\in \mathcal{N}(i)} w_{ik} & \text{if}~i = j,\\
        0 & \text{if}~[v_i,v_j] \notin \mathcal{E}(M),
    \end{cases}
\end{align}
with cotangent weights $w_{ij} = \frac{1}{2} (\cot\alpha_{ij} + \cot\alpha_{ji})$. Here, $\alpha_{ij}$ and $\alpha_{ji}$ are the angles opposite to the edge $[v_i,v_j]$
\footnote{A closed surface has no boundary. Therefore, an edge must correspond to $2$ opposite angles.}
and $\mathcal{N}(i)$ is the index set of the adjacent vertices $v_i$, as shown in
\cref{fig:OppositeAngleOneRing}. Readers can refer to \cite[Chap. 24]{XGST20book} for a detailed derivation of the discrete Dirichlet energy on triangulation. $A(f)$ in \eqref{def:CE} is the area of the triangulation $f(M)$, which can be written as the sum of areas of all triangles in $f(M)$, that is, $
A(f) = \sum_{T_{ijk} \in \mathcal{F}(M)} |f(T_{ijk})|.
$

Let $v_{ij} = v_i - v_j$ and $\mathbf{f}_{ij} = \mathbf{f}_{i} - \mathbf{f}_{j}$ for simplicity,
$\alpha_{ij}(f)$ and $\alpha_{ki}(f)$ be the angle, respectively, opposite to the edges $[\mathbf{f}_i,\mathbf{f}_j]$ and $[\mathbf{f}_k,\mathbf{f}_i] \in \mathcal{E}\big(f(M)\big)$,
And let $L(f)$ be a Laplacian matrix with respect to $f$ as $L$, by replacing the cotangent weights as $w_{ij}(f) = \frac{1}{2} (\cot\alpha_{ij}(f) + \cot\alpha_{ji}(f))$ in the target sphere $\mathbb{S}^2$. $A(f)$ and its gradient with respect to $\mathbf{f}$ can be represented as in the following lemma.

\begin{lemma} \label{lemma:Af}
The area of the image of $f$ and its gradient can be represented as
\begin{align*}
        &A(f) = \frac{1}{2} \langle L(f)\mathbf{f}, \mathbf{f} \rangle, \\ 
        &\nabla_\mathbf{f} A(f) = L(f)\mathbf{f}.
    \end{align*}
\end{lemma}

\begin{proof}
The area of a triangle $f(T_{ijk}) \in \mathcal{F}(f(M))$ can be calculated by
\begin{align}
    |f(T_{ijk})| &= \frac{1}{2}\|\mathbf{f}_{ij}\times \mathbf{f}_{jk}\|
    = \frac{\|\mathbf{f}_{ij}\times \mathbf{f}_{jk}\|^2}{4|f(T_{ijk})|}\nonumber \\
    &= \frac{\langle \mathbf{f}_{jk}\times (\mathbf{f}_{ij}\times \mathbf{f}_{jk}), \mathbf{f}_{ij} \rangle}{4|f(T_{ijk})|}
    = \frac{\langle \mathbf{f}_{jk}\times (\mathbf{f}_{ki}\times \mathbf{f}_{ij}), \mathbf{f}_{ij} \rangle}{4|f(T_{ijk})|} \label{eq:deripApistart}\\
    &= \frac{\big\langle \langle \mathbf{f}_{ij}, \mathbf{f}_{jk} \rangle \mathbf{f}_{ki} - \langle \mathbf{f}_{jk}, \mathbf{f}_{ki} \rangle \mathbf{f}_{ij}, \mathbf{f}_{ij} \big\rangle}{4|f(T_{ijk})|} \label{eq:deripApiend}\\
    &= -\frac{
    \langle \mathbf{f}_{ki}, \mathbf{f}_{ij} \rangle \langle \mathbf{f}_{jk}, \mathbf{f}_{jk} \rangle +
    \langle \mathbf{f}_{ij}, \mathbf{f}_{jk} \rangle \langle \mathbf{f}_{ki}, \mathbf{f}_{ki} \rangle +
    \langle \mathbf{f}_{jk}, \mathbf{f}_{ki} \rangle \langle \mathbf{f}_{ij},\mathbf{f}_{ij} \rangle}{8|f(T_{ijk})|}\nonumber\\
    &= \frac{1}{4} \big( \cot\alpha_{jk}(f) \|\mathbf{f}_{jk}\|^2 + \cot\alpha_{ki}(f) \|\mathbf{f}_{ki}\|^2 + \cot\alpha_{ij}(f) \|\mathbf{f}_{ij}\|^2 \big), \label{eq:3cotf}
\end{align}
In the last equation \eqref{eq:3cotf}, we use the formula $\cot\alpha_{ij}(f) = -\frac{\langle \mathbf{f}_{jk}, \mathbf{f}_{ki} \rangle}{2|T_{ijk}|}$.
Hence, the area $A(f)$ can be represented as
\begin{align*}
    A(f) &= \sum_{T_{ijk} \in \mathcal{F}(M)} |f(T_{ijk})| \\
    &= \frac{1}{4}\sum_{T_{ijk} \in \mathcal{F}(M)} \big( \cot\alpha_{jk}(f) \|\mathbf{f}_{jk}\|^2 + \cot\alpha_{ki}(f) \|\mathbf{f}_{ki}\|^2 + \cot\alpha_{ij}(f) \|\mathbf{f}_{ij}\|^2 \big)\\
    &= \frac{1}{2} \sum_{ij:[v_i,v_j]\in \mathcal{E}(M)} w_{ij}(f) \|\mathbf{f}_{ij}\|^2
    = \frac{1}{2} \langle L(f) \mathbf{f}, \mathbf{f} \rangle.
\end{align*}
Then, by using the derivation of \eqref{eq:deripApistart}-\eqref{eq:deripApiend}, we have
\begin{align}
    \frac{\partial |f(T_{ijk})|}{\partial \mathbf{f}_i} =& \frac{2\mathbf{f}_{jk}\times (\mathbf{f}_{ij} \times \mathbf{f}_{jk})}{8|f(T_{ijk})|}
    = \frac{\langle \mathbf{f}_{ij}, \mathbf{f}_{jk} \rangle \mathbf{f}_{ki} - \langle \mathbf{f}_{jk}, \mathbf{f}_{ki} \rangle \mathbf{f}_{ij}}{4|f(T_{ijk})|} \nonumber\\
    =& \frac{1}{2}\big(\cot\alpha_{ij}(f) \mathbf{f}_{ij} - \cot\alpha_{ki}(f) \mathbf{f}_{ki}\big) \in \mathbb{R}^{1\times 3}, \label{eq:pfV}
\end{align}
Hence, the derivative of $A(f)$ with respect to $\mathbf{f}_i$ is
\begin{align*} 
    \frac{\partial A(f)}{\partial \mathbf{f}_i} =
    \sum_{T_{ijk} \in \mathcal{F}(M)} \frac{\partial |f(T_{ijk})|}{\partial \mathbf{f}_i}
    = \frac{1}{2}\sum_{T_{ijk} \in \mathcal{F}(M)} \big(\cot\alpha_{ij}(f) \mathbf{f}_{ij} - \cot\alpha_{ki}(f) \mathbf{f}_{ki}\big)
    = \mathbf{e}_i^\top L(f) \mathbf{f},
\end{align*}
for $i = 1,2,\cdots,n$. This completes the proof.
\end{proof}

\begin{figure}[htp]
    \centering
\subfloat[Opposite angles]{\label{subfig:OppositeAngle}
    \includegraphics[clip,trim = {1cm 1cm 0cm 1cm},width = 0.4\textwidth]{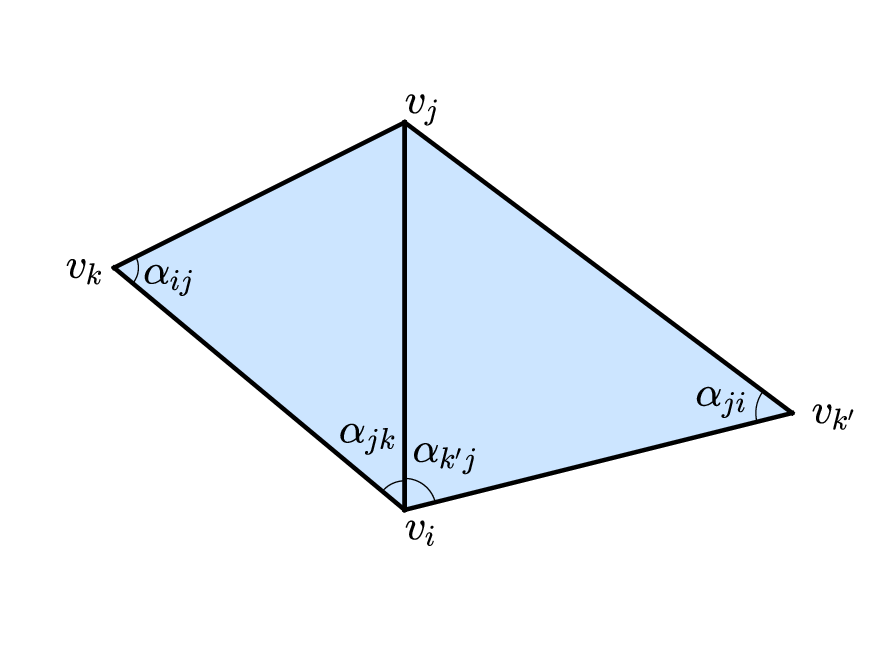}
}
\subfloat[Adjacent vertices indices]{\label{subfig:OneRing}
    \includegraphics[clip,trim = {1.5cm 0.5cm 1.4cm 0cm},width = 0.3\textwidth]{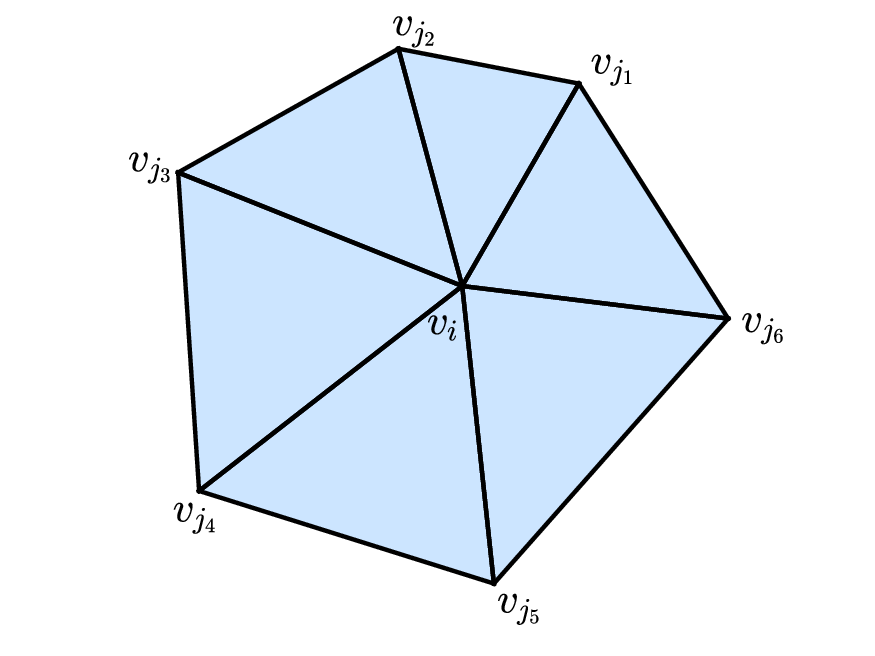}
}
\caption{Illustrations of opposite angles and adjacent vertices indices set. (a) $\alpha_{ij}$ and $\alpha_{ji}$ are the pair of opposite angles of edge $[v_i,v_j]$; $\alpha_{jk}$ and $\alpha_{k'j}$ are the opposite angles of edge $[v_j,v_k]$ in $T_{ijk}$ and $[v_{k'},v_j]$ in $T_{k'ji}$, respectively.
(b) $v_{j_1},\cdots,v_{j_6}$ are the adjacent vertices of $v_i$. Hence, $\mathcal{N}(i) = \{j_1,j_2,\cdots,j_6\}$.}
    \label{fig:OppositeAngleOneRing}
\end{figure}

From \eqref{def:CE} and \Cref{lemma:Af}, we have
\begin{align}
    &E_C(f) = \frac{1}{2} \langle D(f)\mathbf{f},\mathbf{f} \rangle, \quad \text{  with  } D(f) = L - L(f),  \label{eq:CE=D(f)}\\
    &\nabla_\mathbf{f} E_C(f) = D(f) \mathbf{f}. \label{eq:pEpf}
\end{align}
Clearly, $D(f)$ has the same sparsity structure as $L$ and $L(f)$, and the corresponding cotangent weight is $\big[D(f)\big]_{ij} := \tilde{w}_{ij} = w_{ij} - w_{ij}(f)$.

\begin{remark}
If $f^*$ is a conformal map, the inner angles of the triangles in $M$ are identical to those in $f^*(M)$. We have $w_{ij}(f^*) = w_{ij}$ and $D(f) = \mathbf{0}$, which leads to $E_C(f^*) = 0$.
\end{remark}

Let $\mathbf{f} = [\mathbf{x},\mathbf{y},\mathbf{z}]$ and $\vectx(\mathbf{f}) = [\mathbf{x}^\top,\mathbf{y}^\top,\mathbf{z}^\top]^\top$ be its vectorization. Since the vertices of $\mathbf{f}$ are on the sphere, we adopt the spherical coordinates to represent all vertices by
\begin{subequations}
    \begin{equation} \label{eq:xyz}
    \mathbf{x} = \cos\boldsymbol{\theta} \odot \sin{\boldsymbol{\phi}},\quad
    \mathbf{y} = \sin\boldsymbol{\theta} \odot \sin{\boldsymbol{\phi}},\quad
    \mathbf{z} = \cos{\boldsymbol{\phi}},
\end{equation}
and let
\begin{equation} \label{eq:uvw}
    \mathbf{u} = \cos\boldsymbol{\theta} \odot \cos{{\boldsymbol{\phi}}},\quad
    \mathbf{v} = \sin\boldsymbol{\theta} \odot \cos{{\boldsymbol{\phi}}},\quad
    \mathbf{w} = \sin{{\boldsymbol{\phi}}},
\end{equation}
\end{subequations}
where $\boldsymbol{\theta}, {{\boldsymbol{\phi}}} \in \mathbb{R}^{n}$ are the azimuth and elevation angle vectors of the corresponding vertices of $\mathbf{f}$, respectively, and $\odot$ is the Hadamard product.
Similar to the stereographic projection in \cite{MHTL19,PTKC15}, the spherical coordinate representation eliminates the unit-length constraint and transforms the Cartesian coordinate $[\mathbf{x},\mathbf{y},\mathbf{z}]$ into the spherical coordinate $[\boldsymbol{\theta},{{\boldsymbol{\phi}}}]$, reducing the input variable scale by $1/3$.
Furthermore, the spherical coordinate projects the whole sphere into a bounded region $[0,2\pi] \times [0,\pi]$, which avoids the computational error near the north pole \cite{PTKC15} or the equator \cite{MHTL19} by using the stereographic projection. In the further discussion, we denote that $\nabla := [\nabla_{\boldsymbol{\theta}}^\top, \nabla_{{\boldsymbol{\phi}}}^\top]^\top$ with respect to $(\boldsymbol{\theta},{{\boldsymbol{\phi}}})$ unless a special illustration in the rest of the paper. The gradient of any scalar $a$ and Jacobi matrix of any vector $\mathbf{a}$ with respect to $(\boldsymbol{\theta},{{\boldsymbol{\phi}}})$ are represented as
\[
\nabla a = \left[ \frac{\partial a}{\partial \boldsymbol{\theta}_1}, \cdots, \frac{\partial a}{\partial \boldsymbol{\theta}_n}, \frac{\partial a}{\partial {{\boldsymbol{\phi}}}_1}, \cdots, \frac{\partial a}{\partial {{\boldsymbol{\phi}}}_n} \right]^\top,\quad
\nabla \mathbf{a} = \left[ \frac{\partial \mathbf{a}}{\partial \boldsymbol{\theta}_1}, \cdots, \frac{\partial \mathbf{a}}{\partial \boldsymbol{\theta}_n}, \frac{\partial \mathbf{a}}{\partial {\boldsymbol{\phi}}_1}, \cdots, \frac{\partial \mathbf{a}}{\partial {\boldsymbol{\phi}}_n} \right],
\]
respectively. Furthermore, the gradient of $\mathbf{x},\mathbf{y},\mathbf{z},\mathbf{u},\mathbf{v},\mathbf{w}$ can be represented by themselves,
\begin{align} \label{eq:nablaxyzuvw}
    \begin{array}{l@{}l@{}l@{}l}
    &\nabla \mathbf{x} = [-\diag(\mathbf{y}),\diag(\mathbf{u})], \quad
    &\nabla \mathbf{y} = [\diag(\mathbf{x}),\diag(\mathbf{v})], \quad
    &\nabla \mathbf{z} = [\mathbf{0}_{n\times n},-\diag(\mathbf{w})],\\
    &\nabla \mathbf{u} = [-\diag(\mathbf{v}),-\diag(\mathbf{x})], \quad
    &\nabla \mathbf{v} = [\diag(\mathbf{u}),-\diag(\mathbf{y})], \quad
    &\nabla \mathbf{w} = [\mathbf{0}_{n\times n},\diag(\mathbf{z})].
    \end{array}
\end{align}


We now vectorize the gradient \eqref{eq:pEpf} and let
\begin{align} \label{eq:pqr}
\mathbf{p} = D(f)\mathbf{x},\quad \mathbf{q} = D(f)\mathbf{y},\quad \mathbf{r} = D(f)\mathbf{z}.
\end{align}
From \eqref{eq:CE=D(f)}, we have $E_C(f) = \frac{1}{2}(\mathbf{x}^\top \mathbf{p}+\mathbf{y}^\top \mathbf{q}+\mathbf{z}^\top \mathbf{r})$. Then, from \eqref{eq:pEpf} and \eqref{eq:nablaxyzuvw}, the gradient of $E_C(f)$ with respect to $(\boldsymbol{\theta},{\boldsymbol{\phi}})$ is
\begin{align} \label{eq:gradientE}
    \mathbf{g}:=\nabla E_C &= \nabla\vectx(\mathbf{f})^\top \nabla_{\vectx(\mathbf{f})} E_C\\
    &=\nabla \mathbf{x}^\top \mathbf{p} + \nabla \mathbf{y}^\top \mathbf{q} + \nabla \mathbf{z}^\top \mathbf{r}
    = \begin{bmatrix}
        -\diag(\mathbf{y}) \mathbf{p} +
        \diag(\mathbf{x}) \mathbf{q}\\
        \diag(\mathbf{u}) \mathbf{p} +
        \diag(\mathbf{v}) \mathbf{q} -
        \diag(\mathbf{w}) \mathbf{r}
    \end{bmatrix}
    \in \mathbb{R}^{2n\times 1}.
\end{align}
Moreover, the Hessian matrix is also obtained
\begin{align} \label{eq:HessianE0}
    H := \nabla\nabla E_C &= (\nabla \mathbf{x}^\top \nabla \mathbf{p} + \nabla \mathbf{y}^\top \nabla \mathbf{q} + \nabla \mathbf{z}^\top \nabla \mathbf{r}) \\
    &+\begin{bmatrix}
        -\diag(\mathbf{p}) \nabla \mathbf{y} + \diag(\mathbf{q}) \nabla \mathbf{x} \\
        \diag(\mathbf{p}) \nabla \mathbf{u} + \diag(\mathbf{q}) \nabla \mathbf{v} - \diag(\mathbf{r}) \nabla \mathbf{w}
    \end{bmatrix}\in \mathbb{R}^{2n\times 2n}, \nonumber
\end{align}
where
\begin{align*}
   \nabla \mathbf{p} = D(f)\nabla \mathbf{x} -\sum_{\ell = 1}^n \mathbf{x}_\ell\nabla(L(f)\mathbf{e}_\ell),\
   \nabla \mathbf{q} = D(f)\nabla \mathbf{y} -\sum_{\ell = 1}^n \mathbf{y}_\ell\nabla(L(f)\mathbf{e}_\ell),\
   \nabla \mathbf{r} = D(f)\nabla \mathbf{z} -\sum_{\ell = 1}^n \mathbf{z}_\ell\nabla(L(f)\mathbf{e}_\ell).
\end{align*}
Via the chain rule
\begin{align*}
    \nabla (L(f)\mathbf{e}_\ell) =& \nabla_\mathbf{x} (L(f)\mathbf{e}_\ell)\nabla \mathbf{x} + \nabla_\mathbf{y} (L(f)\mathbf{e}_\ell)\nabla \mathbf{y} + \nabla_\mathbf{z} (L(f)\mathbf{e}_\ell)\nabla \mathbf{z} \in \mathbb{R}^{n\times 2n},
\end{align*}
we have
\begin{equation} \label{eq:chain}
\resizebox{\textwidth}{!}{$
    \begin{bmatrix}
         \sum_{\ell = 1}^n \mathbf{x}_\ell\nabla(L(f)\mathbf{e}_\ell)\\
         \sum_{\ell = 1}^n \mathbf{y}_\ell\nabla(L(f)\mathbf{e}_\ell)\\
         \sum_{\ell = 1}^n \mathbf{z}_\ell\nabla(L(f)\mathbf{e}_\ell)
    \end{bmatrix}
    =
    \begin{bmatrix}
        \sum_{\ell = 1}^n \mathbf{x}_\ell\nabla_\mathbf{x}(L(f)\mathbf{e}_\ell) & \sum_{\ell = 1}^n \mathbf{x}_\ell\nabla_\mathbf{y}(L(f)\mathbf{e}_\ell) & \sum_{\ell = 1}^n \mathbf{x}_\ell\nabla_\mathbf{z}(L(f)\mathbf{e}_\ell)\\
        \sum_{\ell = 1}^n \mathbf{y}_\ell\nabla_\mathbf{x}(L(f)\mathbf{e}_\ell) & \sum_{\ell = 1}^n \mathbf{y}_\ell\nabla_\mathbf{y}(L(f)\mathbf{e}_\ell) & \sum_{\ell = 1}^n \mathbf{y}_\ell\nabla_\mathbf{z}(L(f)\mathbf{e}_\ell)\\
        \sum_{\ell = 1}^n \mathbf{z}_\ell\nabla_\mathbf{x}(L(f)\mathbf{e}_\ell) & \sum_{\ell = 1}^n \mathbf{z}_\ell\nabla_\mathbf{y}(L(f)\mathbf{e}_\ell) & \sum_{\ell = 1}^n \mathbf{z}_\ell\nabla_\mathbf{z}(L(f)\mathbf{e}_\ell)
    \end{bmatrix} \begin{bmatrix}
        \nabla \mathbf{x}\\ \nabla \mathbf{y}\\ \nabla \mathbf{z}
    \end{bmatrix}.
    $}
\end{equation}
Plugging \eqref{eq:chain} into \eqref{eq:HessianE0}, we can rewrite the Hessian matrix as
\begin{align}  \label{eq:HessianE}
    H &=
    \begin{bmatrix}
        \nabla \mathbf{x} \\ \nabla \mathbf{y} \\ \nabla \mathbf{z}
    \end{bmatrix}^\top
    (L_1 - L_2)
    \begin{bmatrix}
        \nabla \mathbf{x} \\ \nabla \mathbf{y} \\ \nabla \mathbf{z}
    \end{bmatrix}
    - \begin{bmatrix}
        \diag(\mathbf{p}\odot \mathbf{x} + \mathbf{q}\odot \mathbf{y}) &
        \diag(\mathbf{p}\odot \mathbf{v} - \mathbf{q}\odot \mathbf{u}) \\
        \diag(\mathbf{p}\odot \mathbf{v} - \mathbf{q}\odot \mathbf{u}) &
        \diag(\mathbf{p}\odot \mathbf{x} + \mathbf{q}\odot \mathbf{y} + \mathbf{r}\odot \mathbf{z})
    \end{bmatrix}\\
    &\equiv H_1 - H_2 - K, \label{eq:HessianE=2}
\end{align}
where
\begin{align*}
    H_s &=
    \left[
        \nabla \mathbf{x}^\top , \nabla \mathbf{y}^\top , \nabla \mathbf{z}^\top
    \right]
    L_s
    \begin{bmatrix}
        \nabla \mathbf{x} \\ \nabla \mathbf{y} \\ \nabla \mathbf{z}
    \end{bmatrix}, s = 1,2, \\
    K &= \begin{bmatrix}
        \diag(\mathbf{p}\odot \mathbf{x} + \mathbf{q}\odot \mathbf{y}) &
        \diag(\mathbf{p}\odot \mathbf{v} - \mathbf{q}\odot \mathbf{u}) \\
        \diag(\mathbf{p}\odot \mathbf{v} - \mathbf{q}\odot \mathbf{u}) &
        \diag(\mathbf{p}\odot \mathbf{x} + \mathbf{q}\odot \mathbf{y} + \mathbf{r}\odot \mathbf{z})
    \end{bmatrix}
\end{align*}
with
\begin{align} \label{eq:L1L2}
\resizebox{\textwidth}{!}{$
L_1 =
\begin{bmatrix}
        D(f) & & \\
        & D(f) & \\
        & & D(f)
    \end{bmatrix},\quad
L_2 = \begin{bmatrix}
        \sum_{\ell = 1}^n \mathbf{x}_\ell\nabla_\mathbf{x}(L(f)\mathbf{e}_\ell) & \sum_{\ell = 1}^n \mathbf{x}_\ell\nabla_\mathbf{y}(L(f)\mathbf{e}_\ell) & \sum_{\ell = 1}^n \mathbf{x}_\ell\nabla_\mathbf{z}(L(f)\mathbf{e}_\ell)\\
        \sum_{\ell = 1}^n \mathbf{y}_\ell\nabla_\mathbf{x}(L(f)\mathbf{e}_\ell) & \sum_{\ell = 1}^n \mathbf{y}_\ell\nabla_\mathbf{y}(L(f)\mathbf{e}_\ell) & \sum_{\ell = 1}^n \mathbf{y}_\ell\nabla_\mathbf{z}(L(f)\mathbf{e}_\ell)\\
        \sum_{\ell = 1}^n \mathbf{z}_\ell\nabla_\mathbf{x}(L(f)\mathbf{e}_\ell) & \sum_{\ell = 1}^n \mathbf{z}_\ell\nabla_\mathbf{y}(L(f)\mathbf{e}_\ell) & \sum_{\ell = 1}^n \mathbf{z}_\ell\nabla_\mathbf{z}(L(f)\mathbf{e}_\ell)
    \end{bmatrix}.
    $}
\end{align}

In the matrix $L_2$, the Jacobian matrix of each column of $L(f)$ should be considered. For the $\ell$-th column entries $L(f)\mathbf{e}_\ell$, it is clear that $\left[L(f)\right]_{i\ell} \neq 0$ if and only if $i \in \mathcal{N}(\ell)$ or $i=\ell$ by the definition of the Laplacian matrix in \eqref{def:L}. The nondiagonal entries are negative cotangent weights, and the diagonal entries are the sums of the cotangent weights. Hence, the Jacobian matrix of $L(f)\mathbf{e}_\ell$ is formed by the gradient of cotangent weights $w_{ij}(f)$, especially the gradient of cotangent functions $c_{ij}(f):=\cot\alpha_{ij}(f)$. Before discussing entries of $\nabla L\mathbf{e}_\ell$, we first give a lemma for the sparsity structure of the Jacobian block $\sum_{\ell = 1}^n \mathbf{a}_\ell\nabla_\mathbf{b} (L(f)\mathbf{e}_\ell)$, $\mathbf{a},\mathbf{b} = \mathbf{x},\mathbf{y},\mathbf{z}$.

\begin{lemma} \label{lemma:sparsity}
Each subblock $\sum_{\ell = 1}^n \mathbf{a}_\ell\nabla_\mathbf{b} (L(f)\mathbf{e}_\ell)$ of $L_2$ in \eqref{eq:L1L2} with $\mathbf{a},\mathbf{b} = \mathbf{x},\mathbf{y},\mathbf{z}$ is of identical sparsity structure to $L$.
\end{lemma}
\begin{proof}
By the cotangent formula in \eqref{eq:pfV}, one can observe that the gradients with respect to $\mathbf{x},\mathbf{y},\mathbf{z}$ are the $1$st, $2$nd and $3$rd entries of that with respect to $\mathbf{f}$, respectively. Hence, $\sum_{\ell = 1}^n \mathbf{a}_\ell\nabla_\mathbf{b} (L(f)\mathbf{e}_\ell)$ has identical sparsity for all $\mathbf{a},\mathbf{b} = \mathbf{x},\mathbf{y},\mathbf{z}$.
Here, we denote the nonzero indices set of $\mathbf{e}_i^\top \nabla_\mathbf{b}(L(f)\mathbf{e}_j)$ as
\[
S(i,j) = \{k~|~\big[\nabla_\mathbf{b}(L(f)\mathbf{e}_j)\big]_{ik} \neq 0\}.
\]
Clearly, the nonzero indices set of $\mathbf{e}_i^\top \sum_{\ell = 1}^n \mathbf{a}_\ell\nabla_\mathbf{b} (L(f)\mathbf{e}_\ell)$ is $\bigcup_{\ell = 1}^n S(i,\ell)$.
\begin{itemize}
\item For $i \neq j$ and $[v_i,v_j] \notin \mathcal{E}(M)$, the entries are $0$, and therefore, the gradients are also $0$. Hence,
\begin{align} \label{eq:Sij0}
            S(i,j) = \varnothing.
        \end{align}
\item For $i \neq j$ and $[v_i,v_j] \in \mathcal{E}(M)$, by the definition of cotangent weight $w_{ij}(f) = \frac{1}{2}\big(c_{ij}(f) + c_{ji}(f)\big)$, we have $\mathbf{e}_i^\top \nabla_\mathbf{b} (L(f) \mathbf{e}_j) = -\nabla_\mathbf{b} w_{ij}(f) = -\frac{1}{2}\big(\nabla_\mathbf{b} c_{ij}(f) + \nabla_\mathbf{b} c_{ji}(f)\big)$. As shown in \Cref{subfig:OppositeAngle}, one can see that $w_{ij}(f)$ relates to only $4$ vertices $\mathbf{f}_i,\mathbf{f}_j,\mathbf{f}_k,\mathbf{f}_{k'}$. Therefore, the entries of $\nabla_\mathbf{b} w_{ij}(f)$ are $0$ except for the $i,j,k,k'$-th entries; that is,
\begin{align} \label{eq:Sij}
            S(i,j) = \{i,j,k,k'\} \subset \mathcal{N}(i)\cup \{i\}.
        \end{align}
\item For $i=j$, we have $\mathbf{e}_i^\top \nabla_\mathbf{b} (L(f) \mathbf{e}_i) = \nabla_\mathbf{b} \big[L(f)\big]_{ii} = \sum_{j\in \mathcal{N}(i)} \nabla_\mathbf{b} w_{ij}(f)$. Hence, we can easily verify that $\mathbf{e}_i^\top \nabla_\mathbf{b} (L_f \mathbf{e}_i)$ relates to the whole adjacent vertices of $\mathbf{f}_i$ from \Cref{subfig:OneRing}. It immediately follows that
\begin{align} \label{eq:Sjj}
            S(i,i) = \mathcal{N}(i)\cup \{i\}.
        \end{align}
\end{itemize}
Combining \eqref{eq:Sij0}, \eqref{eq:Sij} and \eqref{eq:Sjj}, we have
$
S(i,\ell) \subset \mathcal{N}(i)\cup \{i\} = S(i,i), \text{ for } \ell \neq i.
$
It follows that
$
\bigcup_{\ell = 1}^n S(i,\ell)
= S(i,i) = \mathcal{N}(i)\cup \{i\}.
$
Since the index set of nonzero entries of $\mathbf{e}_i^\top L$ is $\mathcal{N}(i)\cup \{i\}$, the lemma is obtained immediately.
\end{proof}

The \Cref{lemma:sparsity} demonstrates that $\big[\sum_{\ell = 1}^n \mathbf{a}_\ell\nabla_\mathbf{b} (L(f)\mathbf{e}_\ell)\big]_{ij} \neq 0$ if and only if $[v_i,v_j]\in\mathcal{E}(M)$ or $i = j$ for every $i,j$. As a result, $L_2$ in \eqref{eq:L1L2} is stacked by $9$ matrices with the same sparsity as $L$ in $3\times 3$ form. Additionally, one can further observe that $H$ in \eqref{eq:HessianE} is also stacked by $4$ matrices with the same sparsity as $L$ in $2\times 2$ form. We summarize this idea as the following theorem.
\begin{theorem} \label{thm:sparsity}
$H$ in \eqref{eq:HessianE} is of identical sparsity to
$\mathbf{1}_{2\times 2} \otimes L.$
\end{theorem}
\begin{proof}
By the Hessian matrix representation of $H$ in \eqref{eq:HessianE} and the \Cref{lemma:sparsity}, the proof is obtained.
\end{proof}

\Cref{thm:sparsity} demonstrates the high sparsity of the Hessian matrix. This property confirms the feasibility of practically solving the large-scale linear system
\begin{align} \label{eq:newtonstep}
    H\mathbf{s} = -\mathbf{g},
\end{align}
which inspires us to apply a Newton-type algorithm for solving the optimization problem \eqref{opt:conformal}. Furthermore, in the representation of $H$ in \eqref{eq:HessianE}, $L_1 - L_2 = \nabla_{\vectx(\mathbf{f})} \nabla_{\vectx(\mathbf{f})} E_C$ is the Hessian matrix of the conformal energy with respect to $\mathbf{x},\mathbf{y},\mathbf{z}$. The term $\big[ \nabla \mathbf{x}^\top,\nabla \mathbf{y}^\top,\nabla \mathbf{z}^\top \big]^\top$ in \eqref{eq:HessianE} is the Jacobian matrix of $\mathbf{x},\mathbf{y},\mathbf{z}$ with respect to $\boldsymbol{\theta}$ and $\boldsymbol{\phi}$, which is stacked by $3\times 2$ diagonal matrices. The matrix $K$ in \eqref{eq:HessianE} is also stacked by $2\times 2$ diagonal matrices. These structures are invariant for arbitrary parameterization, and therefore, the sparsity structure of Hessian matrix $H$ is also invariant. Hence, this property is also available on other parameterizations of closed surfaces, including other expressions of spheres and other target regions.

We proceed further to analyze the entries of Hessian matrix $H$.
The entries of $L_1$ are derived in \eqref{eq:pqr} and \eqref{eq:CE=D(f)}, and the entries of $K$ in \eqref{eq:HessianE} are derived by \eqref{eq:xyz}-\eqref{eq:uvw} and \eqref{eq:pqr}.
We now give the specific derivation for entries of $L_2$ in \eqref{eq:HessianE}. First, we give the gradients of
\[
c_{ij}(f) = -\frac{\langle \mathbf{f}_{ki},\mathbf{f}_{jk} \rangle}{2\left|f(T_{ijk})\right|},\quad
c_{jk}(f) = -\frac{\langle \mathbf{f}_{ij},\mathbf{f}_{ki} \rangle}{2\left|f(T_{ijk})\right|},\quad
c_{ki}(f) = -\frac{\langle \mathbf{f}_{jk},\mathbf{f}_{ij} \rangle}{2\left|f(T_{ijk})\right|}
\]
in the triangle $[\mathbf{f}_i,\mathbf{f}_j,\mathbf{f}_k]$, where $c_{ij}(f):= \cot\alpha_{ij}(f)$ as before. Since $c_{ij}(f), c_{jk}(f), c_{ki}(f)$ only relates to vertices $\mathbf{f}_i,\mathbf{f}_j,\mathbf{f}_k$, their gradients with respect to other vertices are $0$. The gradients with respect to $\mathbf{f}_j$ are calculated by
\begin{subequations}
\begin{align}
    \frac{\partial}{\partial \mathbf{f}_j} c_{ij}(f) &= \frac{1}{|f(T_{ijk})|} \left[ - \frac{\partial |f(T_{ijk})|}{\partial \mathbf{f}_j} c_{ij}(f) - \frac{ 1 }{2} \mathbf{f}_{ki} \right] \label{eq:pcijpj},\\
    \frac{\partial}{\partial \mathbf{f}_j} c_{jk}(f) &= \frac{1}{|f(T_{ijk})|} \left[ - \frac{\partial |f(T_{ijk})|}{\partial \mathbf{f}_j} c_{jk}(f) + \frac{ 1 }{2} \mathbf{f}_{ki} \right] \label{eq:pcjkpj},\\
    \frac{\partial}{\partial \mathbf{f}_j} c_{ki}(f) &= \frac{1}{|f(T_{ijk})|} \left[ - \frac{\partial |f(T_{ijk})|}{\partial \mathbf{f}_j} c_{ki}(f) + \frac{1}{2} (\mathbf{f}_{jk} - \mathbf{f}_{ij}) \right] ,\label{eq:pckipj}
\end{align}
\end{subequations}
with
\begin{subequations}
\begin{equation}
    \frac{\partial |f(T_{ijk})|}{\partial \mathbf{f}_i} =
    \frac{1}{2}\big(c_{ij}(f) \mathbf{f}_{ij} - c_{ki}(f) \mathbf{f}_{ki}\big), \label{eq:pApi}
\end{equation}
\begin{equation}
    \frac{\partial |f(T_{ijk})|}{\partial \mathbf{f}_j} =
    \frac{1}{2}\big(c_{jk}(f) \mathbf{f}_{jk} - c_{ij}(f) \mathbf{f}_{ij}\big), \label{eq:pApj}
\end{equation}
\begin{equation}
    \frac{\partial |f(T_{ijk})|}{\partial \mathbf{f}_k} =
    \frac{1}{2}\big(c_{ki}(f) \mathbf{f}_{ki} - c_{jk}(f) \mathbf{f}_{jk}\big). \label{eq:pApk}
\end{equation}
\end{subequations}
The others can be obtained by rotating the subscript $i,j,k$ in turn.

Now, we derive the entries of matrix $L_2$ in \eqref{eq:HessianE}. Let us consider the block $\sum_{\ell = 1}^n \mathbf{a}_\ell \nabla_\mathbf{b} (L(f)\mathbf{e}_\ell)$ for $\mathbf{a},\mathbf{b} = \mathbf{x},\mathbf{y},\mathbf{z}$, the $(i,j)$-th entry of which is
\begin{align*}
    \left[\sum_{\ell = 1}^n \mathbf{a}_\ell\nabla_\mathbf{b} (L(f)\mathbf{e}_\ell)\right]_{ij} &=
    \sum_{\ell = 1}^n \mathbf{a}_\ell \frac{\partial}{\partial \mathbf{b}_j} \big[L(f)\big]_{i\ell}
    = \mathbf{a}_i \frac{\partial}{\partial \mathbf{b}_j}  \big[L(f)\big]_{ii} + \sum_{\ell \in \mathcal{N}(i)} \mathbf{a}_\ell \frac{\partial}{\partial \mathbf{b}_j}  \big[L(f)\big]_{i\ell}\\
    &= \mathbf{a}_i \frac{\partial}{\partial \mathbf{b}_j}  \sum_{\ell\in \mathcal{N}(i)} w_{i\ell}(f) - \sum_{\ell \in \mathcal{N}(i)} \mathbf{a}_\ell \frac{\partial}{\partial \mathbf{b}_j} w_{i\ell}(f)\\
    &= \sum_{\ell \in \mathcal{N}(i)} \mathbf{a}_{i\ell} \frac{\partial}{\partial \mathbf{b}_j}  w_{i\ell}(f).
\end{align*}

(i) For $j \neq i$, there are only three vertices $v_j,v_k,v_{k'}$ related to $v_j$ among all adjacent vertices of $v_i$, as shown in \Cref{fig:OppositeAngleOneRing}. Therefore,
\begin{align}
&\left[\sum_{\ell = 1}^n \mathbf{a}_\ell\nabla_\mathbf{b} (L(f)\mathbf{e}_\ell)\right]_{ij} = \sum_{\ell \in \{j,k,k'\}} \mathbf{a}_{i\ell} \frac{\partial}{\partial \mathbf{b}_j} w_{i\ell}(f)\nonumber\\
=& \mathbf{a}_{ik} \frac{\partial}{\partial \mathbf{b}_j} w_{ik}(f) +
\mathbf{a}_{ik'} \frac{\partial}{\partial \mathbf{b}_j} w_{ik'}(f) +
\mathbf{a}_{ij} \frac{\partial}{\partial \mathbf{b}_j} w_{ij}(f)\nonumber\\
=& \frac{1}{2} \left[\left(
\mathbf{a}_{ij} \frac{\partial}{\partial \mathbf{b}_j} c_{ij}(f) -
\mathbf{a}_{ki} \frac{\partial}{\partial \mathbf{b}_j} c_{ki}(f) \right) +
\left(
\mathbf{a}_{ik'} \frac{\partial}{\partial \mathbf{b}_j} c_{ik'}(f) -
\mathbf{a}_{ji} \frac{\partial}{\partial \mathbf{b}_j} c_{ji}(f)
\right)\right].\label{eq:nxLyij}
\end{align}
One can see that the terms in the first bracket relate only to triangle $T_{ijk}$, while the terms in the second bracket relate to triangle $T_{k'ji}$, which are of the same form. This characteristic is also similar to $L$, whose cotangent weight is $w_{ij} = \frac{1}{2}(c_{ij}+c_{ji})$ with $c_{ij},c_{ji}$ also relating to $T_{ijk}$ and $T_{k'ji}$, respectively. Hence, we only give the specific representation of the first bracket of \eqref{eq:nxLyij}. The second bracket is obtained similarly.
By \eqref{eq:pcijpj}-\eqref{eq:pckipj}, we have
\begin{align}
&\mathbf{a}_{ij} \frac{\partial}{\partial \mathbf{b}_j} c_{ij}(f) +
\mathbf{a}_{ik} \frac{\partial}{\partial \mathbf{b}_j} c_{ki}(f)\nonumber\\
=& -\frac{1}{2|f(T_{ijk})|} \left[
\frac{2\partial |f(T_{ijk})|}{\partial \mathbf{a}_i} \frac{2\partial |f(T_{ijk})|}{\partial \mathbf{b}_j} + (2\mathbf{a}_{ki}\mathbf{b}_{jk} - \mathbf{a}_{jk}\mathbf{b}_{ki})
 \right] . \label{eq:bajni}
\end{align}

(ii) For $j = i$, the related vertices are the whole adjacent vertices of $v_i$ and itself, i.e., $\mathcal{N}(i)\cup \{i\}$. Hence,
\begin{align}
&\left[\sum_{\ell = 1}^n \mathbf{b}_\ell\nabla_\mathbf{a} (L(f)\mathbf{e}_\ell)\right]_{ii}
= \sum_{\ell \in \mathcal{N}(i)} \mathbf{a}_{i\ell} \frac{\partial}{\partial \mathbf{b}_i} w_{i\ell}(f)\nonumber\\
=& \frac{1}{2}\sum_{\{j,k\}\in S_{\mathcal{E}}(i)} \left( \mathbf{a}_{ij} \frac{\partial}{\partial \mathbf{b}_i} c_{ij}(f) -
\mathbf{a}_{ki} \frac{\partial}{\partial \mathbf{b}_i} c_{ki}(f) \right),\nonumber 
\end{align}
where $S_{\mathcal{E}}(i) = \big\{\{j,k\}~|~j,k\in \mathcal{N}(i), [v_j,v_k] \in \mathcal{E}(M)\big\}$.
Similarly, via the gradient formulas \eqref{eq:pcijpj}-\eqref{eq:pckipj}, we have
\begin{align}
&\mathbf{a}_{ij} \frac{\partial}{\partial \mathbf{b}_i} c_{ij}(f) -
\mathbf{a}_{ki} \frac{\partial}{\partial \mathbf{b}_i} c_{ki}(f)\nonumber\\
=& -\frac{1}{2|f(T_{ijk})|} \left[
\frac{2\partial |f(T_{ijk})|}{\partial \mathbf{a}_i} \frac{2\partial |f(T_{ijk})|}{\partial \mathbf{b}_i} + \mathbf{a}_{jk}\mathbf{b}_{jk}
 \right] . \label{eq:bajei}
\end{align}

We now have derived the entries of matrix $L_1 - L_2$, which is the Hessian matrix of the conformal energy with respect to Cartesian coordinates $[\mathbf{x},\mathbf{y},\mathbf{z}]$. The following theorem shows that the nullity of $L_1 - L_2$ is $3$.

\begin{theorem}
$L_1 - L_2$ is a symmetric Laplacian matrix with its null space having orthogonal basis $[\mathbf{1}^\top,\mathbf{0}^\top,\mathbf{0}^\top]^\top$, $[\mathbf{0}^\top,\mathbf{1}^\top,\mathbf{0}^\top]^\top$ and $[\mathbf{0}^\top,\mathbf{0}^\top,\mathbf{1}^\top]^\top$.
\end{theorem}

\begin{proof}
Since $L_1 - L_2$ is the Hessian matrix of the conformal energy with respect to Cartesian coordinates $[\mathbf{x},\mathbf{y},\mathbf{z}]$, it is obviously symmetric. By the gradient formulas \eqref{eq:pcijpj}-\eqref{eq:pckipj}, it is easily seen that
\begin{align*}
        \left(\frac{\partial }{\partial \mathbf{f}_i} + \frac{\partial }{\partial \mathbf{f}_j} + \frac{\partial }{\partial \mathbf{f}_k}\right)c_{ij}(f) = 0.
    \end{align*}
Therefore, by using the representation according to the adjacent vertices, we have
\begin{align*}
    \sum_{j = 1}^n \left[\sum_{\ell = 1}^n \mathbf{a}_\ell\nabla_\mathbf{b} (L(f)\mathbf{e}_\ell)\right]_{ij}
    &= \frac{1}{2}\sum_{j\in \mathcal{N}(i)} \mathbf{a}_{ij} \left( \frac{\partial}{\partial \mathbf{b}_i} + \frac{\partial}{\partial \mathbf{b}_j} + \frac{\partial}{\partial \mathbf{b}_k} \right) c_{ij}(f) \\
    &- \frac{1}{2}\sum_{j\in \mathcal{N}(i)} \mathbf{a}_{ji} \left( \frac{\partial}{\partial \mathbf{b}_i} + \frac{\partial}{\partial \mathbf{b}_j} + \frac{\partial}{\partial \mathbf{b}_{k'}} \right) c_{ji}(f)
    = 0, ~ i = 1,2,\cdots,n.
    \end{align*}
Additionally, by the equality $\mathbf{1}^\top(\nabla_\mathbf{b} L(f)\mathbf{e}_{\ell}) = \nabla_\mathbf{b} (\mathbf{1}^\top L(f)\mathbf{e}_{\ell}) = 0$, we also have
\[
    \sum_{i = 1}^n \left[\sum_{\ell = 1}^n \mathbf{a}_\ell\nabla_\mathbf{b} (L(f)\mathbf{e}_\ell)\right]_{ij}
    =
    \left[\sum_{\ell = 1}^n \mathbf{a}_\ell\nabla_\mathbf{b} (\mathbf{1}^\top L(f)\mathbf{e}_\ell)\right] \mathbf{e}_j
    = 0, ~ j = 1,2,\cdots,n.
    \]
We conclude that the sum of each row and column of $\sum_{\ell = 1}^n \mathbf{a}_\ell\nabla_\mathbf{b} (L(f)\mathbf{e}_\ell)$ are zero, which guarantees the $3$ orthogonal basis of the null space.
\end{proof}



Finally, we focus on
\begin{align*}
    H_2 := \left[
        \nabla \mathbf{x}^\top , \nabla \mathbf{y}^\top , \nabla \mathbf{z}^\top
    \right]
    L_2
    \begin{bmatrix}
        \nabla \mathbf{x} \\ \nabla \mathbf{y} \\ \nabla \mathbf{z}
    \end{bmatrix}
    = \begin{bmatrix}
        H_{2,\boldsymbol{\theta}\boldsymbol{\theta}} &
        H_{2,\boldsymbol{\theta}\boldsymbol{\phi}} \\
        H_{2,\boldsymbol{\phi}\boldsymbol{\theta}} &
        H_{2,\boldsymbol{\phi}\boldsymbol{\phi}}
    \end{bmatrix}
\end{align*}
defined in \eqref{eq:HessianE} and \eqref{eq:HessianE=2}.
The formulas \eqref{eq:bajni} and \eqref{eq:bajei} show that the entries of $L_2$ can be rewritten as the inner products of two terms related to $\mathbf{a}$ and $\mathbf{b}$, respectively. Taking \eqref{eq:bajei} as an example, we have
\begin{align*}
    \mathbf{a}_{ij} \frac{\partial}{\partial \mathbf{b}_i} c_{ij}(f) +
\mathbf{a}_{ik} \frac{\partial}{\partial \mathbf{b}_i} c_{ki}(f)
= -\frac{1}{2|f(T_{ijk})|} \left\langle \left[
    \frac{2\partial |f(T_{ijk})|}{\partial \mathbf{a}_i}, \mathbf{a}_{jk}
    \right], \left[
    \frac{2\partial |f(T_{ijk})|}{\partial \mathbf{b}_i}, \mathbf{b}_{jk}
    \right] \right\rangle,
\end{align*}
From gradient formulas \eqref{eq:pApi}-\eqref{eq:pApk}, we can find that the first and second terms are only associated with $\mathbf{a}$ and $\mathbf{b}$, respectively. Additionally, \eqref{eq:bajni} can also be written as a similar inner product form. Then, since $\mathbf{f}_i$ depends only on $(\boldsymbol{\theta}_i,{\boldsymbol{\phi}}_i)$ for $i = 1,2,\cdots,n$, the Jacobian matrices of $\mathbf{x},\mathbf{y},\mathbf{z}$ with respect to $\boldsymbol{\theta},{\boldsymbol{\phi}}$ are diagonal. It is easy to express the products of Jacobian matrices and $L_2$. Without loss of generality, we discuss only entry $\big[H_{2,\boldsymbol{\theta}\boldsymbol{\phi}}\big]_{ii}$,

\begin{align*}
    \big[H_{2,\boldsymbol{\theta}\boldsymbol{\phi}}\big]_{ii} =& \sum_{\{j,k\}\in S_{\mathcal{E}}(i)} \sum_{\mathbf{a},\mathbf{b} = \mathbf{x},\mathbf{y},\mathbf{z}} \left( \mathbf{a}_{ij} \frac{\partial}{\partial \mathbf{b}_i} c_{ij}(f) -
    \mathbf{a}_{ki} \frac{\partial}{\partial \mathbf{b}_i} c_{ki}(f) \right) \frac{\partial \mathbf{a}_i}{\partial \boldsymbol{\theta}_i} \frac{\partial \mathbf{b}_i}{\partial {\boldsymbol{\phi}}_i}\\
    =& -\sum_{\{j,k\}\in S_{\mathcal{E}}(i)} \frac{1}{2|f(T_{ijk})|} \sum_{\mathbf{a},\mathbf{b} = \mathbf{x},\mathbf{y},\mathbf{z}}  \left\langle \frac{\partial \mathbf{a}_i}{\partial \boldsymbol{\theta}_i} \left[
    \frac{2\partial |f(T_{ijk})|}{\partial \mathbf{a}_i}, \mathbf{a}_{jk}
    \right], \frac{\partial \mathbf{b}_i}{\partial {\boldsymbol{\phi}}_i} \left[
    \frac{2\partial |f(T_{ijk})|}{\partial \mathbf{b}_i}, \mathbf{b}_{jk}
    \right] \right\rangle\\
    =& -\sum_{\{j,k\}\in S_{\mathcal{E}}(i)} \frac{1}{2|f(T_{ijk})|} \left\langle \sum_{\mathbf{a} = \mathbf{x},\mathbf{y},\mathbf{z}}\frac{\partial \mathbf{a}_i}{\partial \boldsymbol{\theta}_i} \left[
    \frac{2\partial |f(T_{ijk})|}{\partial \mathbf{a}_i}, \mathbf{a}_{jk}
    \right], \sum_{\mathbf{b} = \mathbf{x},\mathbf{y},\mathbf{z}} \frac{\partial \mathbf{b}_i}{\partial {\boldsymbol{\phi}}_i} \left[
    \frac{2\partial |f(T_{ijk})|}{\partial \mathbf{b}_i}, \mathbf{b}_{jk}
    \right] \right\rangle\\
    =& -\sum_{\{j,k\}\in S_{\mathcal{E}}(i)} \frac{1}{2|f(T_{ijk})|} \left\langle \frac{\partial\mathbf{f}_i}{\partial\boldsymbol{\theta}_i}^\top \left[
    \frac{2\partial |f(T_{ijk})|}{\partial \mathbf{f}_i}, \mathbf{f}_{jk}
    \right], \frac{\partial\mathbf{f}_i}{\partial{\boldsymbol{\phi}}_i}^\top \left[
    \frac{2\partial |f(T_{ijk})|}{\partial \mathbf{f}_i}, \mathbf{f}_{jk}
    \right] \right\rangle.
\end{align*}

The partial differential terms in \eqref{eq:pApi}-\eqref{eq:pApk} show that they are the linear combination of $\mathbf{f}$. Hence, only the sum with respect to $\mathbf{x},\mathbf{y},\mathbf{z}$ is necessary, such as
\begin{subequations} \label{eq:abi}
\begin{equation} \label{eq:abi_jk}
    \begin{bmatrix}
        \mathbf{a}^i_{jk} \\ \mathbf{b}^i_{jk}
    \end{bmatrix}
    :=
    \left(\frac{\partial\mathbf{f}_i}{\partial(\boldsymbol{\theta}_i,{\boldsymbol{\phi}}_i)}\right)^\top \mathbf{f}_{jk}
    = \left[ \begin{aligned}
        -&\mathbf{y}_i\mathbf{x}_{jk} + \mathbf{x}_i\mathbf{y}_{jk} \\
         &\mathbf{u}_i\mathbf{x}_{jk} + \mathbf{v}_i\mathbf{y}_{jk} - \mathbf{w}_i\mathbf{z}_{jk}
    \end{aligned} \right],
    \end{equation}
    \begin{equation}\label{eq:abi_ki}
    \begin{bmatrix}
        \mathbf{a}^i_{ki} \\ \mathbf{b}^i_{ki}
    \end{bmatrix}
    :=
    \left(\frac{\partial\mathbf{f}_i}{\partial(\boldsymbol{\theta}_i,{\boldsymbol{\phi}}_i)}\right)^\top \mathbf{f}_{ki}
    = \left[ \begin{aligned}
        -&\mathbf{y}_i\mathbf{x}_{ki} + \mathbf{x}_i\mathbf{y}_{ki} \\
         &\mathbf{u}_i\mathbf{x}_{ki} + \mathbf{v}_i\mathbf{y}_{ki} - \mathbf{w}_i\mathbf{z}_{ki}
    \end{aligned} \right],
    \end{equation}
    \begin{equation} \label{eq:abi_ij}
    \begin{bmatrix}
        \mathbf{a}^i_{ij} \\ \mathbf{b}^i_{ij}
    \end{bmatrix}
    :=
    \left(\frac{\partial\mathbf{f}_i}{\partial(\boldsymbol{\theta}_i,{\boldsymbol{\phi}}_i)}\right)^\top \mathbf{f}_{ij}
    = \left[ \begin{aligned}
        -&\mathbf{y}_i\mathbf{x}_{ij} + \mathbf{x}_i\mathbf{y}_{ij} \\
         &\mathbf{u}_i\mathbf{x}_{ij} + \mathbf{v}_i\mathbf{y}_{ij} - \mathbf{w}_i\mathbf{z}_{ij}
    \end{aligned} \right].
\end{equation}
\end{subequations}
The others are defined similarly by modifying the superscripts and subscripts. Here, $\mathbf{a}^i_{\cdot\cdot}$ denotes the $\partial\boldsymbol{\theta}_i$ term, and $\mathbf{b}^i_{\cdot\cdot}$ denotes the $\partial{\boldsymbol{\phi}}_i$ term.
Thus, we have
\begin{align*}
    \frac{\partial\mathbf{f}_i}{\partial\boldsymbol{\theta}_i}^\top \left[
    \frac{2\partial |f(T_{ijk})|}{\partial \mathbf{f}_i}, \mathbf{f}_{jk}
    \right]
    = \frac{\partial\mathbf{f}_i}{\partial\boldsymbol{\theta}_i}^\top
    &\big[
    c_{ij}(f)\mathbf{f}_{ij} - c_{ki}(f)\mathbf{f}_{ki}, \mathbf{f}_{jk}
    \big]\\
    = &\big[
    c_{ij}(f)\mathbf{a}_{ij}^i - c_{ki}(f)\mathbf{a}_{ki}^i, \mathbf{a}_{jk}^i
    \big]
    .
\end{align*}
We can see that they differ only in $\mathbf{f}_{\cdot\cdot}$ and $\mathbf{a}_{\cdot\cdot}^i,\mathbf{b}_{\cdot\cdot}^i$. Hence, the entries of $H_2$ differ from only those of $L_2$ with the notations
$\mathbf{x}_{\cdot\cdot},\mathbf{y}_{\cdot\cdot},\mathbf{z}_{\cdot\cdot}$
in \eqref{eq:bajni} and \eqref{eq:bajei} replaced by $\mathbf{a}_{\cdot\cdot}^i$ and $\mathbf{b}_{\cdot\cdot}^i$.
Taking $H_{2,\boldsymbol{\theta}\boldsymbol{\phi}}$ as an example, the diagonal and nondiagonal entries are
\begin{align*}
\big[H_{2,\boldsymbol{\theta}\boldsymbol{\phi}}\big]_{ii} =
& -\sum_{\{j,k\}\in S_{\mathcal{E}}(i)} \frac{1}{4|f(T_{ijk})|} \left[
\frac{2\partial |f(T_{ijk})|}{\partial \boldsymbol{\theta}_i} \frac{2\partial |f(T_{ijk})|}{\partial {\boldsymbol{\phi}}_i} + \mathbf{a}^i_{jk}\mathbf{b}^i_{jk}
 \right], \\
\big[H_{2,\boldsymbol{\theta}\boldsymbol{\phi}}\big]_{ij} =
&-\frac{1}{4|f(T_{ijk})|} \left[
\frac{2\partial |f(T_{ijk})|}{\partial \boldsymbol{\theta}_i} \frac{2\partial |f(T_{ijk})|}{\partial {\boldsymbol{\phi}}_j} + (2\mathbf{a}_{ki}^i\mathbf{b}_{jk}^j - \mathbf{a}_{jk}^i\mathbf{b}_{ki}^j)
 \right] \nonumber \\
 &-\frac{1}{4|f(T_{k'ji})|} \left[
\frac{2\partial |f(T_{k'ji})|}{\partial \boldsymbol{\theta}_i} \frac{2\partial |f(T_{k'ji})|}{\partial {\boldsymbol{\phi}}_j} + (2\mathbf{a}_{ik'}^i\mathbf{b}_{k'j}^j - \mathbf{a}_{k'j}^i\mathbf{b}_{ik'}^j)
 \right], 
\end{align*}
respectively, where
\begin{subequations}
\begin{equation}
    \frac{\partial |f(T_{ijk})|}{\partial \boldsymbol{\theta}_i} =
    \frac{1}{2}\big(c_{ij}(f) \mathbf{a}^i_{ij} - c_{ki}(f) \mathbf{a}^i_{ki}\big), \label{eq:pApthetai}
\end{equation}
\begin{equation}
    \frac{\partial |f(T_{ijk})|}{\partial {\boldsymbol{\phi}}_i} =
    \frac{1}{2}\big(c_{ij}(f) \mathbf{b}^i_{ij} - c_{ki}(f) \mathbf{b}^i_{ki}\big). \label{eq:pApphii}
\end{equation}
\end{subequations}
The other entries are obtained similarly by replacing $\partial \mathbf{x}_i$, $\partial \mathbf{y}_i$, $\partial \mathbf{z}_i$ and $\mathbf{x}_{\cdot\cdot}$, $\mathbf{y}_{\cdot\cdot}$, $\mathbf{z}_{\cdot\cdot}$ by $\partial \boldsymbol{\theta}_i,\partial {\boldsymbol{\phi}}_i$ and $\mathbf{a}_{\cdot\cdot}^i,\mathbf{b}_{\cdot\cdot}^i$ as in \eqref{eq:abi_jk}-\eqref{eq:abi_ij}, respectively.
Now, we have obtained the entries of Hessian matrix $H$ in \eqref{eq:HessianE}.
\begin{theorem} \label{thm:Hquantity}
Let $H = \begin{bmatrix}
H_{\boldsymbol{\theta}\boldsymbol{\theta}} & H_{\boldsymbol{\theta}{\boldsymbol{\phi}}} \\ H_{{\boldsymbol{\phi}}\boldsymbol{\theta}} & H_{{\boldsymbol{\phi}}{\boldsymbol{\phi}}}
\end{bmatrix}$ with $H_{\boldsymbol{\theta}{\boldsymbol{\phi}}} = H_{{\boldsymbol{\phi}}\boldsymbol{\theta}}^\top$. The diagonal and nondiagonal entries of the blocks are
\begin{equation*} 
\left\{
\begin{array}{r@{}l}
\big[ H_{\boldsymbol{\theta}\boldsymbol{\theta}} \big]_{ii} &=
\sum_{\{j,k\}\in S_{\mathcal{E}}(i)} \frac{1}{4|f(T_{ijk})|} \left[
\left(\frac{2\partial |f(T_{ijk})|}{\partial \boldsymbol{\theta}_i}\right)^2 + \big(\mathbf{a}^i_{jk}\big)^2
\right] \\[4pt]
&+ \sum_{j \in \mathcal{N}(i)}\tilde{w}_{ij}(\mathbf{x}_i\mathbf{x}_{j} + \mathbf{y}_i\mathbf{y}_{j}),\\[2pt]
\big[ H_{\boldsymbol{\theta}\boldsymbol{\theta}} \big]_{ij} &=
\frac{1}{4|f(T_{ijk})|} \left[
\frac{2\partial |f(T_{ijk})|}{\partial \boldsymbol{\theta}_i} \frac{2\partial |f(T_{ijk})|}{\partial {\boldsymbol{\theta}}_j} + (2\mathbf{a}_{ki}^i\mathbf{a}_{jk}^j - \mathbf{a}_{jk}^i\mathbf{a}_{ki}^j)
\right] \\[4pt]
&+\frac{1}{4|f(T_{k'ji})|} \left[
\frac{2\partial |f(T_{k'ji})|}{\partial \boldsymbol{\theta}_i} \frac{2\partial |f(T_{k'ji})|}{\partial {\boldsymbol{\theta}}_j} + (2\mathbf{a}_{ik'}^i\mathbf{a}_{k'j}^j - \mathbf{a}_{k'j}^i\mathbf{a}_{ik'}^j)
\right] \\[4pt]
&- \tilde{w}_{ij}(\mathbf{x}_i\mathbf{x}_j + \mathbf{y}_i\mathbf{y}_j);\\
\end{array}
\right.
\end{equation*}
\begin{equation*}
\left\{
\begin{array}{r@{}l}
\big[ H_{\boldsymbol{\phi}\boldsymbol{\phi}} \big]_{ii} &=
\sum_{\{j,k\}\in S_{\mathcal{E}}(i)} \frac{1}{4|f(T_{ijk})|} \left[
\left(\frac{2\partial |f(T_{ijk})|}{\partial \boldsymbol{\phi}_i}\right)^2 + \big(\mathbf{b}^i_{jk}\big)^2
\right] \\[4pt]
&+ \sum_{j \in \mathcal{N}(i)}\tilde{w}_{ij}(\mathbf{x}_i\mathbf{x}_{j} + \mathbf{y}_i\mathbf{y}_{j} + \mathbf{z}_i\mathbf{z}_{j}),\\[4pt]
\big[ H_{\boldsymbol{\phi}\boldsymbol{\phi}} \big]_{ij} &=
\frac{1}{4|f(T_{ijk})|} \left[
\frac{2\partial |f(T_{ijk})|}{\partial \boldsymbol{\phi}_i} \frac{2\partial |f(T_{ijk})|}{\partial {\boldsymbol{\phi}}_j} + (2\mathbf{b}_{ki}^i\mathbf{b}_{jk}^j - \mathbf{b}_{jk}^i\mathbf{b}_{ki}^j)
\right] \\[4pt]
&+\frac{1}{4|f(T_{k'ji})|} \left[
\frac{2\partial |f(T_{k'ji})|}{\partial \boldsymbol{\phi}_i} \frac{2\partial |f(T_{k'ji})|}{\partial {\boldsymbol{\phi}}_j} + (2\mathbf{b}_{ik'}^i\mathbf{b}_{k'j}^j - \mathbf{b}_{k'j}^i\mathbf{b}_{ik'}^j)
\right] \\[4pt]
&- \tilde{w}_{ij}(\mathbf{u}_i\mathbf{u}_{j} + \mathbf{v}_i\mathbf{v}_{j} + \mathbf{w}_i\mathbf{w}_{j});\\
\end{array}
\right.
\end{equation*}
\begin{equation} \label{eq:Hpt}
\left\{
\begin{array}{r@{}l}
\big[H_{\boldsymbol{\phi}\boldsymbol{\theta}}\big]_{ii} &=
\sum_{\{j,k\}\in S_{\mathcal{E}}(i)} \frac{1}{4|f(T_{ijk})|} \left[
\frac{2\partial |f(T_{ijk})|}{\partial \boldsymbol{\phi}_i} \frac{2\partial |f(T_{ijk})|}{\partial {\boldsymbol{\theta}}_i} + \mathbf{b}^i_{jk}\mathbf{a}^i_{jk}
\right] \\[4pt]
&+\sum_{j \in \mathcal{N}(i)} \tilde{w}_{ij}(\mathbf{v}_i\mathbf{x}_j - \mathbf{u}_i\mathbf{y}_j), \\[2pt]
\big[H_{\boldsymbol{\phi}\boldsymbol{\theta}}\big]_{ij} &=
\frac{1}{4|f(T_{ijk})|} \left[
\frac{2\partial |f(T_{ijk})|}{\partial \boldsymbol{\phi}_i} \frac{2\partial |f(T_{ijk})|}{\partial {\boldsymbol{\theta}}_j} + (2\mathbf{b}_{ki}^i\mathbf{a}_{jk}^j - \mathbf{b}_{jk}^i\mathbf{a}_{ki}^j)
\right] \\[4pt]
&+\frac{1}{4|f(T_{k'ji})|} \left[
\frac{2\partial |f(T_{k'ji})|}{\partial \boldsymbol{\phi}_i} \frac{2\partial |f(T_{k'ji})|}{\partial {\boldsymbol{\theta}}_j} + (2\mathbf{b}_{ik'}^i\mathbf{a}_{k'j}^j - \mathbf{b}_{k'j}^i\mathbf{a}_{ik'}^j)
\right] \\[4pt]
&-\tilde{w}_{ij}(\mathbf{v}_i\mathbf{x}_j - \mathbf{u}_i\mathbf{y}_j);
 \end{array}
\right.
\end{equation}
respectively.
\end{theorem}

Finally, we present a theorem for the proposed algorithm in \Cref{sec:algorithm}, which demonstrates singularity and the corresponding eigenpair of $H$.

\begin{theorem} \label{thm:Heigen}
The null space of $H$ defined in \eqref{eq:HessianE} has a basis $[
\mathbf{1}^\top, \mathbf{0}^\top
]^\top$, where $\text{dim}(\mathbf{1})=\text{dim}(\mathbf{0})=n$.
\end{theorem}


\begin{proof}
To prove the assertion, we need to indicate only that

\begin{align*}
    \mathbf{e}_i^\top H_{\boldsymbol{\theta}\boldsymbol{\theta}} \mathbf{1} = 0,
    \quad \text{ and } \quad
    \mathbf{e}_i^\top H_{{\boldsymbol{\phi}}\boldsymbol{\theta}} \mathbf{1} = 0, \quad \text{for } i = 1,2,\cdots,n.
\end{align*}
We first consider $H_{{\boldsymbol{\phi}}\boldsymbol{\theta}}$.
By \eqref{eq:Hpt} we have
\begin{align}
\mathbf{e}_i^\top H_{{\boldsymbol{\phi}}\boldsymbol{\theta}} \mathbf{1} = \sum_{\{j,k\}\in S_{\mathcal{E}}(i)} \frac{1}{4|f(T_{ijk})|} & \Bigg[
\frac{2\partial |f(T_{ijk})|}{\partial {\boldsymbol{\phi}}_i} \cdot 2\left(
\frac{\partial }{\partial \boldsymbol{\theta}_i} + \frac{\partial }{\partial \boldsymbol{\theta}_j} + \frac{\partial }{\partial \boldsymbol{\theta}_k}\right)|f(T_{ijk})| \nonumber\\
&+\mathbf{b}^i_{jk}\mathbf{a}^i_{jk} + 2\mathbf{b}_{ki}^i\mathbf{a}_{jk}^j - \mathbf{b}_{jk}^i\mathbf{a}_{ki}^j + 2\mathbf{b}_{ij}^i\mathbf{a}_{jk}^k - \mathbf{b}_{jk}^i\mathbf{a}_{ij}^k
 \Bigg] \label{eq:H2phithetasumi}
\end{align}
It is easy to verify that
\begin{align}
    \mathbf{a}_{ij}^i = \mathbf{a}_{ij}^j,\quad
    \mathbf{a}_{jk}^j = \mathbf{a}_{jk}^k,\quad
    \mathbf{a}_{ki}^k = \mathbf{a}_{ki}^i. \label{eq:supexchange}
\end{align}
Combining \eqref{eq:supexchange} with \eqref{eq:pApthetai}, we have
\begin{align*}
\left(
\frac{\partial }{\partial \boldsymbol{\theta}_i} + \frac{\partial }{\partial \boldsymbol{\theta}_j} + \frac{\partial }{\partial \boldsymbol{\theta}_k}\right)|f(T_{ijk})| = 0.
\end{align*}
Hence, by using \eqref{eq:supexchange} again, \eqref{eq:H2phithetasumi} becomes
\begin{align*}
\mathbf{e}_i^\top H_{{\boldsymbol{\phi}}\boldsymbol{\theta}} \mathbf{1} =&
\sum_{\{j,k\}\in S_{\mathcal{E}}(i)} \frac{1}{4|f(T_{ijk})|} \left[
\mathbf{b}^i_{jk}\mathbf{a}^i_{jk} + 2\mathbf{b}_{ki}^i\mathbf{a}_{jk}^j - \mathbf{b}_{jk}^i\mathbf{a}_{ki}^j + 2\mathbf{b}_{ij}^i\mathbf{a}_{jk}^k - \mathbf{b}_{jk}^i\mathbf{a}_{ij}^k
\right] \\
= & \sum_{\{j,k\}\in S_{\mathcal{E}}(i)} \frac{1}{4|f(T_{ijk})|}
\mathbf{b}^i_{jk}\left( \mathbf{a}_{ij}^i + \mathbf{a}^i_{jk} + \mathbf{a}_{ki}^i \right) = 0.
\end{align*}
Similarly, we also have $\mathbf{e}_i^\top H_{\boldsymbol{\theta}\boldsymbol{\theta}} \mathbf{1} = 0$. Notably, $\mathbf{e}_i^\top H_{{\boldsymbol{\phi}}{\boldsymbol{\phi}}} \mathbf{1} \neq 0$. Therefore, the theorem is proved.
\end{proof}

\begin{remark} \label{remark:rotation}
Geometrically, the one-dimensional null space of $H$ reveals that the conformal energy is invariant up to a rotation along the latitude, while that along the longitude is not characterized, since it is not linearly related to ${\boldsymbol{\phi}}$.
\end{remark}

\section{Hessian-Based Trust Region Algorithm} \label{sec:algorithm}

In this section, we develop an HBTR algorithm
to minimize the conformal energy for the computation of the conformal map from a closed surface of genus-$0$ to a unit sphere. Here, we review the optimization problem,
\begin{align} \label{prob:CEM}
    \min E_C(\boldsymbol{\theta},{\boldsymbol{\phi}}):= \frac{1}{2} \langle D(f)\mathbf{f}, \mathbf{f} \rangle, \qquad
    \mathbf{f} = [\cos\boldsymbol{\theta} \odot \sin{\boldsymbol{\phi}},
    \sin\boldsymbol{\theta} \odot \sin{\boldsymbol{\phi}},
    \cos{\boldsymbol{\phi}}].
\end{align}
We loosen the box constraint $(\boldsymbol{\theta},{\boldsymbol{\phi}}) \in [0,2\pi]^n \times [0,\pi]^n$ and consider \eqref{prob:CEM} as an unconstrained problem. In \Cref{sec:Hessian}, the gradient vector and Hessian matrix of $E_C({\boldsymbol{\theta}},{\boldsymbol{\phi}})$ are derived in \eqref{eq:gradientE} and \Cref{thm:Hquantity}, respectively. The sparsity of the Hessian matrix shown in \Cref{thm:sparsity} guarantees the feasibility of fast computation associated with $H$. Specifically, we solve the large-scale sparse linear system to obtain the Newton direction: $H\mathbf{s} = -\mathbf{g}$ as in \eqref{eq:newtonstep}.
As \Cref{thm:Heigen} demonstrated, $H$ is singular, and the general solution can be expressed as a spherical solution with arbitrary rotation along latitude according to \Cref{remark:rotation}. Therefore, we fix the first entry of $\mathbf{s}$ and let
\[
H = \begin{bmatrix}
    h_{11} & \mathbf{h}_1^\top \\ \mathbf{h}_1 & \hat{H}
\end{bmatrix}, \quad
\mathbf{s} = \begin{bmatrix}
    0 \\ \hat{\mathbf{s}}
\end{bmatrix}, \quad
\mathbf{g} = \begin{bmatrix}
    g_{1} \\ \hat{\mathbf{g}}
\end{bmatrix}.
\]
Then, we solve the linear system
\begin{align} \label{eq:newtonstephat}
    \hat{H}\hat{\mathbf{s}} = -\hat{\mathbf{g}}.
\end{align}
It is easy to verify that $H\mathbf{s} = -\mathbf{g}$ holds for $\mathbf{s} = [0,\hat{\mathbf{s}}^\top ]^\top$. Geometrically, this approach means that the longitude of the first vertex always remains invariant and that meaningless rotation of the sphere is avoided during the iteration.

\begin{remark}
$H$ can be rearranged from $L \otimes \mathbf{1}_{2\times 2}$ to $\mathbf{1}_{2\times 2} \otimes L$. The rearranged matrix is constructed by the $2\times 2$ block in $L$ form such that the block LU decomposition is appropriate for the fast computation of the linear system \eqref{eq:newtonstephat}.
\end{remark}

Therefore, it is natural to use the Newton-type method to solve the optimization problem \eqref{prob:CEM}. However, the pure Newton method is insufficient for solving the optimization problem directly because the Newton method is well known to have local quadratic convergence. Unfortunately, it is almost impractical to directly seek an initial guess sufficiently close to the ideal solution for \eqref{prob:CEM}, especially when facing surfaces with high-curvature regions or complicated shapes. Additionally, the conformal energy of \eqref{prob:CEM} with respect to $(\boldsymbol{\theta},{\boldsymbol{\phi}})$ is nonconvex and nonlinear. Hence, the quadratic convergence generally disappears at the beginning of iterations, and it might take much time to reach the neighborhood with quadratic convergence via the Newton step. Moreover, importantly, the Hessian matrix $H$ is not uniformly positive semidefinite; that is, $H$ may be indefinite at some points $(\boldsymbol{\theta},{\boldsymbol{\phi}})$. Indeed, negative curvature occurs frequently during iteration in practical experiments. Consequently, it is possible that the Newton direction may not be a descent direction.

To overcome this drawback, we introduce the negative gradient direction, which is a descent direction. More specifically, we search the trial step $\mathbf{d}$ from the $2$D subspace spanned by the Newton direction $\mathbf{s}$ and the gradient direction $\mathbf{g}$. If $\mathbf{s}$ is not a descent direction, it is a negative curvature direction, which is still beneficial information for the choice of descent direction.
Moreover, there must be a descent direction in this $2$D subspace since $-\mathbf{g} \in \text{span}(\mathbf{s},\mathbf{g})$. When the iterative point is in the convergence neighborhood, the Newton direction guarantees the quadratic convergence of the algorithm. We utilize the trust region method to search the trial step from $\text{span}(\mathbf{s},\mathbf{g})$. In other words, during each iterative step, we consider the trust region subproblem proposed by Shultz et al.\cite{GSRB85}
\begin{align} \label{opt:subtr2D}
\begin{array}{cl}
    \min & \mathbf{d}^\top H \mathbf{d} + \mathbf{g}^\top \mathbf{d}\\
     \text{s.t.} & \|\mathbf{d}\|\leq\mit\Delta, \ \mathbf{d}\in \text{span}(\mathbf{s},\mathbf{g}),
\end{array}
\end{align}
where $\mit\Delta$ is the trust region radius. The $2$D optimization problem \eqref{opt:subtr2D} is easy to solve and costs less time. The global and local convergence of the trust region method with subproblem \eqref{opt:subtr2D} has been proven by \cite{GSRB85}, and the practical experiment has been verified later in \cite{RBRS88}.

For the error measurement, it is worth noting that $E_C(\boldsymbol{\theta},{\boldsymbol{\phi}})$ is a periodic function. Since we loosen the box constraint $(\boldsymbol{\theta},{\boldsymbol{\phi}}) \in [0,2\pi]^n\times [0,\pi]^n$, $(\boldsymbol{\theta},{\boldsymbol{\phi}})$ might move larger than $2\pi$, while the vertices on the sphere move much less. Hence, it is inappropriate to check the length of step $\mathbf{d}$. Inspired by \cite{MHTM21}, we adopt the optimal rotation and measure the error with
\begin{align} \label{eq:error}
    \delta^{(k)} = \min_{R\in SO(3)} \left\|\mathbf{f}^{(k+1)} - \mathbf{f}^{(k)}R\right\|^2,
\end{align}
where $SO(3) = \left\{R\in \mathbb{R}^{3\times 3} ~|~ R^\top R = I,~\det(R) = 1\right\}$ and $\mathbf{f}^{(k)}$ are the vertices on the sphere at the $k$-th iteration. The error measurement eliminates the error from rotation. In addition, it avoids the miscalculation caused by the periodicity of $E_C(\boldsymbol{\theta},{\boldsymbol{\phi}})$.

We summarize the proposed Hessian-based trust region (HBTR) algorithm for the optimization problem \eqref{prob:CEM}.

\begin{algorithm}
\caption{HBTR for Spherical CEM Problem}
\label{alg:TR}
\begin{algorithmic}[1]
\REQUIRE Triangulation $M$ with vertices $\{v_i,i = 1,2,\cdots, n\}$, tolerance $\varepsilon$.
    \ENSURE $\mathbf{f} \in \mathbb{R}^{n\times 3}$ inducing the conformal map $f$ as in \eqref{eq:bary}.
    \STATE Set $k = 0$ and $\delta^{(0)} = +\infty$.
    \STATE Compute the initial guess $\mathbf{f}$ and the corresponding conformal energy $E_C^{(0)}$ in \eqref{def:CE}.
    \WHILE{$\delta>\varepsilon$}
    \STATE Compute the gradient vector $\mathbf{g}$ and the Hessian matrix $H$ by \eqref{eq:gradientE} and \Cref{thm:Hquantity}.
    \STATE Solve the linear system $\hat{H}\hat{\mathbf{s}} = -\hat{\mathbf{g}}$ via block LU decomposition to get the Newton direction $\mathbf{s}$.
    \STATE Solve the trust region subproblem \eqref{opt:subtr2D} to get the trial step $\mathbf{d}$.
    \STATE Let $E \gets E_C\big((\boldsymbol{\theta},{\boldsymbol{\phi}}) + \mathbf{d}\big)$. If $E_C^{(k)} > E$, set $k \gets k+1$ and update
\begin{align*}
        &(\boldsymbol{\theta},{\boldsymbol{\phi}}) \gets (\boldsymbol{\theta},{\boldsymbol{\phi}}) + \mathbf{d},\\
        &E_C^{(k)} \gets E.
    \end{align*}
    \STATE Compute the error $\delta^{(k)}$ by \eqref{eq:error} and tune the trust region radius $\mit\Delta$.
    \ENDWHILE
\end{algorithmic}
\end{algorithm}

\begin{remark}
Similar to the disk parameterization in \cite{MHTM21}, the bijectivity of the resulting map can almost be guaranteed under the spherical coordinate representation. If not, the folding triangles can be removed by the mean value coordinate \cite{Floa03}.
\end{remark}





\section{Numerical Experiments} \label{sec:experiments}

In this section, we describe the numerical performance of our proposed HBTR method for spherical conformal parameterization on several triangulation models.
All experimental programs are executed in MATLAB R2021a on a personal computer with a 2.50 GHz CPU and 64 GB RAM. Most of the triangulation models are taken from AIM@SHAPE shape repository \cite{AIM}, ALICE \cite{ALICE}, Gu's personal website \cite{GuWebsite}, the Stanford 3D scanning repository \cite{Stanford}, Human Connectome Project \cite{HCP}, and TurboSquid \cite{Turbo}. The triangulations of brain cortical surfaces are generated from BraTS datasets \cite{BRATS20232} via library JIGSAW \cite{Darr18,Darr16,Darr15,DEDI16,Darr14phd} and toolbox Iso2Mesh \cite{QFDA09,APSY20,Iso2Mesh}.
From the abovementioned benchmarks, we take the triangulation models for experiments as shown in \Cref{fig:meshes} and present their basic information for numbers of vertices and faces in \Cref{tab:meshes}.
Notably, no folding occurs on the $8$ models by HBTR.
We apply the SCEM algorithm proposed in \cite{MHTL19} for the initial guess in the experiments. Among the vast experiments, the SCEM can stably provide a great initial guess in a very short time, which is appropriate for the HBTR algorithm.
\begin{table}[htp]
    \centering
    \begin{tabular}{ccc||ccc}
    \hline
    Mesh & $\#V$ & $\#F$ & Mesh & $\#V$ & $\#F$\\
    \hline
    Apple & 17839 & 35674 & Fandisk   & 6475  & 12946 \\
    Arnold & 14530 & 29056 & Horse  & 21013  & 42022 \\
    Brain  & 32160  & 64316 & Planck  & 51108 & 102212  \\
    Bunny  & 55684  & 111364 & Venus  & 14303  & 28602 \\
    \hline
    \end{tabular}
    \caption{The features of triangulation models with $\#V$ and $\#F$ being the number of vertices and triangle faces, respectively.}
    \label{tab:meshes}
\end{table}

\begin{figure}[thp]
    \centering
\begin{tabular}{cccc}
        \includegraphics[clip,trim = {3.5cm 1.25cm 3cm 1cm},height = 3cm]{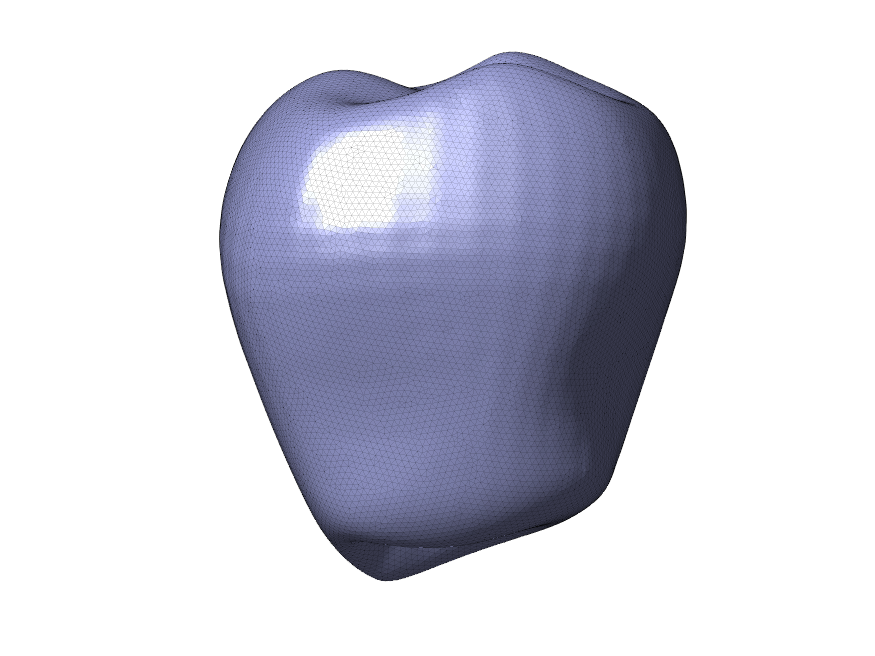} &
        \includegraphics[clip,trim = {4.5cm 1.25cm 4cm 0.5cm},height = 3.5cm]{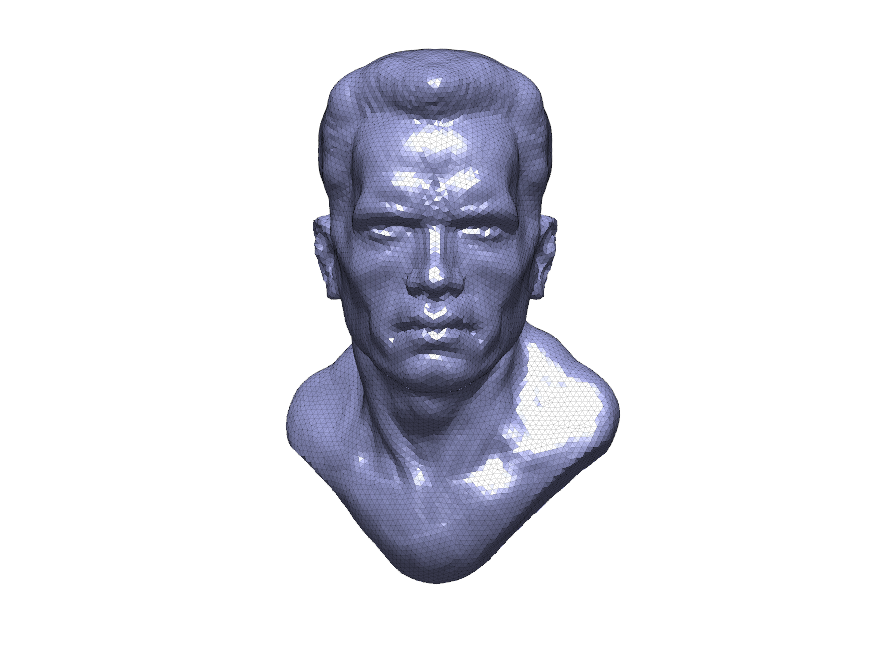} &
        \includegraphics[clip,trim = {3cm 2cm 2.5cm 1.5cm},height = 3cm]{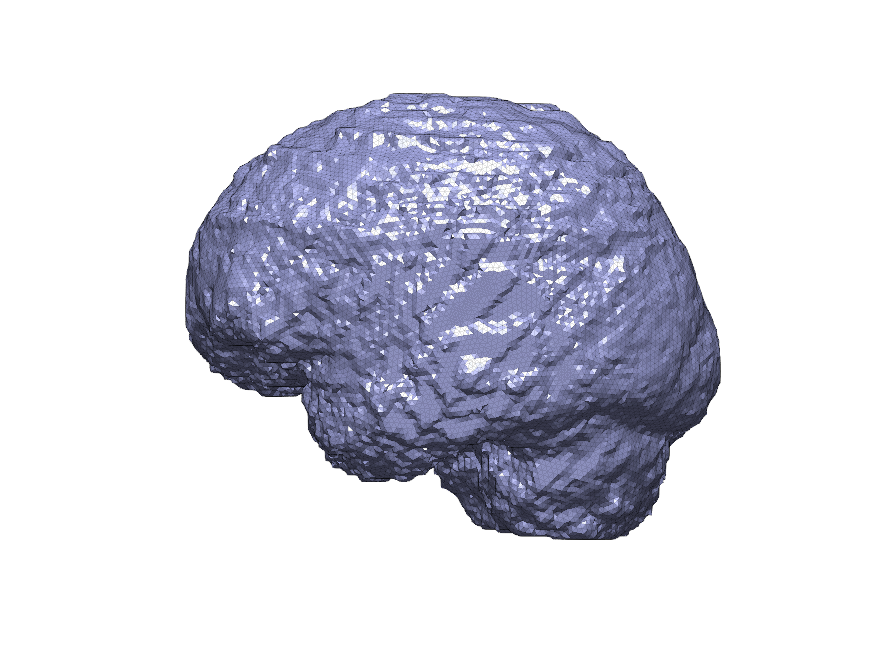} &
        \includegraphics[clip,trim = {2.5cm 0.9cm 2.5cm 1cm},height = 3.5cm]{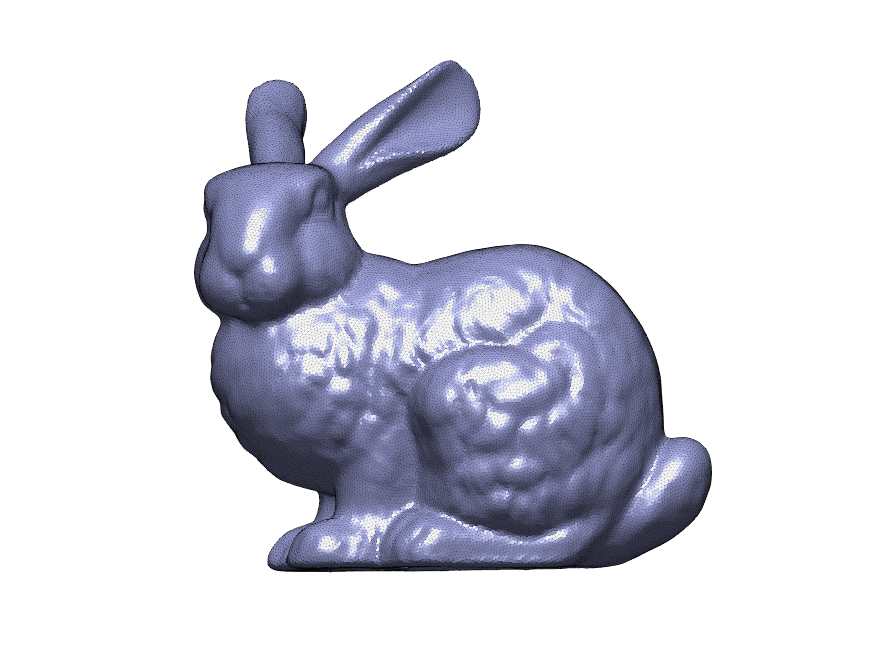} \\
Apple & Arnold & Brain & Bunny \\
        \includegraphics[clip,trim = {3.5cm 2.25cm 2cm 2cm},height = 2.5cm]{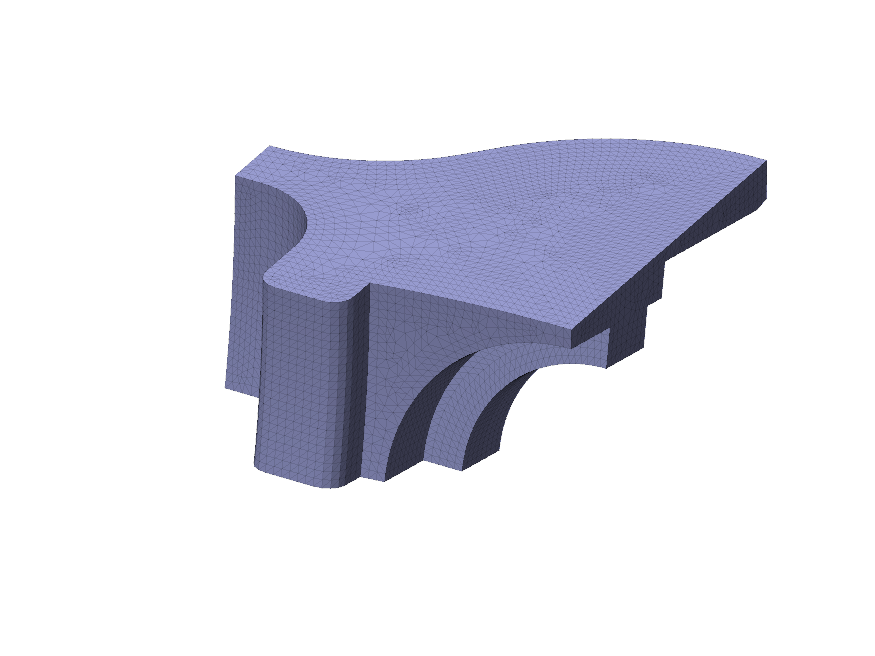} &
        \includegraphics[clip,trim = {4cm 1.25cm 4.5cm 0.5cm},height = 3.5cm]{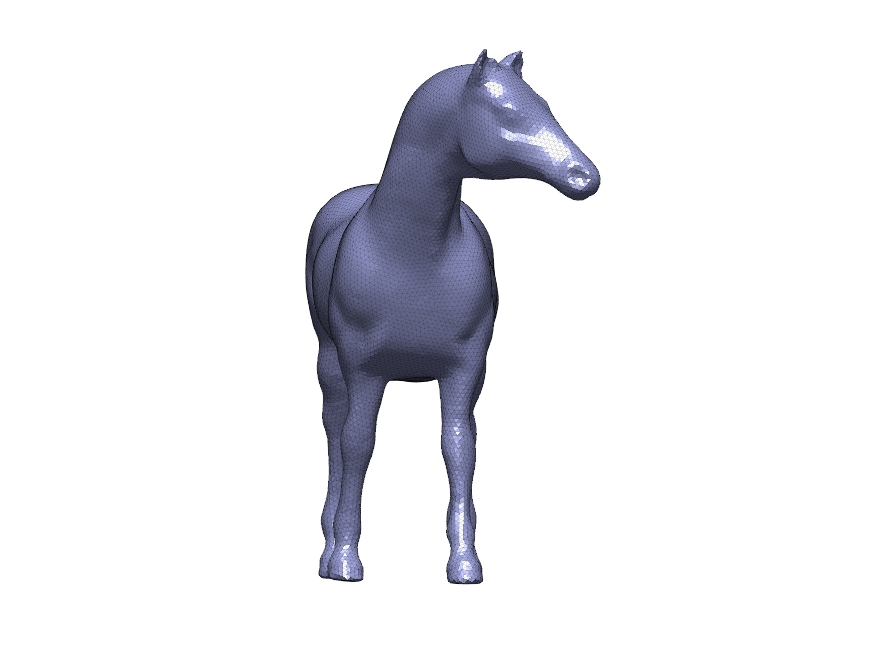} &
        \includegraphics[clip,trim = {4cm 1.75cm 4cm 1cm},height = 3.5cm]{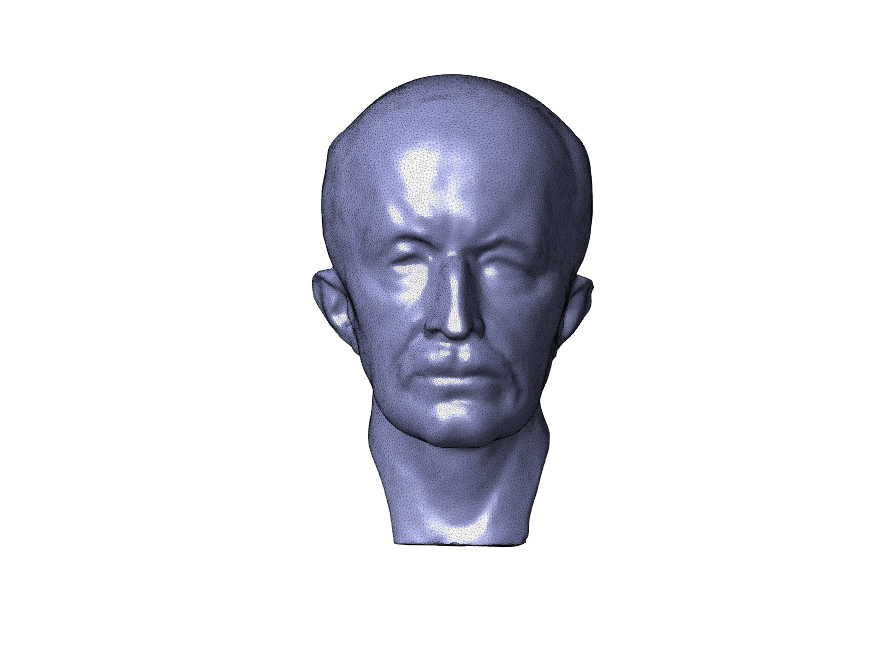} &
        \includegraphics[clip,trim = {4cm 1.25cm 4cm 1cm},height = 3.5cm]{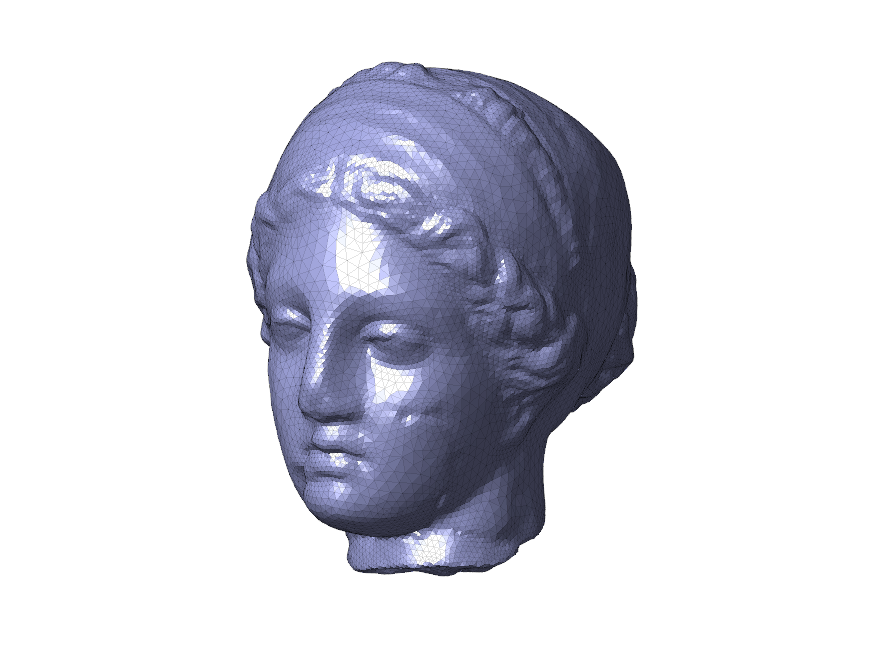} \\
Fandisk &  Horse & \hspace{0.4cm}Planck & Venus \\
\end{tabular}
\caption{The triangulation models for the experiments.}
    \label{fig:meshes}
\end{figure}

\subsection{Convergence behavior and conformal distortion}

We first present the convergence behavior of the proposed
\Cref{alg:TR}. \Cref{fig:fgx} shows the relationship between the number of iterations $k$ and conformal energy $E_C^{(k)}$, the infinity norm of gradient $\|\mathbf{g}^{(k)}\|_\infty$ and the error $\delta^{(k)}$ by \Cref{alg:TR} for models in \Cref{fig:meshes}.
As shown, the conformal energy $E_C^{(k)}$ with the scale on the right decreases linearly first. Meanwhile, $\|\mathbf{g}^{(k)}\|_{\infty}$ and $\delta^{(k)}$ with the scale on the left remain stable within a range. At this moment, the gradient and Newton directions are utilized for the trial step. Then, $E_C^{(k)}$ tends to level off, while $\|\mathbf{g}^{(k)}\|_\infty$ and $\delta^{(k)}$ descend sharply. More specifically, $\|\mathbf{g}^{(k)}\|_\infty$ becomes $10^{-10}$ on most of the models when the iteration loops terminate, implying that the iteration stops at a critical point and the algorithm converges. Moreover, $\delta^{(k)}$ descends in quadratic order. Taking {\it{Brain}} as an example, we have
\[
\delta^{(25)} = 7.4\times 10^{-4},\quad
\delta^{(26)} = 9.3\times 10^{-6}, \quad
\delta^{(27)} = 5.9\times 10^{-10}.
\]
Obviously, it shows the quadratic convergence of the HBTR algorithm.

\begin{figure}[t]
    \centering
\resizebox{\textwidth}{!}{
\begin{tabular}{c@{}c@{}c@{}c}
        \includegraphics[width = 0.24\textwidth]{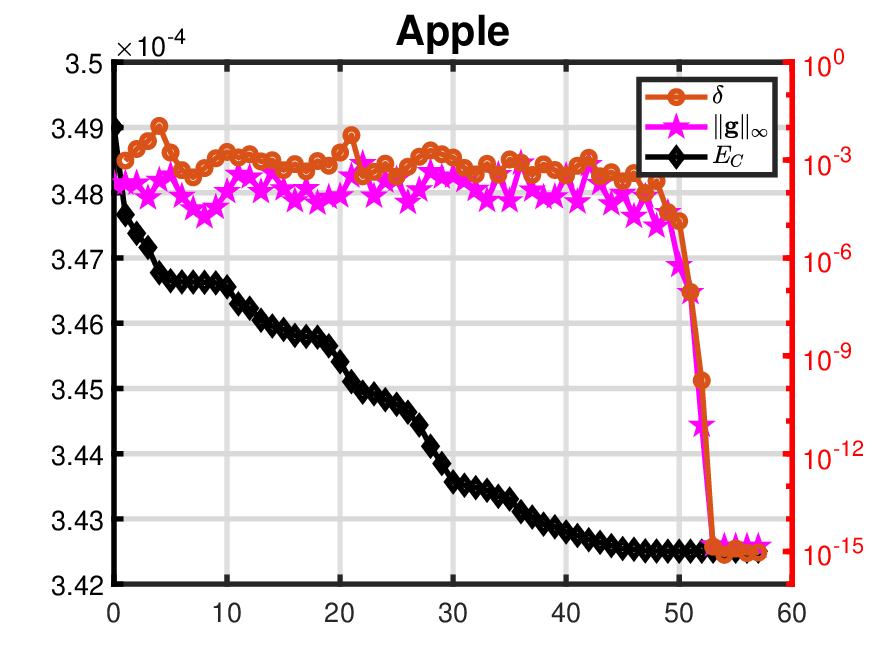} &
        \includegraphics[width = 0.24\textwidth]{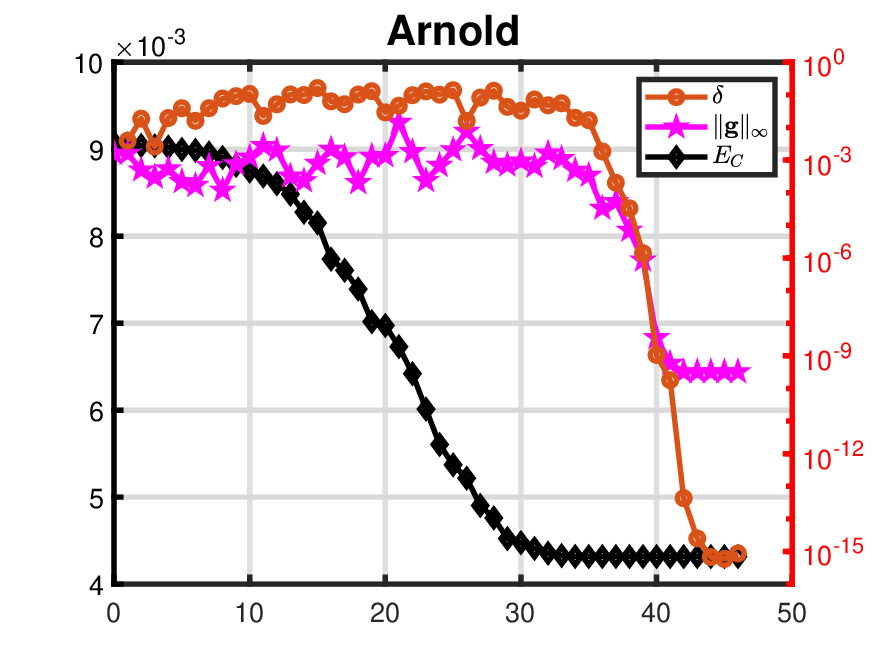} &
        \includegraphics[width = 0.24\textwidth]{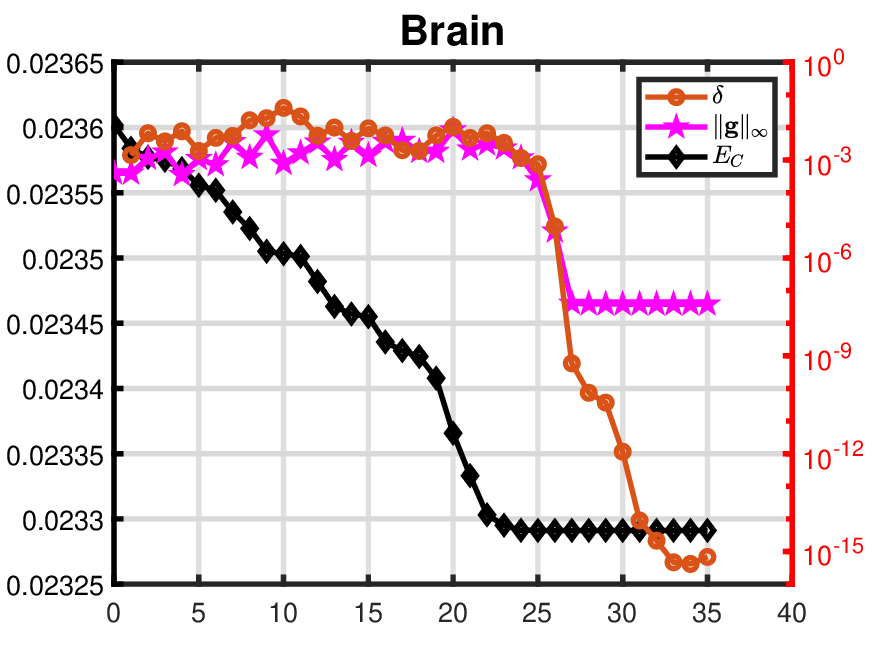} &
        \includegraphics[width = 0.24\textwidth]{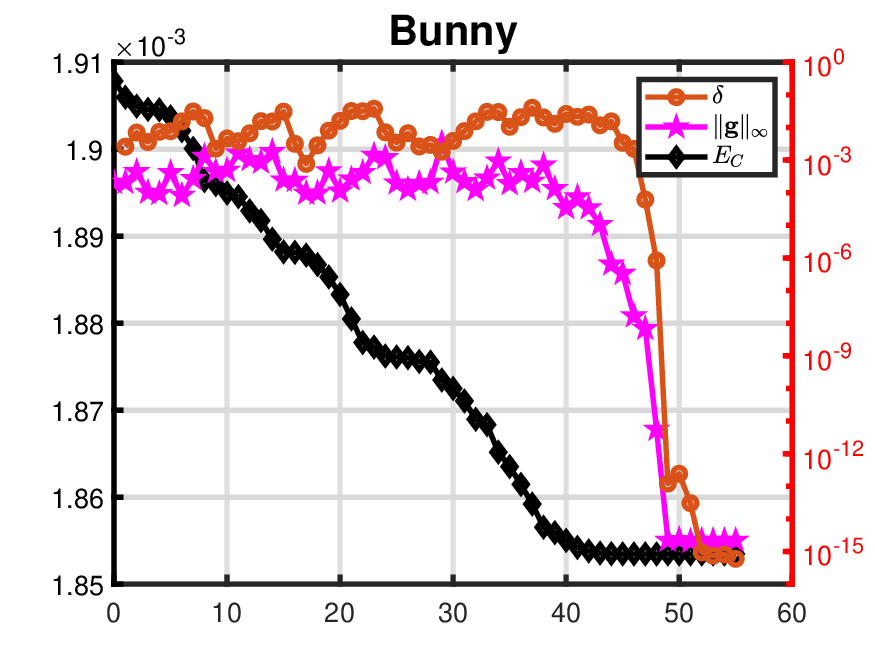} \\
        \includegraphics[width = 0.24\textwidth]{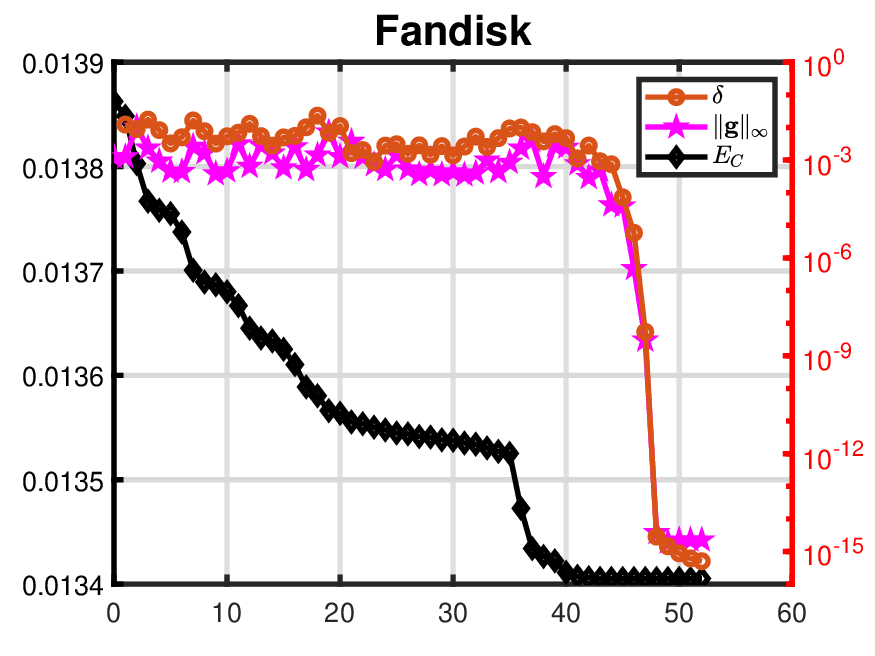} &
        \includegraphics[width = 0.24\textwidth]{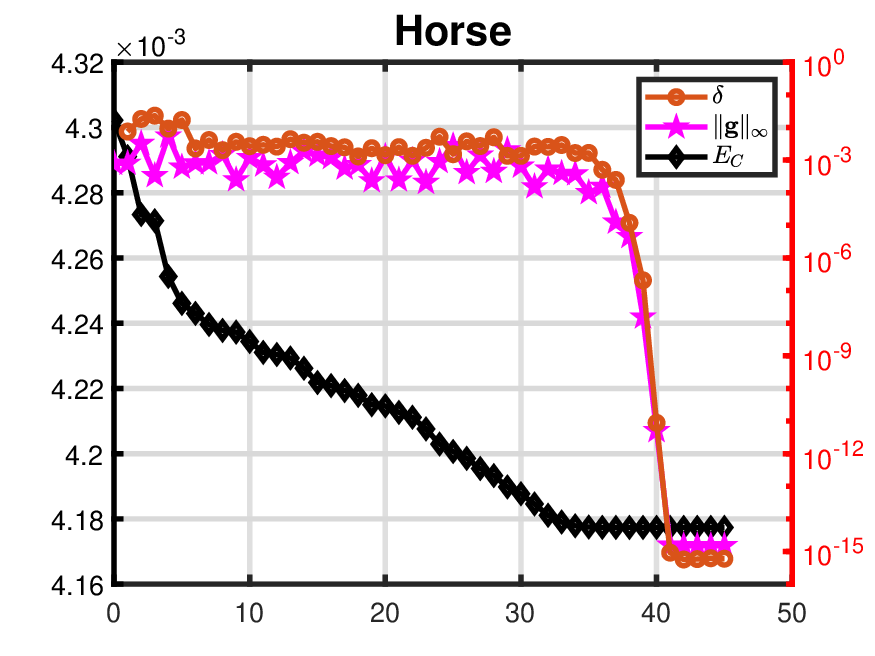} &
        \includegraphics[width = 0.24\textwidth]{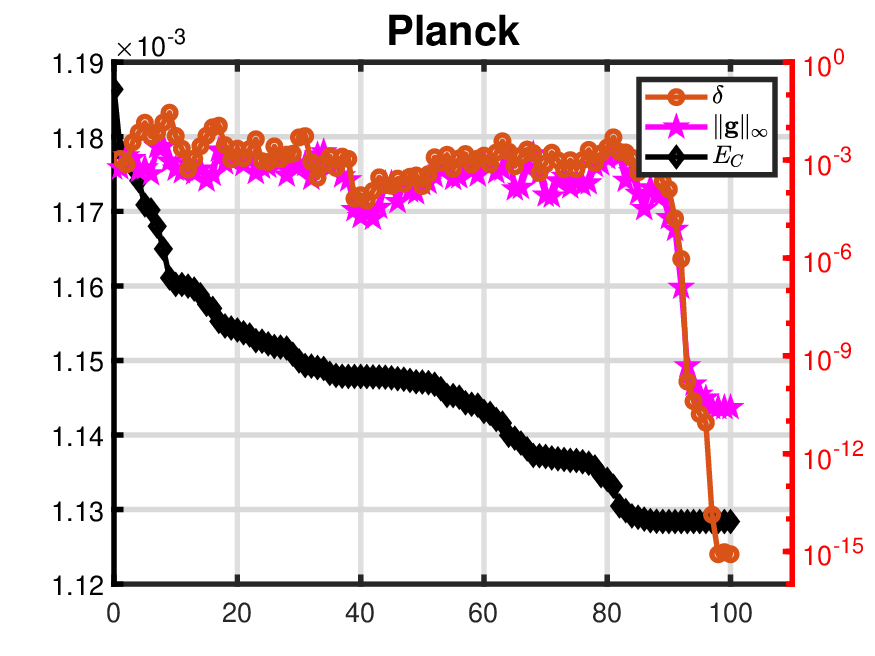} &
        \includegraphics[width = 0.24\textwidth]{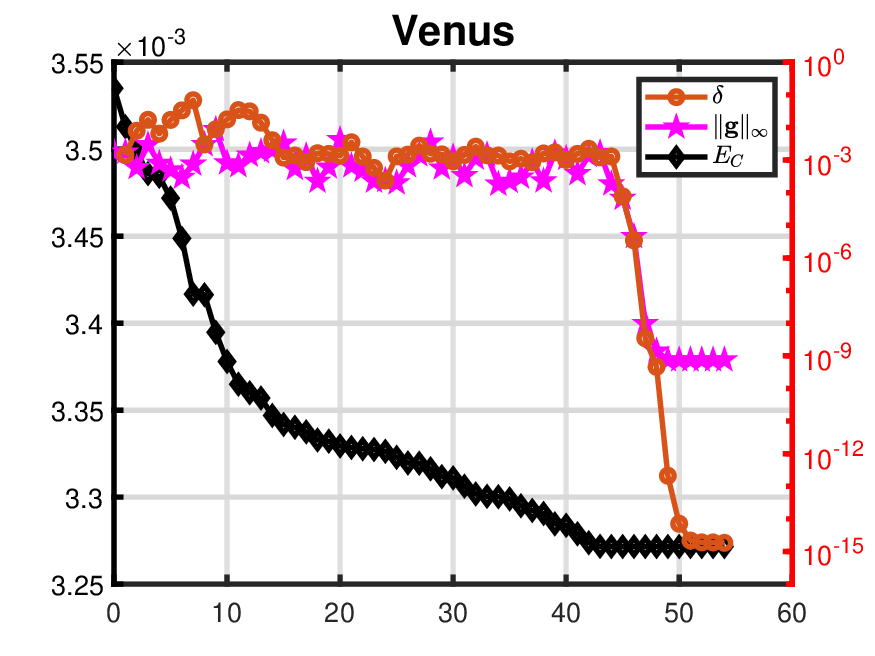} \\
\end{tabular}
}
\caption{The relationship between the number of iterations $k$ and the conformal energy $E_C^{(k)}$, the infinity norm of gradient $\|\mathbf{g}^{(k)}\|_\infty$ and the error $\delta^{(k)}$ by \Cref{alg:TR}. The $x$-axes represent the number of iterations. The left $y$-axes represent the conformal energies, and the right $y$-axes represent infinity norms of gradient and errors, respectively.}
    \label{fig:fgx}
\end{figure}

\Cref{fig:HistogramAngle,fig:HistogramBC} present the histograms of absolute angle distortion (degree) of each angle, denoted as $|\alpha_{jk} - \alpha_{jk}(f)|$ for $v_i$ in $T_{ijk}$, and Beltrami coefficients $\mu$ \cite{PTKC15} of triangular faces, respectively.
If $|\mu| = 0$, the map $f$ is conformal.
The subplots in the upper right of each histogram are the front and back of the angle distortion distributions/Beltrami coefficient distributions on the resulting unit spheres, respectively.
One can see that most of angle distortions are less than $5$ degrees and most of Beltrami coefficients are less than $0.1$ as well, guaranteeing the conformal performance of the HBTR algorithm. Angle distortion has similar performance to the Beltrami coefficient. Furthermore, \Cref{fig:GCAE} shows the absolute value of the discrete Gauss curvature $|\kappa|$, the average angle distortion $\epsilon_\alpha$ and the average of the norm of the Beltrami coefficient $\overline{|\mu|}$ at each vertex for the models {\it Arnold}, {\it Brain}, {\it Fandisk} and {\it Horse}, denoted as
\begin{align*}
     \kappa(v_i) = & 2\pi - \sum_{\{j,k\}\in S_{\mathcal{E}}(i)} \alpha_{jk},\\
    \epsilon_\alpha(v_i) = & \frac{1}{|\mathcal{N}(i)|}\sum_{\{j,k\}\in S_{\mathcal{E}}(i)} |\alpha_{jk} - \alpha_{jk}(f)|,\\
    \overline{|\mu|}(v_i) = & \frac{1}{|\mathcal{N}(i)|}\sum_{T_{ijk}\ni i} |\mu(T_{ijk})|,
\end{align*}
where $|\mathcal{N}(i)|$ is the number of adjacent vertices of $v_i$ and $\mu(T_{ijk})$ is the Beltrami coefficient on triangle $T_{ijk}$. We approximately find $\overline{|\mu|} \propto \epsilon_\alpha \propto \sqrt{|\kappa|}$ in these $4$ models except {\it Brain}, which approximately satisfies $\overline{|\mu|} \propto \epsilon_\alpha \propto |\kappa|$. The figures are plotted according to the relationships.
The textures illustrate the high similarity of high curvatures, large angle distortions and large Beltrami coefficient distributions. In other words, the large angle distortion and the large Beltrami coefficient regions are mainly at those with high curvatures, such as the ears, eyes and nose of {\it Arnold}, corners of {\it Fandisk}, and ears and legs of {\it Horse}. The small angle distortion and small Beltrami coefficient regions are roughly at those with low curvature. Therefore, the HBTR performs relatively poorly at vertices with high curvature, which is an issue in our future work. Additionally, the angle distortion and the Beltrami coefficient have diffusion trends. Taking {\it Fandisk} as an example, high-curvature regions are on corners and edges, and the curvatures in other regions are mostly $0$. Furthermore, the angle distortions and the Beltrami coefficients diffuse from the corners to the adjacent regions gradually. These phenomena also occur in the other models, which are not shown.

\begin{figure}[thp]
    \centering
\resizebox{\textwidth}{!}{
\begin{tabular}{c@{}c@{}c@{}c@{}l}
        \includegraphics[width = 0.24\textwidth]{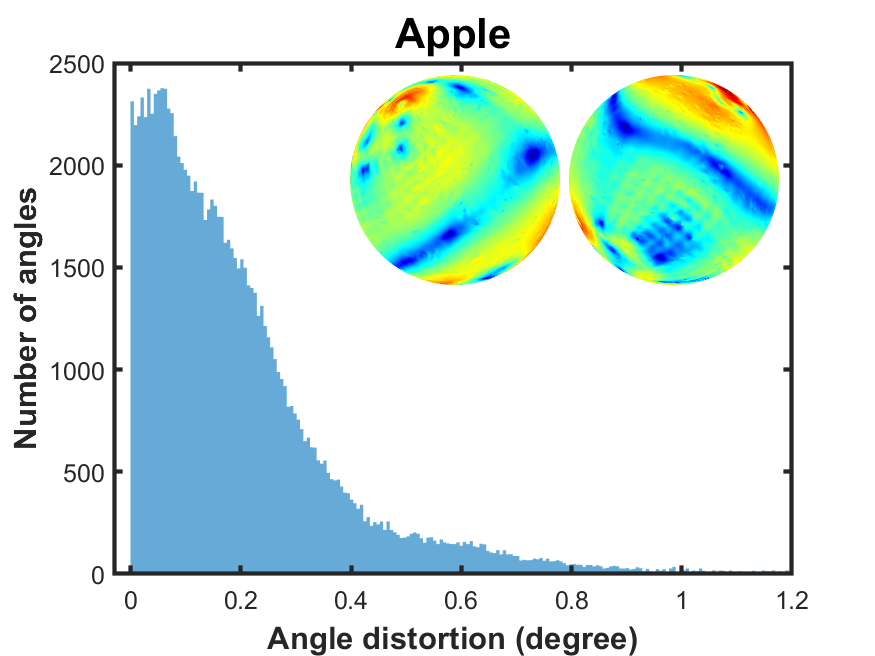} &
        \includegraphics[width = 0.24\textwidth]{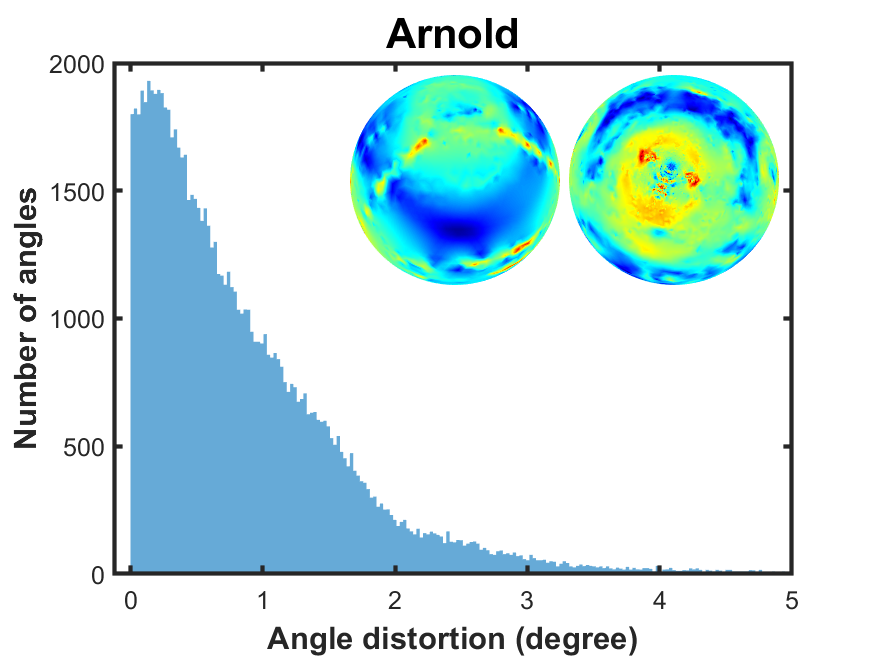} &
        \includegraphics[width = 0.24\textwidth]{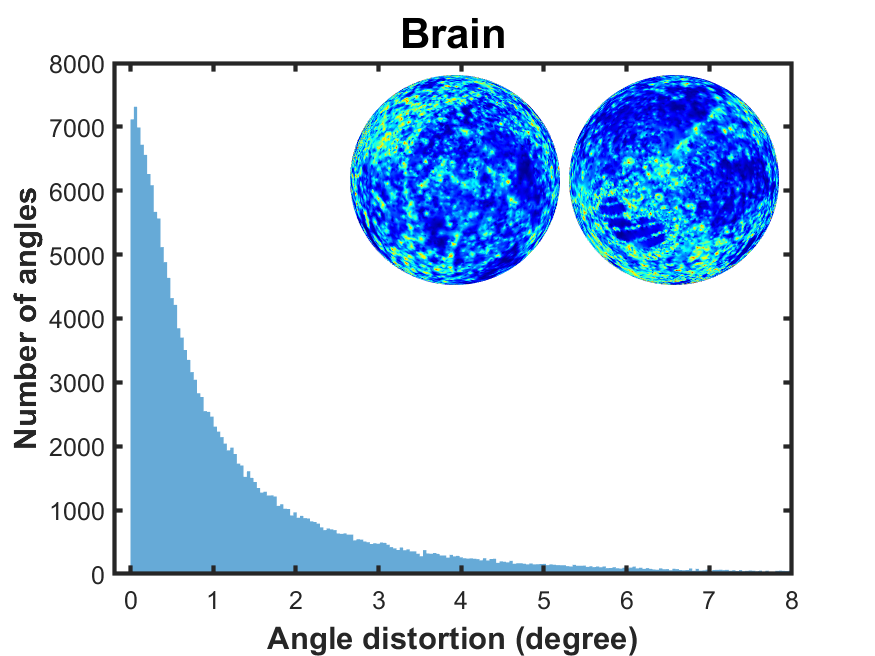} &
        \includegraphics[width = 0.24\textwidth]{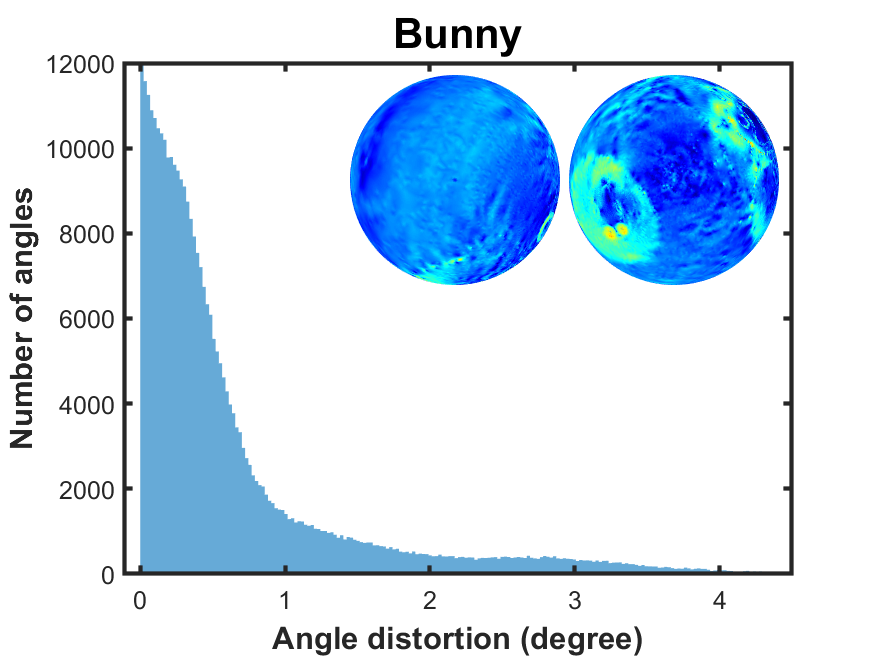} &
\multirow[t]{2}{*}[-2.5cm]{\includegraphics[clip, trim=12cm 0cm 0cm 0cm,width = 0.07\textwidth]{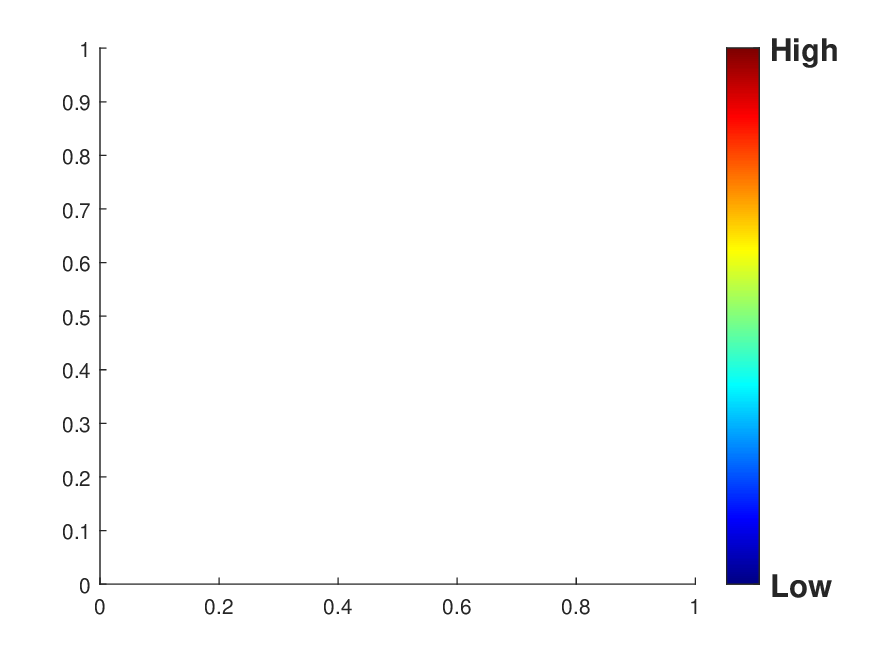}}\\
        \includegraphics[width = 0.24\textwidth]{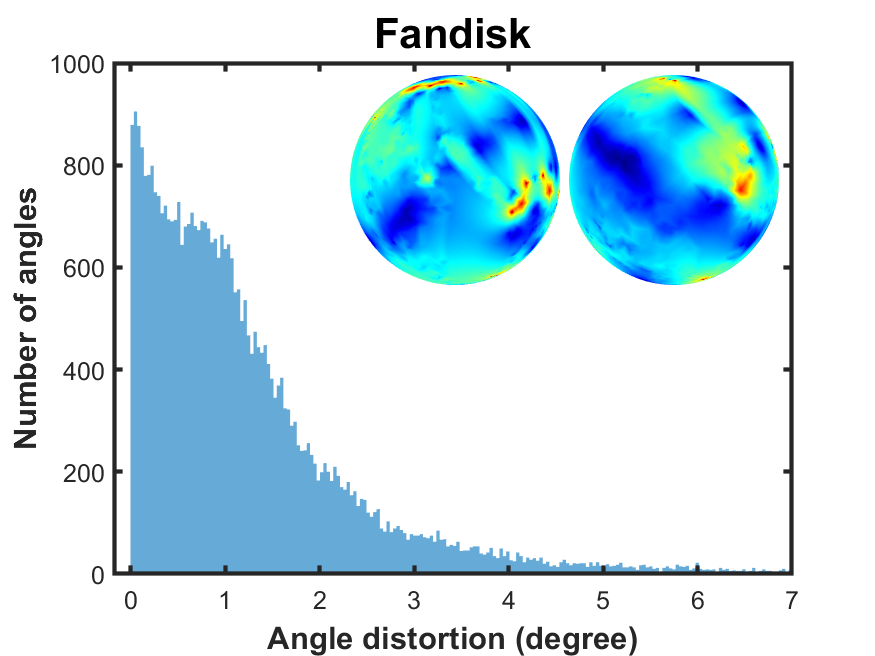} &
        \includegraphics[width = 0.24\textwidth]{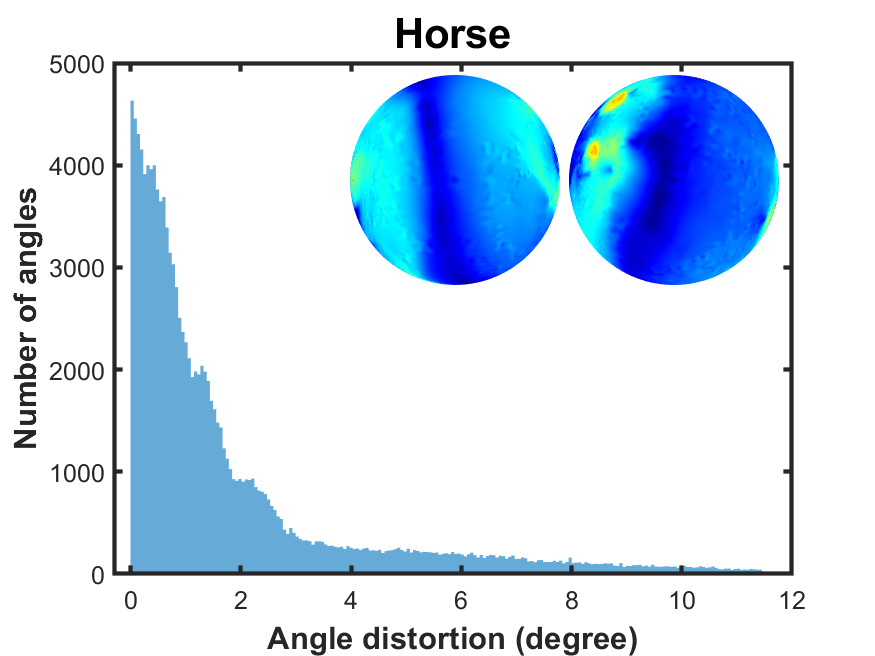} &
        \includegraphics[width = 0.24\textwidth]{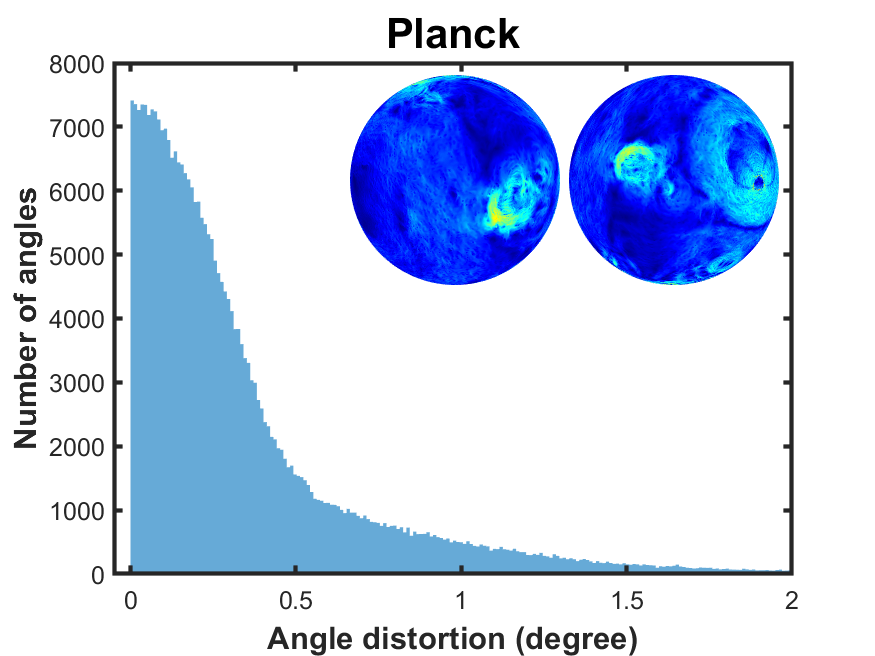} &
        \includegraphics[width = 0.24\textwidth]{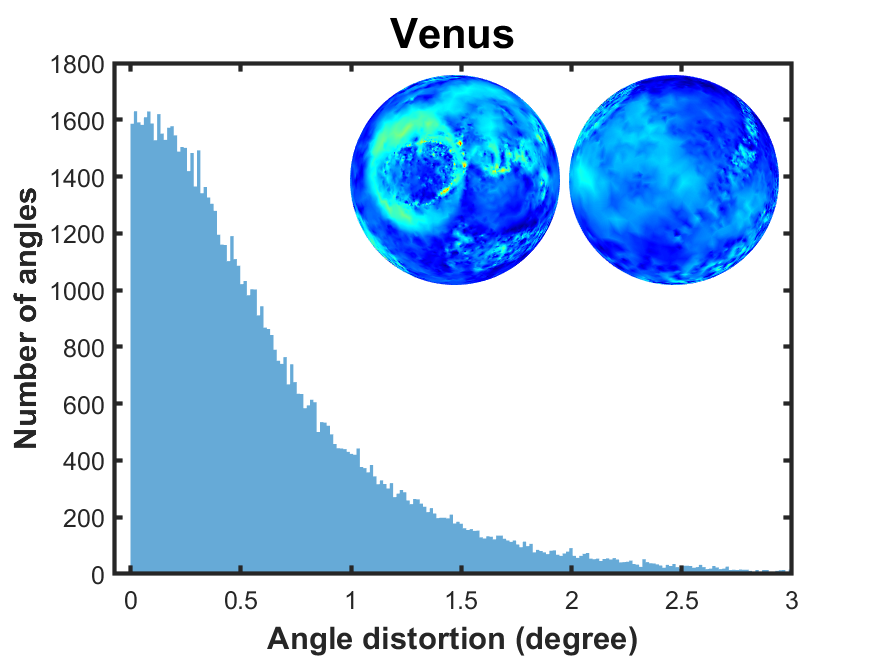}
\end{tabular}
}
\caption{Histograms of angle distortions on triangulation models.}
    \label{fig:HistogramAngle}
\end{figure}

\begin{figure}[thp]
    \centering
\resizebox{\textwidth}{!}{
\begin{tabular}{c@{}c@{}c@{}c@{}l}
        \includegraphics[width = 0.24\textwidth]{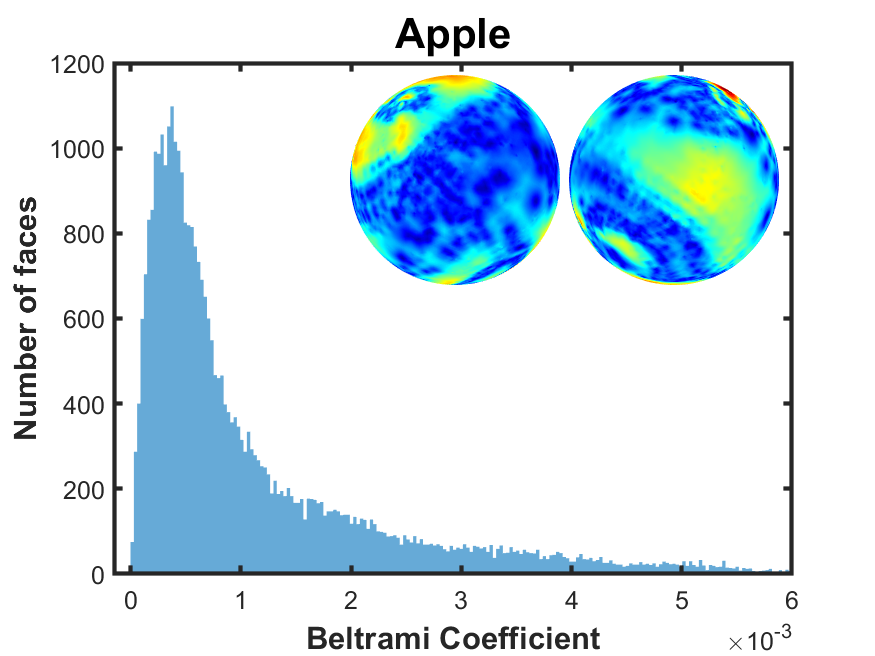} &
        \includegraphics[width = 0.24\textwidth]{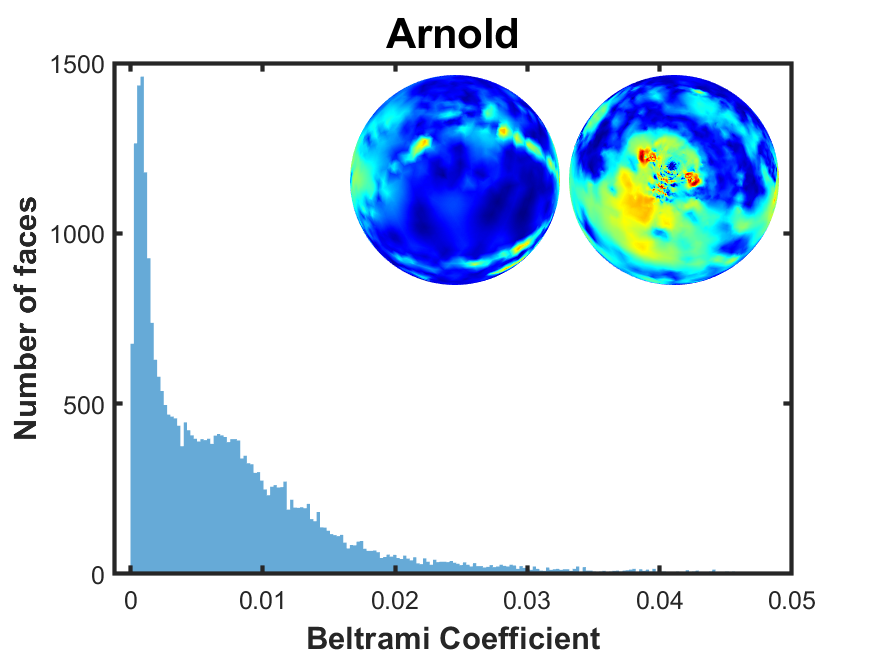} &
        \includegraphics[width = 0.24\textwidth]{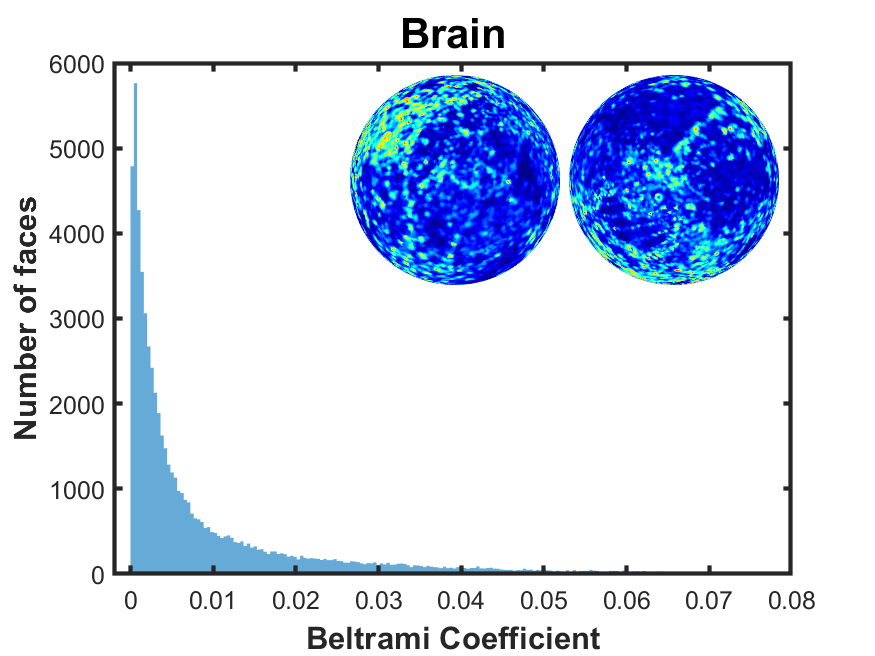} &
        \includegraphics[width = 0.24\textwidth]{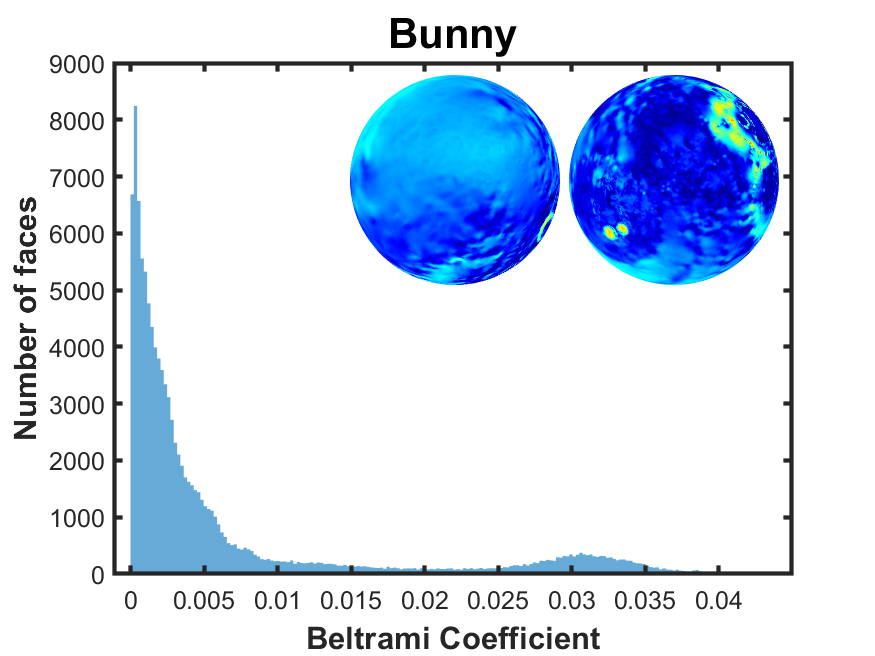} &
\multirow[t]{2}{*}[-2.5cm]{\includegraphics[clip, trim=12cm 0cm 0cm 0cm,width = 0.07\textwidth]{images/colorbar.eps}}\\
        \includegraphics[width = 0.24\textwidth]{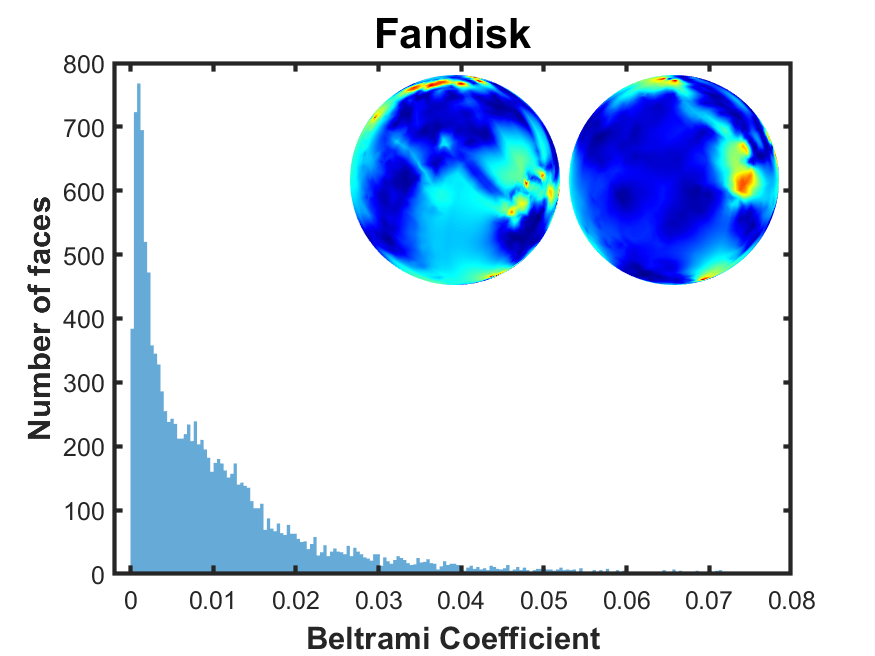} &
        \includegraphics[width = 0.24\textwidth]{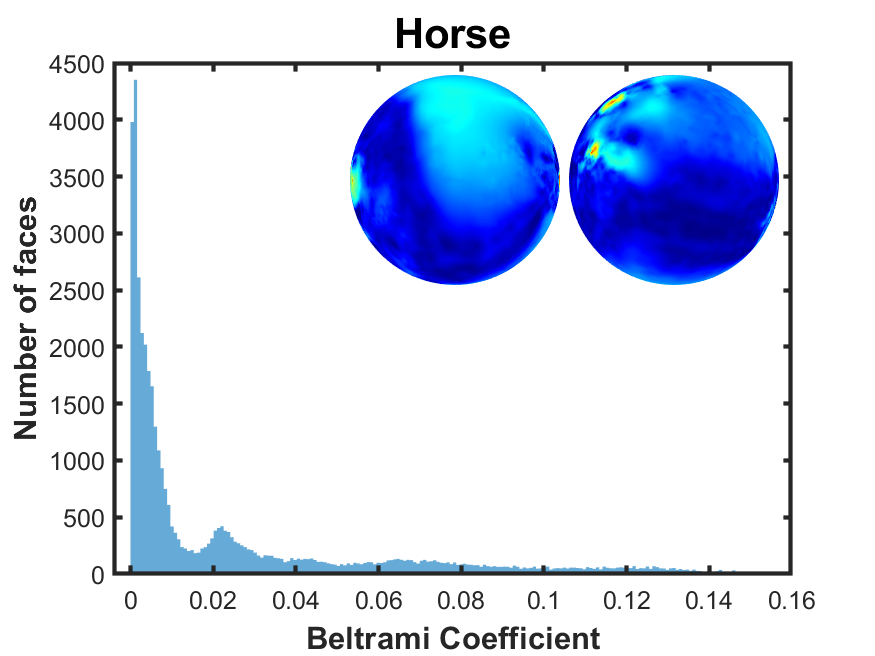} &
        \includegraphics[width = 0.24\textwidth]{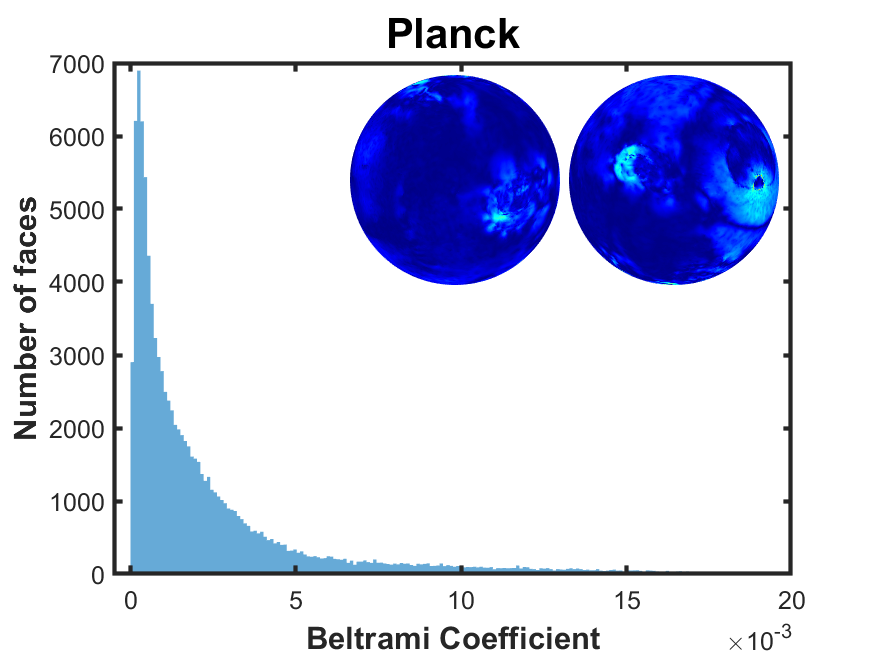} &
        \includegraphics[width = 0.24\textwidth]{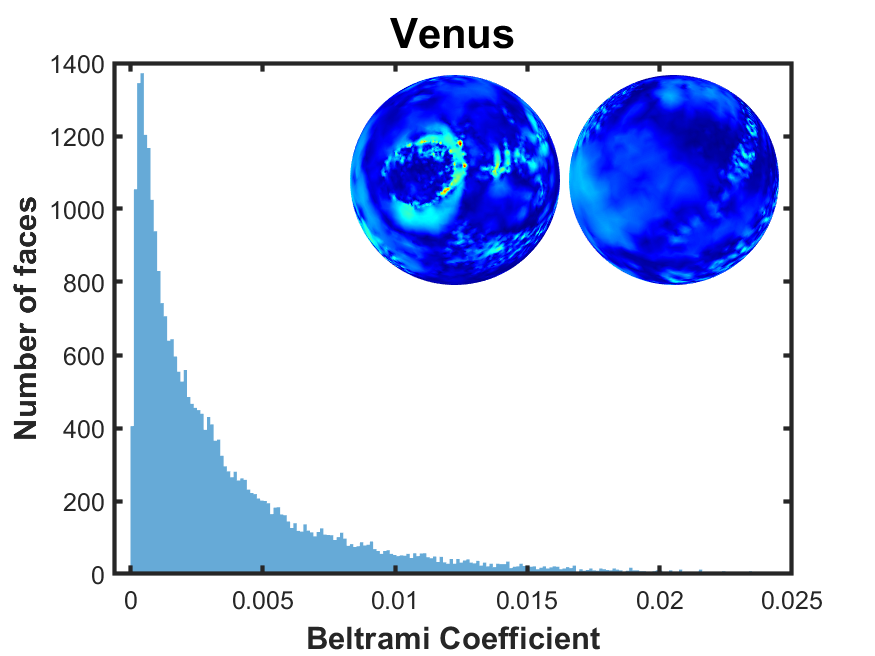}
\end{tabular}
}
\caption{Histograms of the Beltrami coefficients in the triangulation models.}
    \label{fig:HistogramBC}
\end{figure}

\begin{figure}[htp]
    \centering
\begin{tabular}{ccccl}
        \includegraphics[clip,trim = {4.5cm 1.25cm 4cm 0.5cm},height = 3.5cm]{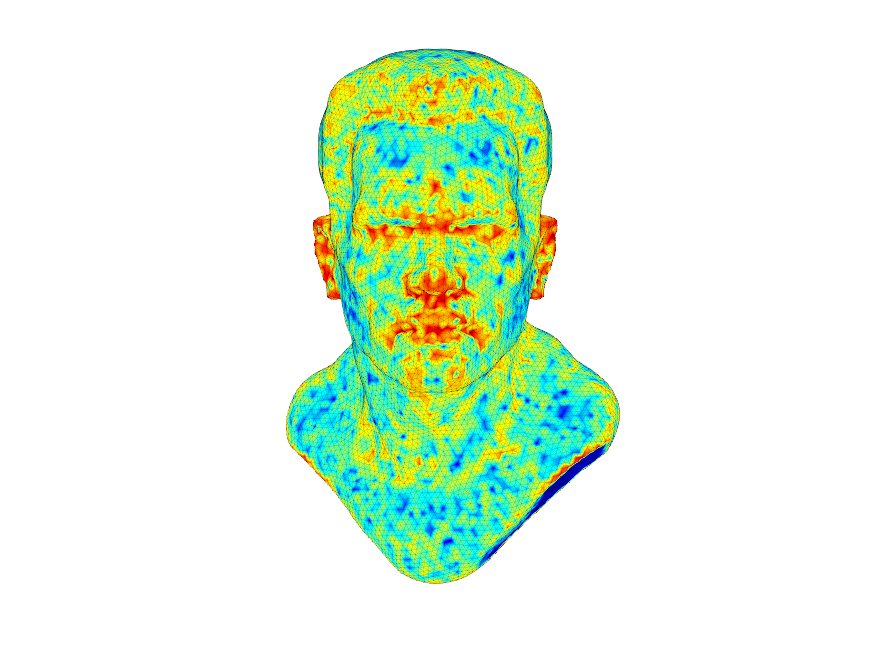} &
        \includegraphics[clip,trim = {3cm 2cm 2.5cm 1.5cm},width = 0.2\textwidth]{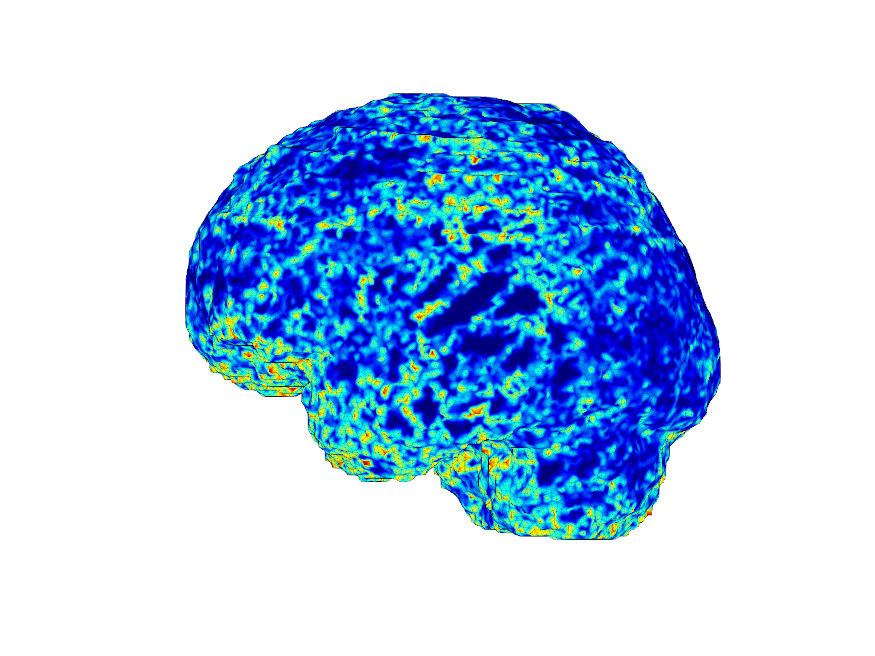} &
        \includegraphics[clip,trim = {3.5cm 2.25cm 2cm 2cm},width = 0.2\textwidth]{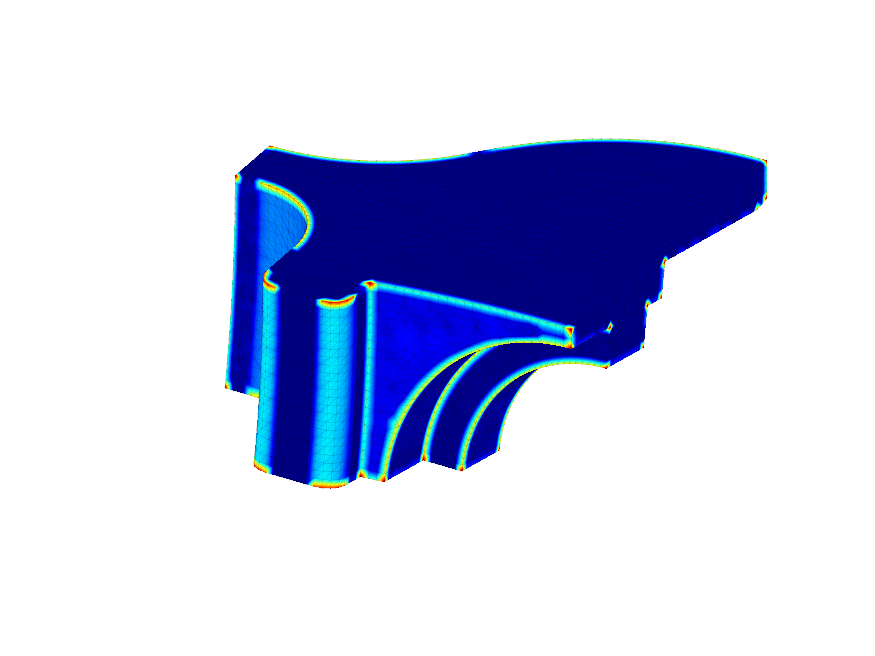} &
        \includegraphics[clip,trim = {4cm 1.25cm 4.5cm 0.5cm},height = 3.5cm]{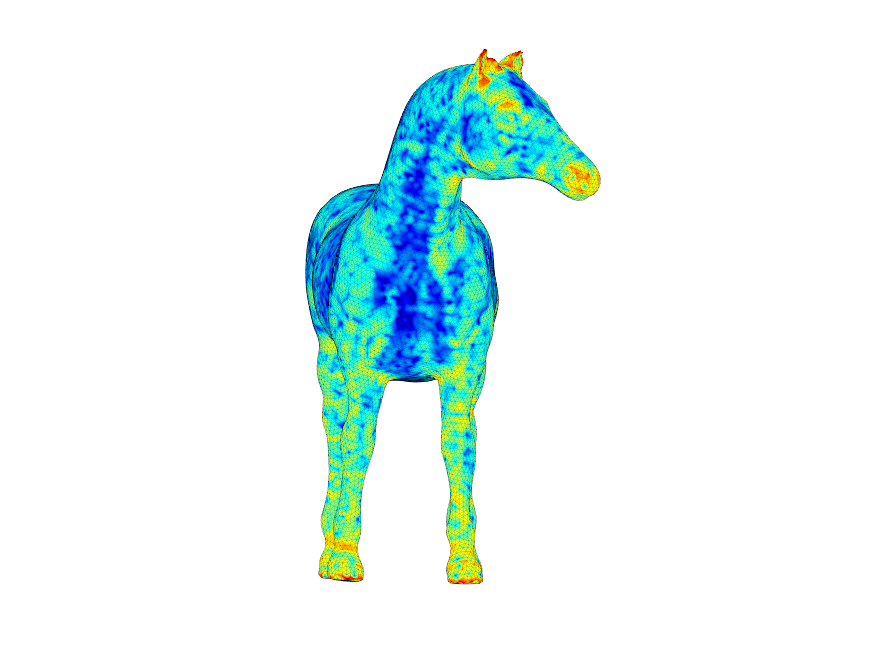} &
\multirow[t]{3}{*}[-6cm]{\includegraphics[clip, trim=12cm 0cm 0cm 0cm,width = 0.1\textwidth]{images/colorbar.eps}}\\
\multicolumn{4}{c}{(a) Gauss curvatures}\\
        \includegraphics[clip,trim = {4.5cm 1.25cm 4cm 0.5cm},height = 3.5cm]{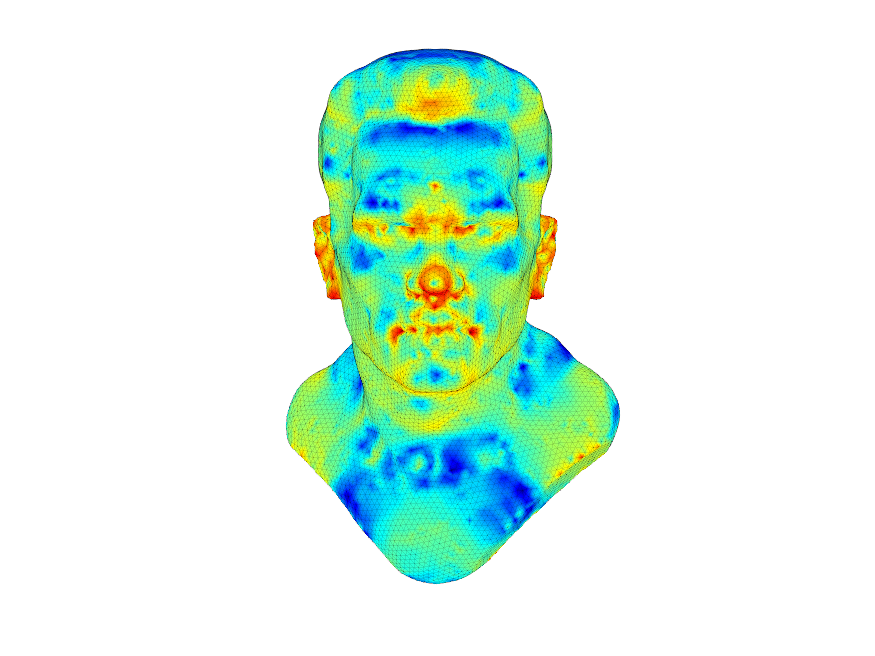} &
        \includegraphics[clip,trim = {3cm 2cm 2.5cm 1.5cm},width = 0.2\textwidth]{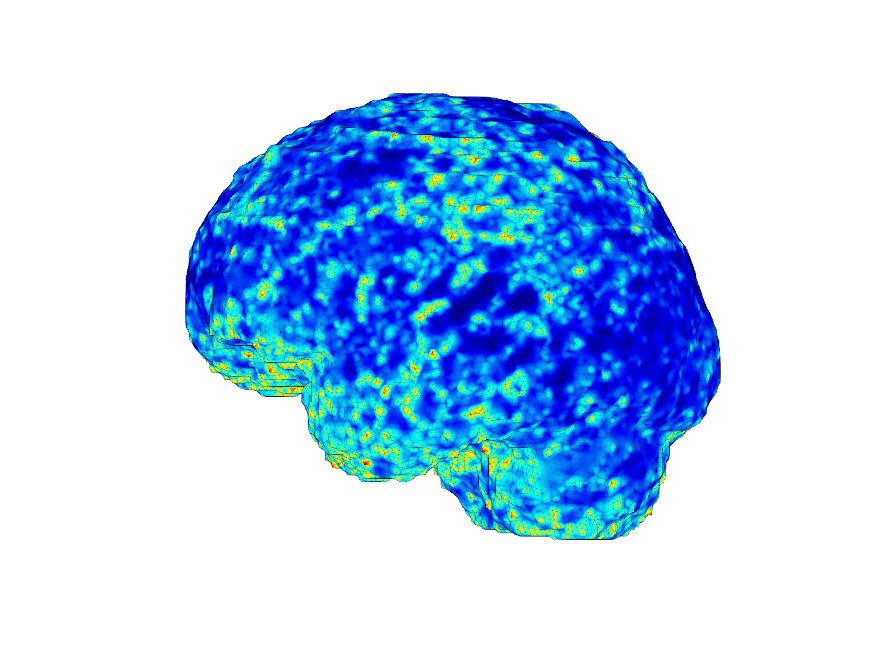} &
        \includegraphics[clip,trim = {3.5cm 2.25cm 2cm 2cm},width = 0.2\textwidth]{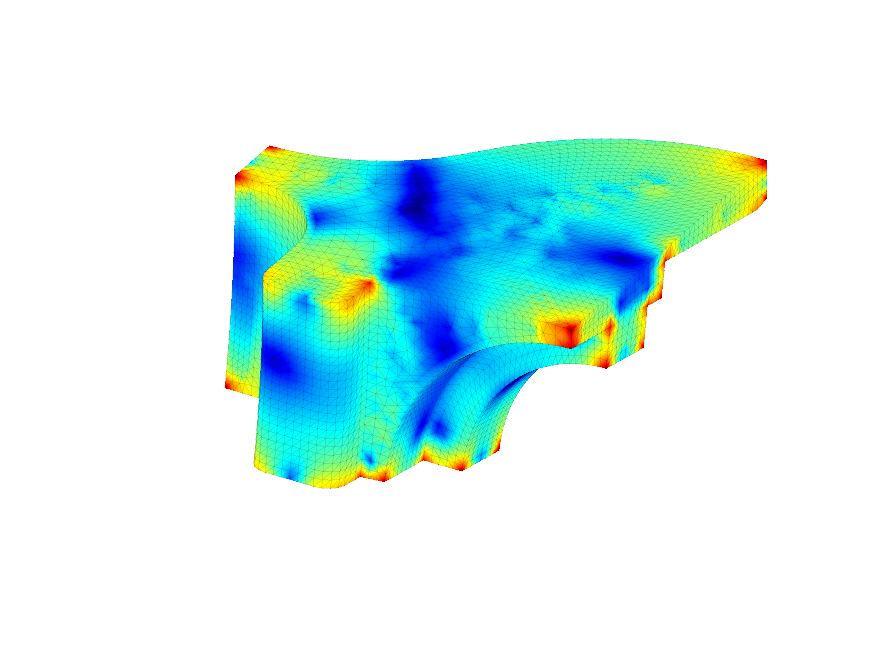} &
        \includegraphics[clip,trim = {4cm 1.25cm 4.5cm 0.5cm},height = 3.5cm]{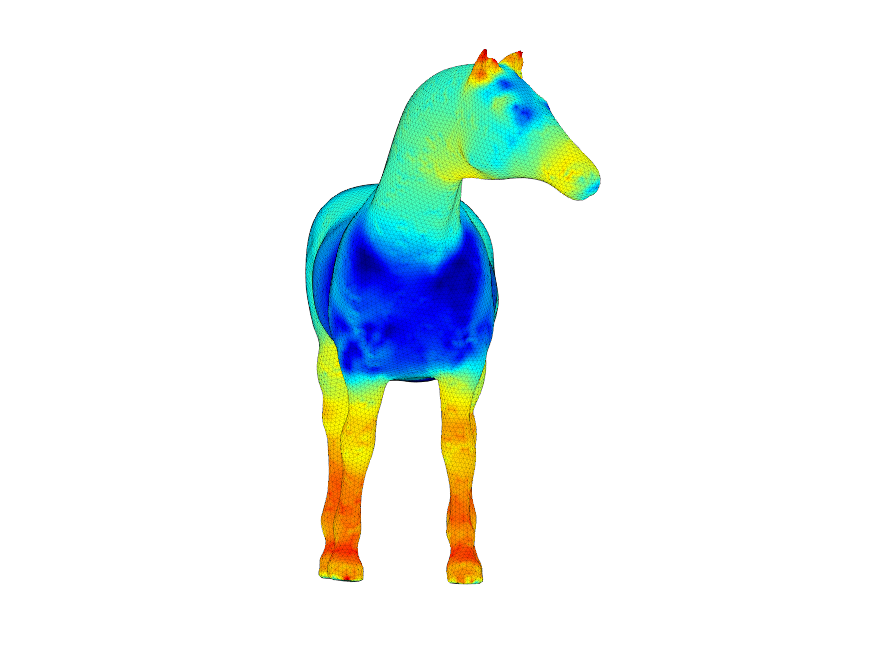} & \\
\multicolumn{4}{c}{(b) Angle distortions}\\
        \includegraphics[clip,trim = {4.5cm 1.25cm 4cm 0.5cm},height = 3.5cm]{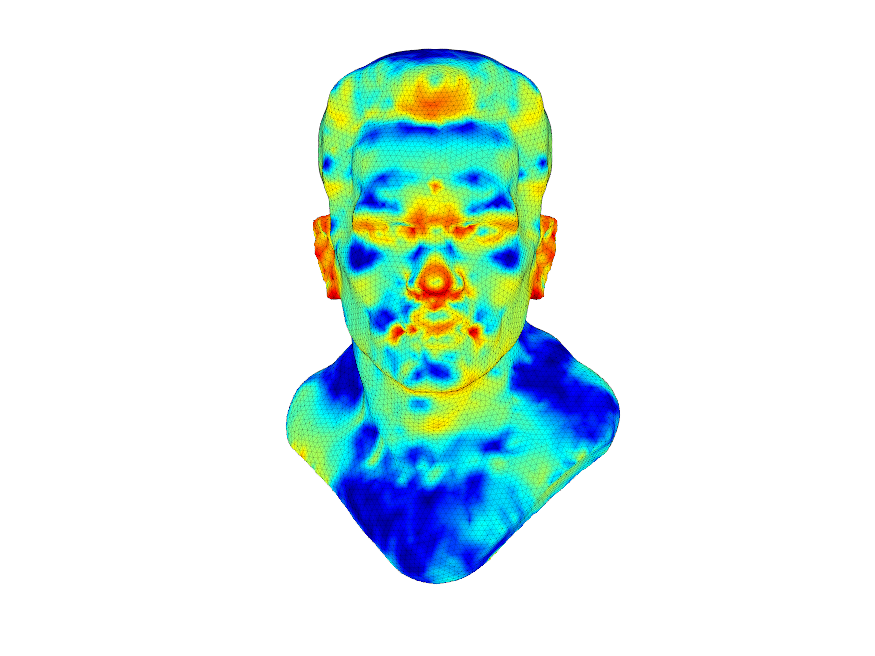} &
        \includegraphics[clip,trim = {3cm 2cm 2.5cm 1.5cm},width = 0.2\textwidth]{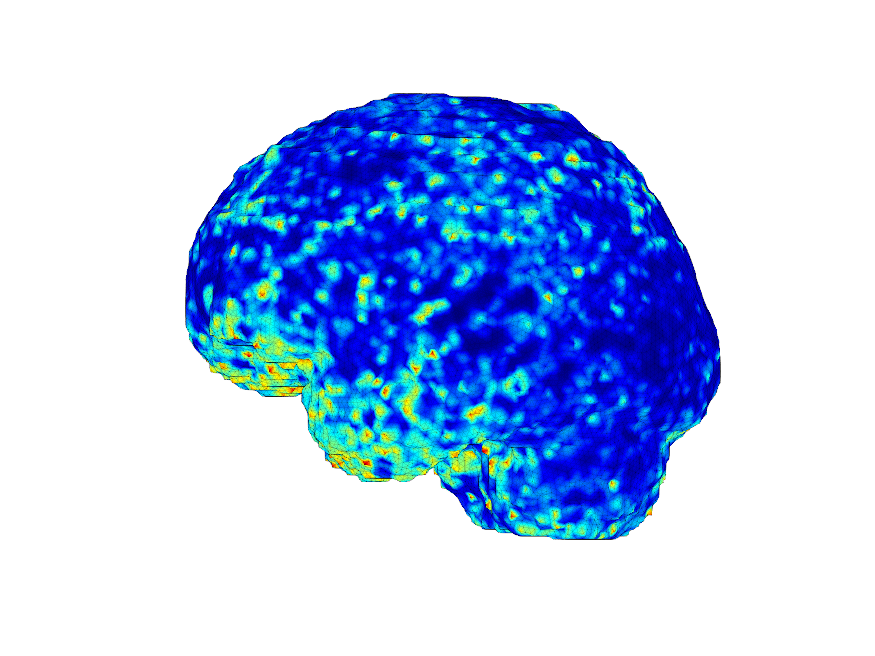} &
        \includegraphics[clip,trim = {3.5cm 2.25cm 2cm 2cm},width = 0.2\textwidth]{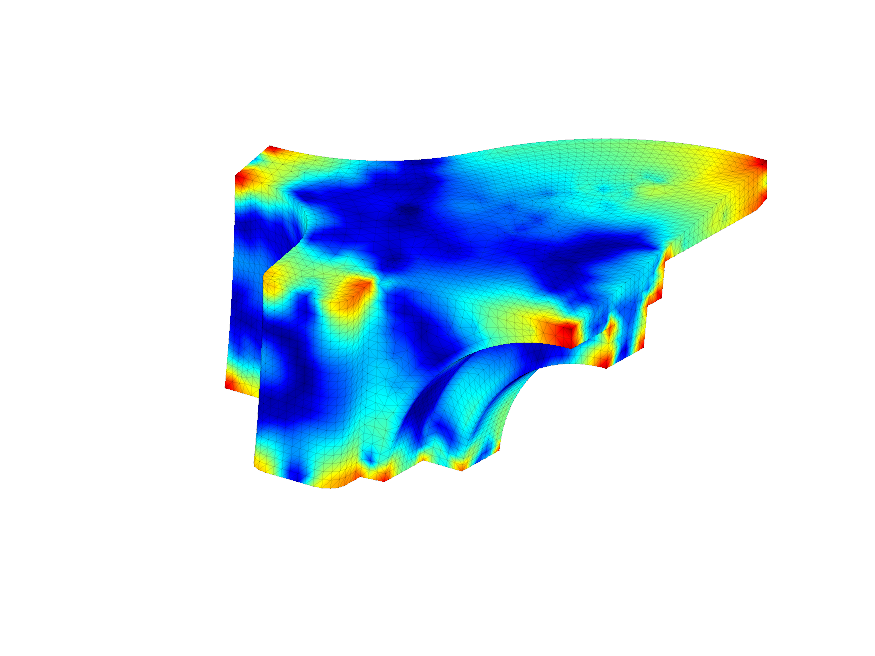} &
        \includegraphics[clip,trim = {4cm 1.25cm 4.5cm 0.5cm},height = 3.5cm]{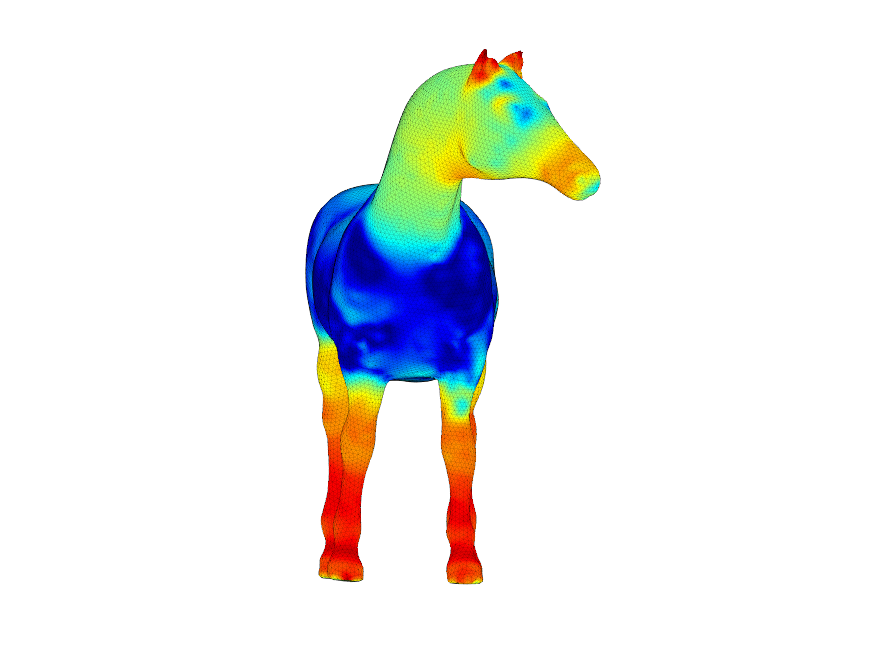} \\
\multicolumn{4}{c}{(c) Beltrami coefficients}
\end{tabular}
\caption{(a) Gauss curvature distributions (top), (b) angle distortion distributions, and (c) Beltrami coefficient distributions (bottom) for models {\it Arnold}, {\it Brain}, {\it Fandisk} and {\it Horse}.}
    \label{fig:GCAE}
\end{figure}

\subsection{Comparison with state-of-the-art algorithms}

In this subsection, we compare the accuracy of the proposed HBTR algorithm with that of two state-of-the-art algorithms for spherical conformal parameterization, namely, FLASH \cite{PTKC15} and the SCEM \cite{MHTL19}. The algorithm FLASH, simply put, via the stereographic projection, applies the composition of two quasi-conformal maps to construct the ideal conformal map, which is not an iterative algorithm. The MATLAB program of FLASH is obtained from Choi's website \cite{Choi}. The SCEM algorithm, as mentioned in \Cref{sec:conformal}, adopts the north-south hemisphere alternating iteration to compute the conformal map. The maximum number of iterations of the HBTR and SCEM is $500$. The loop termination condition of SCEM is that the difference between the conformal energies of two consecutive iteration steps is less than $10^{-9}$. For the HBTR algorithm, we set $\varepsilon = 10^{-9}$, as in \Cref{fig:fgx}, which guarantees convergence.

\Cref{tab:comparison} shows the comparison of conformal energies between FLASH, SCEM and HBTR. We can see that the conformal energies by HBTR are smallest among all models. In the view of angle distortions, it is observed that the HBTR algorithm has well performance for $50$-th percentile and $75$-th percentile as FLASH and SCEM. In addition, among the $8$ testing examples, FLASH, SCEM and HBTR did not produce foldings for spherical conformal maps.

\begin{table}[thp]
    \centering
    \resizebox{\textwidth}{!}{
    \begin{tabular}{|c|ccc|ccc|ccc|}
    \hline
    \multirow{4}{*}{Mesh} &
    \multicolumn{3}{c|}{\multirow{2}{*}{Conformal Energy}} &
    \multicolumn{6}{c|}{Angle Distortion} \\  \cline{5-10}
    &&&&
    \multicolumn{3}{c|}{$50$-th percentile} &
    \multicolumn{3}{c|}{$75$-th percentile} \\ \cline{2-10}
    &
    \multirow{2}{*}{FLASH\cite{PTKC15}} &
    \multirow{2}{*}{SCEM\cite{MHTL19}} &
    \multirow{2}{*}{HBTR} &
    \multirow{2}{*}{FLASH\cite{PTKC15}} &
    \multirow{2}{*}{SCEM\cite{MHTL19}} &
    \multirow{2}{*}{HBTR} &
    \multirow{2}{*}{FLASH\cite{PTKC15}} &
    \multirow{2}{*}{SCEM\cite{MHTL19}} &
    \multirow{2}{*}{HBTR} \\
    & & & & & & & & & \\
    \hline
    Apple  & 3.74e-04 & 3.49e-04 & \bf{3.43e-04} & 0.162 & 0.156 & \bf{0.153} & 0.280 & 0.268 & \bf{0.264} \\
    Arnold  & 7.82e-03 & 9.07e-03 & \bf{4.32e-03} & 0.676 & 0.675 & \bf{0.659} & 1.232 & \bf{1.173} & 1.233 \\
    Brain  & 2.69e-02 & 2.36e-02 & \bf{2.33e-02} & 0.821 & \bf{0.729} & \bf{0.729} & 1.849 & \bf{1.694} & \bf{1.694} \\
    Bunny   & 1.93e-03 & 1.90e-03 & \bf{1.85e-03} & 0.343 & \bf{0.339} & 0.363 & 0.759 & 0.746 & \bf{0.735} \\
    Fandisk    & 3.50e-02 & 1.39e-02 & \bf{1.34e-02} & 1.071 & \bf{0.936} & \bf{0.936} & 1.974 & 1.646 & \bf{1.640} \\
    Horse   & 4.58e-03 & 4.30e-03 & \bf{4.18e-03} & 1.027 & 1.020 & \bf{1.010} & 2.292 & \bf{2.283} & \bf{2.283} \\
    Planck   & 1.71e-03 & 1.19e-03 & \bf{1.13e-03} & 0.311 & 0.245 & \bf{0.231} & 0.578 & 0.457 & \bf{0.446} \\
    Venus   & 6.66e-03 & 3.54e-03 & \bf{3.27e-03} & 0.666 & 0.453 & \bf{0.442} & 1.186 & 0.831 & \bf{0.821} \\
    \hline
    \end{tabular}
    }
    \caption{Comparison of conformal energies and angle distortions between FLASH, SCEM and HBTR. }
    \label{tab:comparison}
\end{table}

\subsection{Convergence behavior of the discrete scheme} \label{subsec:convdiscrete}
In this subsection, we check the numerical convergence of the discrete conformal energy of the resulting map to the continuous energy in \eqref{def: continuous conformal energy}.
We consider two ellipsoids with semiaxis lengths of $(1.1,1,0.9)$ and $(2.0,1,0.3)$ and generate triangular meshes with different resolutions. The basic information of the meshes is in \Cref{tab:ellip}, where $h$ represents the maximum diameter of triangles in the mesh.
Then, we use FLASH, SCEM and HBTR to compute conformal parameterizations. The conformal energies, means and SDs of angle distortion are used to measure the conformal distortion of the algorithms. \Cref{fig:ellip} shows the relationship between the measurements and $h$. The $x$-axis represents $h$ and $y$-axis represents the conformal energy, mean and SD of angle distortion, respectively.
We can see that the conformal energy, mean and SD do not stably decrease as $h$ decreases for FLASH, while $E_C = O(h^2)$ and angle distortions are linearly related to $h$ for SCEM and HBTR. Specifically, as $h\to \frac{1}{2} h$, $E_C \to \frac{1}{4} E_C$, the means and SDs are reduced by half. Therefore, SCEM and HBTR are robust in the respective of convergence of the discrete scheme.
The conformal energies, means and SDs are lowest for HBTR compared with those of FLASH and SCEM, demonstrating the robustness and accuracy of HBTR.

\begin{table}[h]
\centering
\begin{tabular}{c|c|cccccc}
\hline
    \multicolumn{2}{c|}{$\#V$} & 642 & 2562 & 10242 & 40962  & 163842 & 655362 \\
    \hline
    \multicolumn{2}{c|}{$\#F$} & 1280 & 5120 & 20480 & 81920 & 327680 & 1310720 \\
    \hline
    \multirow{2}{*}{$h$} & $(1.1,1,0.9)$ & 0.1796 & 0.0901 & 0.0451 & 0.0226 & 0.0113 & 0.0056 \\    \cline{2-8}
    & $(2.0,1,0.3)$ & 0.3192 & 0.1607 & 0.0806 & 0.0403 & 0.0202 & 0.0101 \\
 \hline
\end{tabular}
\caption{The ellipsoids meshes for checking the convergence of the discrete scheme.}
\label{tab:ellip}
\end{table}


\begin{figure}
    \centering
\subfloat[Conformal energy $(1.1,1,0.9)$]{
    \includegraphics[clip,trim = {0.8cm 0cm 1cm 0.5cm},width = 0.3\textwidth]{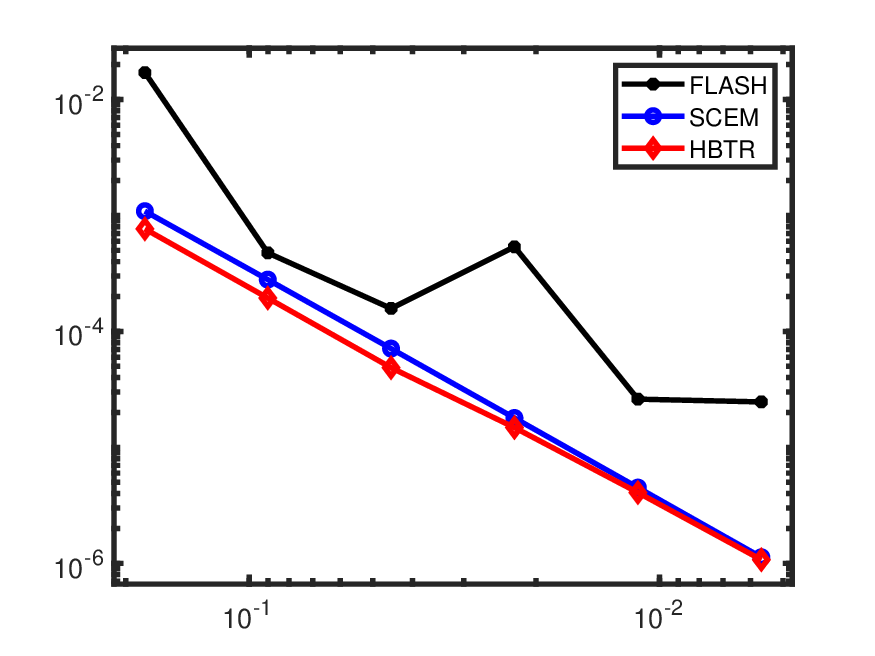}
}
\subfloat[Mean $(1.1,1,0.9)$]{
    \includegraphics[clip,trim = {0.8cm 0cm 1cm 0.5cm},width = 0.3\textwidth]{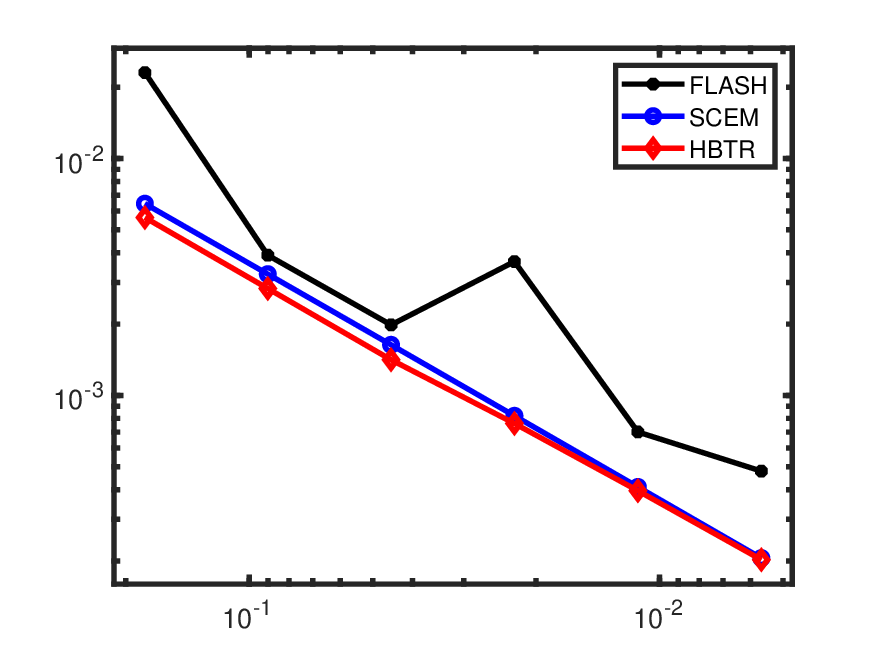}
}
\subfloat[Standard deviation $(1.1,1,0.9)$]{
    \includegraphics[clip,trim = {0.8cm 0cm 1cm 0.5cm},width = 0.3\textwidth]{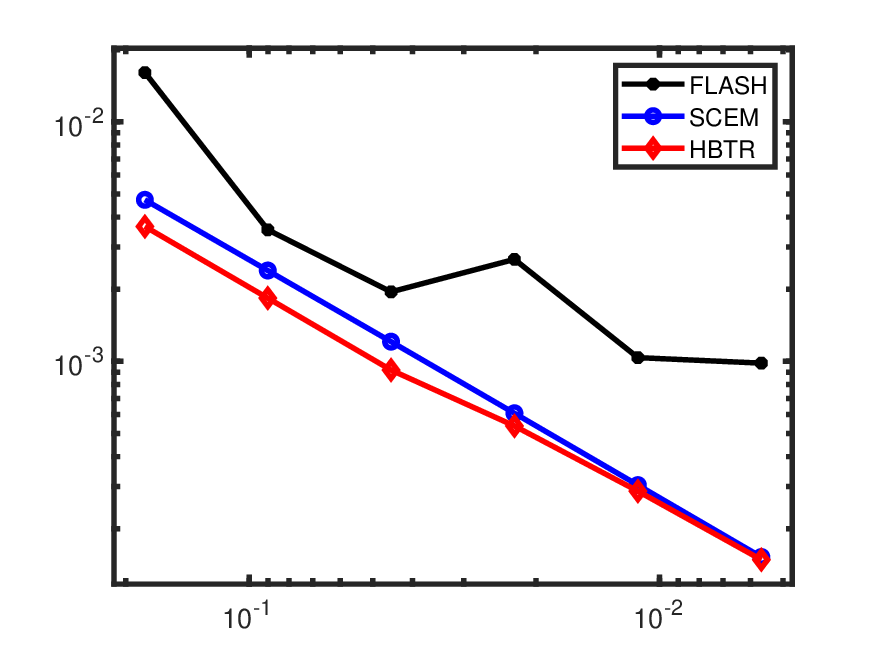}
} \\
\subfloat[Conformal energy $(2.0,1,0.3)$]{
    \includegraphics[clip,trim = {0.8cm 0cm 1cm 0.5cm},width = 0.3\textwidth]{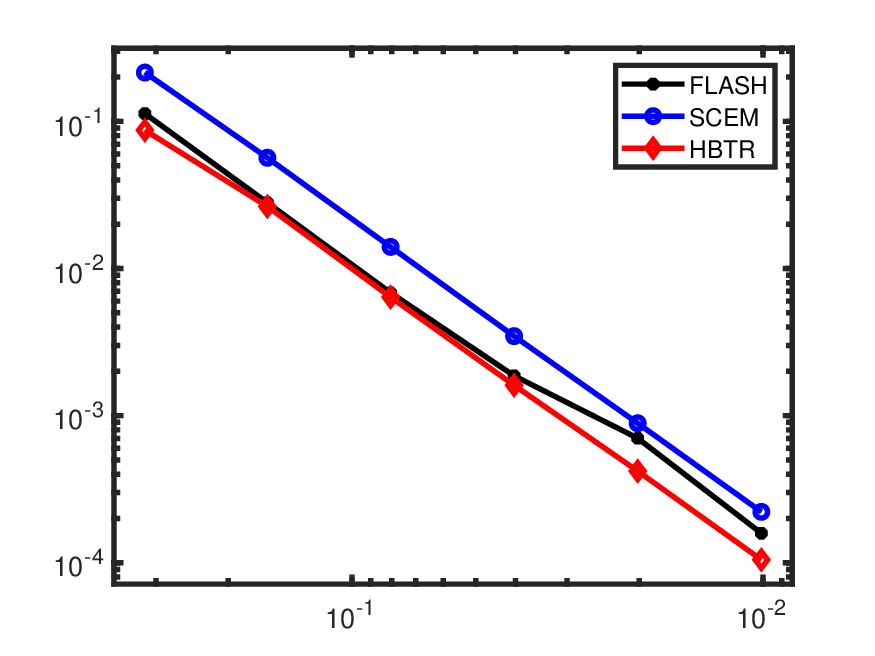}
}
\subfloat[Mean $(2.0,1,0.3)$]{
    \includegraphics[clip,trim = {0.8cm 0cm 1cm 0.5cm},width = 0.3\textwidth]{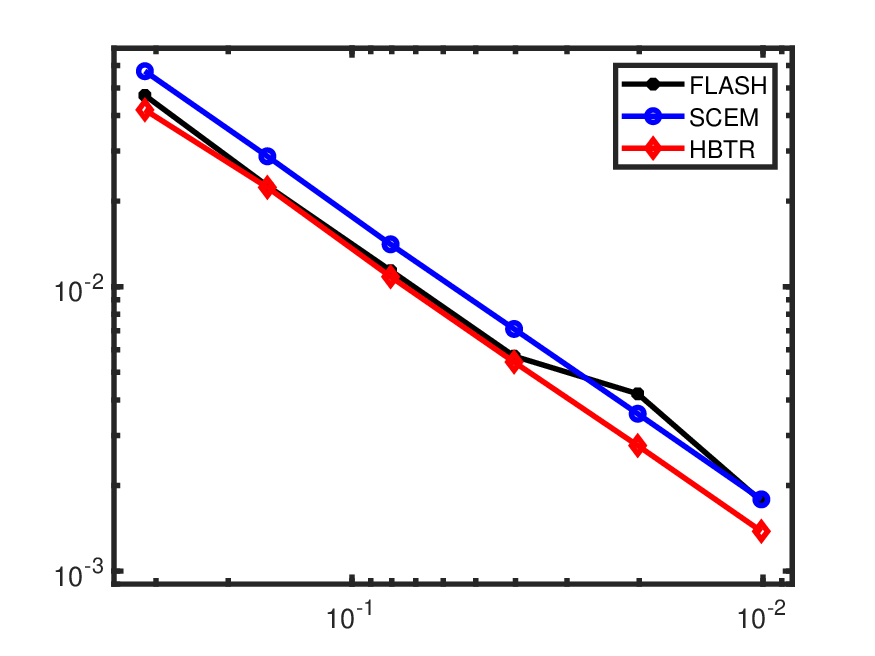}
}
\subfloat[Standard deviation $(2.0,1,0.3)$]{
    \includegraphics[clip,trim = {0.8cm 0cm 1cm 0.5cm},width = 0.3\textwidth]{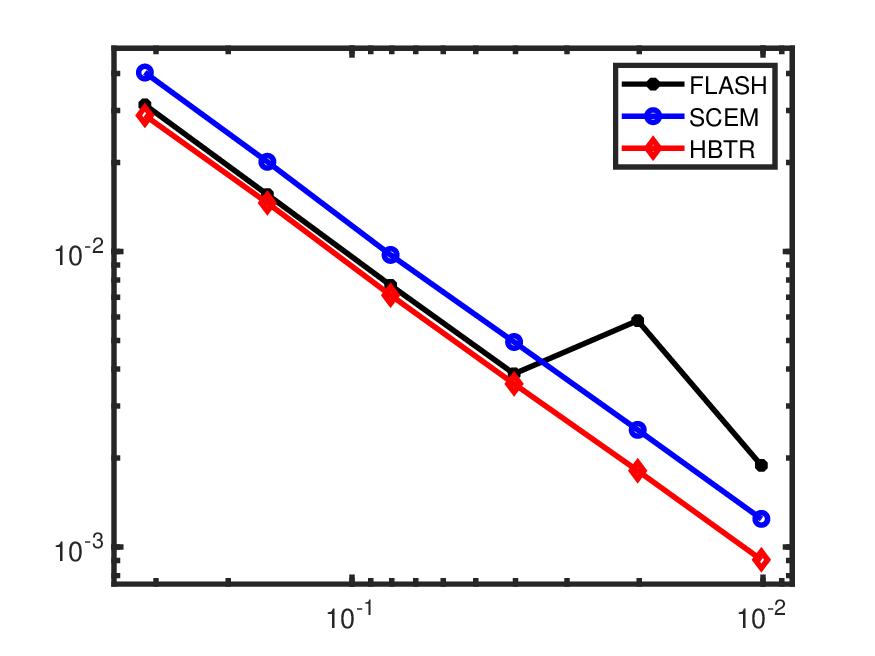}
}
\caption{The relationship between conformal energy, mean and standard deviation of angle distortion and $h$ of algorithms FLASH, SCEM and HBTR for $2$ ellipsoids. The figures from the top row to the bottom row are for ellipsoids $(1.1,1,0.9)$ and $(2.0,1,0.3)$, respectively.}
    \label{fig:ellip}
\end{figure}

\subsection{Removement of folding triangles}
HBTR does not necessarily guarantee the bijectivity of the resulting map; that is, folding triangles may occur in the image region $\mathbb{S}^2$. In this subsection, we apply a postprocessing method, named mean value coordinates \cite{Floa03}, to remove the folding triangles. Let $L_{MV}$ be a Laplacian matrix defined as
\begin{align} \label{def:LMVC}
    \big[L_{MV}\big]_{ij} = \begin{cases}
        -w_{MV,ij} & \text{if}~i\neq j,~[v_i,v_j] \in \mathcal{E}(M),\\
        \sum_{k\in \mathcal{N}(i)} w_{MV,ik} & \text{if}~i = j,\\
        0 & \text{if}~[v_i,v_j] \notin \mathcal{E}(M),
    \end{cases}
\end{align}
with
\[
w_{MV,ij} = \frac{\tan(\alpha_{jk}/2)+\tan(\alpha_{k'j}/2)}{\|v_{ij}\|},
\]
where $\alpha_{jk}$ and $\alpha_{k'j}$ are the angles opposite to vertex $v_i$ in triangles $T_{ijk}$ and $T_{k'ji}$, respectively, as shown in \Cref{subfig:OppositeAngle}. The postprocessing method is concluded in \Cref{alg:MVC}.

\begin{algorithm}[h]
\caption{Mean value coordinates for removing folding triangles}
\label{alg:MVC}
\begin{algorithmic}[1]
\REQUIRE Triangulation $M$ with vertices $\{v_i,i = 1,2,\cdots, n\}$, and $\mathbf{f} \in \mathbb{R}^{n\times 3}$ obtained by \Cref{alg:TR}.
    \ENSURE $\mathbf{f} \in \mathbb{R}^{n\times 3}$ inducing the conformal map $f$ as in \eqref{eq:bary} guaranteeing the bijectivity.
    \STATE Generate Laplacian matrix $L_{MV}$ defined in \eqref{def:LMVC}.
    \STATE Search the folding triangles in $f(M)$. Let $\mathtt{I} = \{i~|~T_{ijk} \text{ is a folding triangle.}\}$ be the index set of vertices contained in the folding triangles and $\mathtt{O} = \{1,2,\cdots,n\}\setminus \mathtt{I}$.
    \WHILE{$\mathtt{I} \neq \varnothing$}
    \STATE Select the face center of an unfolding triangle as north pole and perform the stereographic projection $\Pi(\mathbf{f}) \to \mathbf{h}$.
    \STATE Update the vertices $\mathbf{h}_{\mathtt{I}}$ by solving the linear system,
\begin{align}
        \big[L_{MV}\big]_{\mathtt{I}\mathtt{I}} \mathbf{h}_{\mathtt{I}} = -\big[L_{MV}\big]_{\mathtt{I}\mathtt{O}} \mathbf{h}_{\mathtt{O}}.
    \end{align}
    \STATE Perform the inverse stereographic projection $\Pi^{-1}(\mathbf{h}) \to \mathbf{f}$ and update $\mathtt{I}$ and $\mathtt{O}$.
    \ENDWHILE
\end{algorithmic}
\end{algorithm}
Unlike $L$ in conformal energy in \eqref{def:CE}, $L_{MV}$ is not symmetric. However, its weights $w_{MV}$ must be positive. Therefore, this approach can guarantee the bijectivity of the modified map. We present $2$ examples, the resulting maps of which by HBTR are not bijective. Then, we use \Cref{alg:MVC} to remove the folding triangles. \Cref{tab:MVC} shows the conformal energies, angle distortions and the number of folding triangles before and after the removal. The conformal energies decrease slightly, and the angle distortions are almost unchanged, while the folding triangles disappear.

\begin{table}[h]
    \centering
    \begin{tabular}{c|cc|c|c|c}
    \hline
        Mesh & $\# V$ & $\# F$ & Conformal energy & Mean of angle distortion & $\#$folding\\
    \hline
        Bimba & 502575 & 1005146 & 1.132e-3/~1.113e-3 & 5.328e-3/~5.328e-3 & 24/~0 \\
        RightBrain & 163842 & 327680 & 2.806e-2/~2.799e-2 & 2.252e-2/~2.252e-2 & 34/~0\\
     \hline
    \end{tabular}
    \caption{The result of postprocessing for removing the folding triangles. The left and right of '/' are the values before and after the postprocessing, respectively.}
    \label{tab:MVC}
\end{table}

\section{Application to surface registrations} \label{sec:registration}

Given a fixed surface ${M}_0$ and a series of moving surfaces ${{M}_t,t=1,2,\cdots}$, surface registration aims to find bijective maps from the moving surfaces to the fixed surface. It is broadly applied in computer vision and medical imaging. The goal of surface registration is to transform surfaces from different sources into one coordinate system. Therefore, the registration should ensure that the predominant features in the fixed surface correspond to those in the target surfaces, which are often expressed as landmarks in practical applications. It is generally not easy to manage the registration because of the complicated structure of surfaces. With the help of parameterization, we can transform the closed fixed surface ${M}_0$ into a unit sphere $\mathbb{S}^2$ via conformal map $f_0$ and then register the moving surfaces to the obtained unit sphere $\mathbb{S}^2$ via registration map $f_{reg,t}$. As a result, the map $\tilde{f}_{reg,t} = f_0^{-1} \circ {f}_{reg,t}$ is the registration map from ${M}_t$ to ${M}_0$. To obtain $f_{reg,t}$, we consider the optimization problem
\begin{align} \label{opt:reg}
    \min E_B(\boldsymbol{\theta},{\boldsymbol{\phi}}) := E_C(\boldsymbol{\theta},{\boldsymbol{\phi}}) + \lambda E_{reg}(\boldsymbol{\theta},{\boldsymbol{\phi}}),
\end{align}
where $E_{reg}$ is the registration loss and $\lambda$ is its parameter. The representation of $E_{reg}$ depends on the
expression of predominant features. The most common representation is the landmark-based registration
\begin{align} \label{eq:landmarkbase}
    E_{reg} = \frac{1}{2|S_L|}\sum_{i\in S_L} \|\mathbf{f}_{t,i} - \mathbf{f}_{0,i}\|_F^2, %
\end{align}
where $\mathbf{f}_{t,i}$ are the feature vertices on surface $M_t$ and $S_L$ and $|S_L|$ are the indices set and the number of the landmark vertices, respectively. Landmark-based registration aims to align the landmark vertices such that the features of the surfaces are also aligned.

The conformal energy term guarantees the conformality of the map, while the registration loss term aligns the features of the surfaces. Therefore, the combination of the conformal energy and the registration loss results in a conformal (as possible) registration map. The conformal registration map is an elastic registration and preserves the local shape of the surface, which is widely used in the field of medical imaging. For the optimization problem \eqref{opt:reg}, it is easy to derive the gradient vector and Hessian matrix of registration loss generally. Benefiting from their simple representations, we can also utilize HBTR to solve it. Moreover, the conformal energy is invariant up to arbitrary rotation on $\mathbb{S}^2$. Therefore, we introduce an optimal rotation to further decrease the registration loss. Based on \Cref{alg:TR}, we present the following spherical conformal registration algorithm.

\begin{algorithm}[thp]
\caption{HBTR for spherical conformal registration}
\label{alg:TRREG}
\begin{algorithmic}[1]
\REQUIRE Fixed surface $M_0$ and moving surface $M_t$ with landmarks, registration parameter $\lambda$, tolerance $\varepsilon$.
    \ENSURE The conformal registration map $\tilde{f}_{reg,t}$.
    \STATE Compute the spherical conformal map of $M_0$ by \Cref{alg:TR}, denoted as $f_0$.
    \STATE Set $k = 0$ and $\delta^{(0)} = +\infty$.
    \STATE Compute the initial guess $\mathbf{f}$ of $M_t$.
    \STATE Compute the optimal rotation $R$ according to the landmarks, update $\mathbf{f} \gets \mathbf{f}R$ and compute the corresponding term $E_B^{(0)}$ in \eqref{opt:reg}.
    \WHILE{$\delta>\varepsilon$}
    \STATE Compute gradient vector $\mathbf{g}$ and Hessian matrix $H$ of \eqref{opt:reg}.
    \STATE Solve the linear system $H\mathbf{s} = -\mathbf{g}$ via block LU decomposition.
    \STATE Solve the trust region subproblem \eqref{opt:subtr2D} to get the trial step $\mathbf{d}$.
    \STATE Compute the optimal rotation $R$ according to the landmarks to update $\mathbf{d}$.
    \STATE Let $E \gets E_B\big((\boldsymbol{\theta},{\boldsymbol{\phi}}) + \mathbf{d}\big)$. If $E_B^{(k)} > E$, set $k \gets k+1$ and update
\begin{align*}
        &(\boldsymbol{\theta},{\boldsymbol{\phi}}) \gets (\boldsymbol{\theta},{\boldsymbol{\phi}}) + \mathbf{d},\\
        &E_B^{(k)} \gets E.
    \end{align*}
    \STATE Compute the error $\delta^{(k)}$ by \eqref{eq:error} and tune the trust region radius $\mit\Delta$.
    \ENDWHILE
    \STATE Let $\mathbf{f}_{reg,t} = [\cos\boldsymbol{\theta} \odot \sin{\boldsymbol{\phi}},
\sin\boldsymbol{\theta} \odot \sin{\boldsymbol{\phi}},
\cos{\boldsymbol{\phi}}]$. Compute $\tilde{f}_{reg,t} = f_0^{-1}\circ f_{reg,t}$, which is the registration map of $M_t$.
\end{algorithmic}
\end{algorithm}

To present the registration performance of our method, we
take $5$ right brain cortex meshes {\it RBrain0 - RBrain4} from the Human Connectome Project \cite{HCP} as an example, which are shown in the top row of \Cref{fig:Brain4Reg}, in which the regions in different colors represent different parts of the brain. We select $3$ landmark curves for each brain, which are in red, green and cyan, respectively, as shown in the top row of \Cref{fig:Brain4Reg}. We select {\it RBrain0} as the fixed surface and register {\it RBrain1 - RBrain4} to it. The parameter $\lambda$ in \eqref{opt:reg} is chosen as $1,5,10$ successively.
The middle row shows the resulting spheres by solving the registration problem \eqref{opt:reg} with the landmark curves, respectively, with $\lambda = 5$. The bottom row shows the corresponding registered brains. The high similarity of the landmark curves between the fixed brain and registered brains illustrates the well performance of our method. Notably, all registration maps are bijective. \Cref{tab:Reg} presents their conformal energies, angle distortions and registration losses. As the parameter $\lambda$ increases, the conformal energies and angle distortions remain low.

\begin{figure}[h]
    \centering
\begin{tabular}{c@{~}c@{~}c@{~}c@{~}c}
        \includegraphics[clip,trim = {2.2cm 2.0cm 1.8cm 1.4cm},width = 0.18\textwidth]{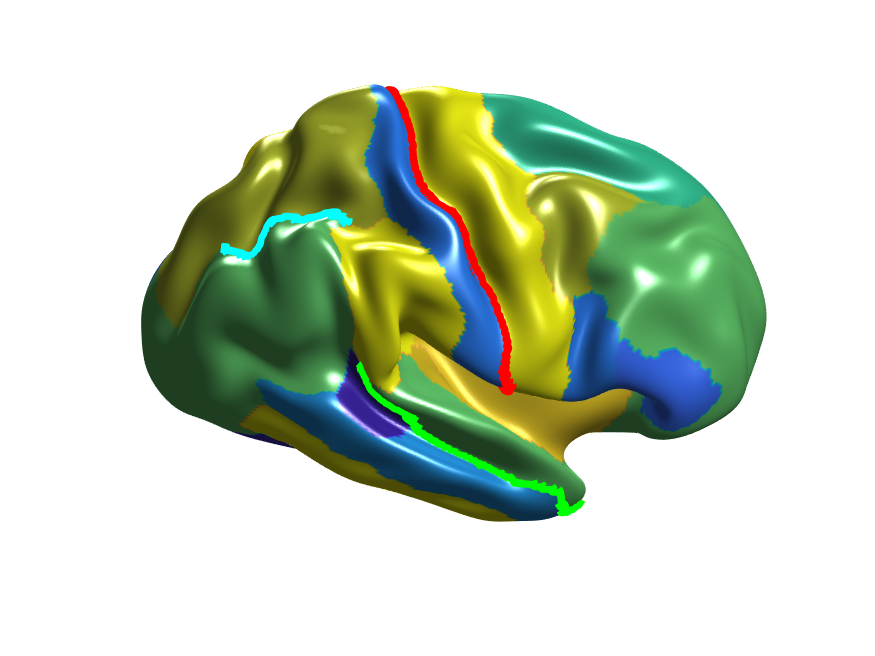} &
        \includegraphics[clip,trim = {2.2cm 2.0cm 1.8cm 1.4cm},width = 0.18\textwidth]{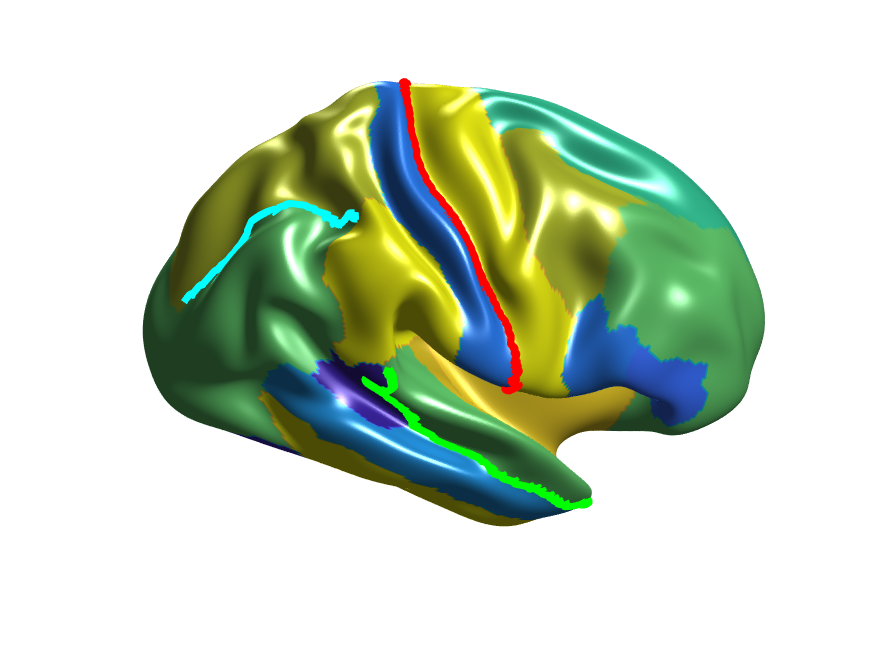} &
        \includegraphics[clip,trim = {2.2cm 2.0cm 1.8cm 1.4cm},width = 0.18\textwidth]{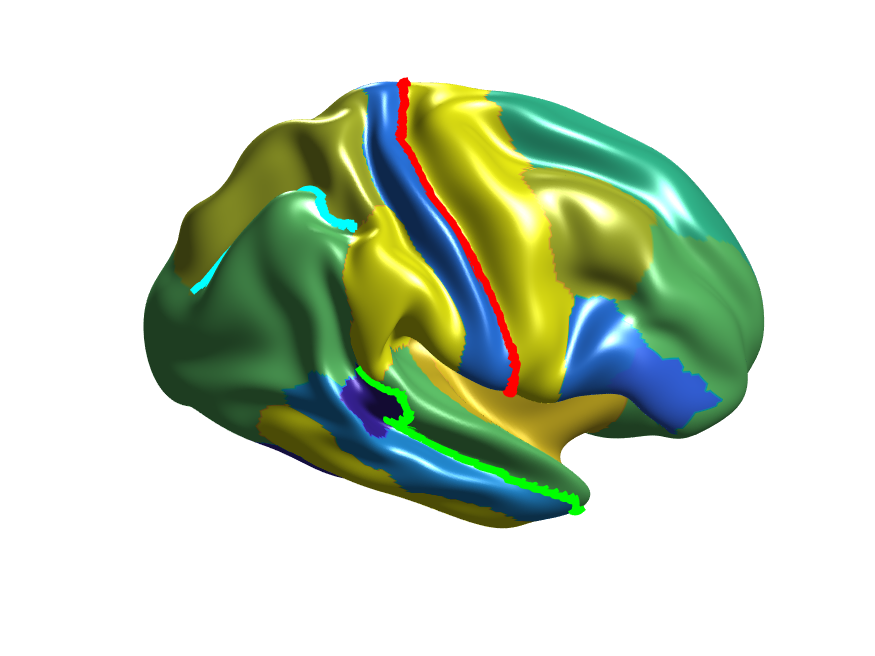} &
        \includegraphics[clip,trim = {2.2cm 2.0cm 1.8cm 1.4cm},width = 0.18\textwidth]{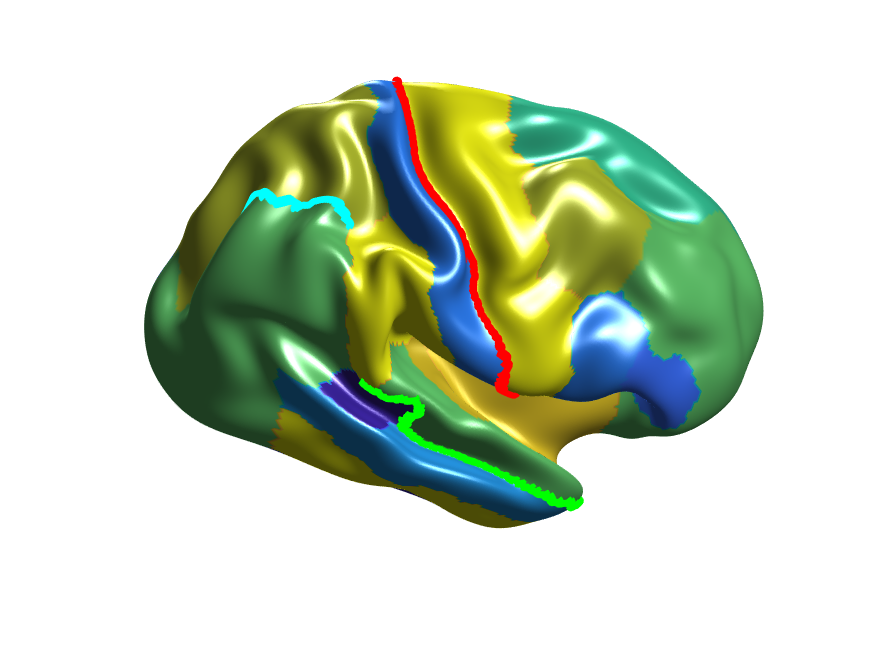} &
        \includegraphics[clip,trim = {2.2cm 2.0cm 1.8cm 1.4cm},width = 0.18\textwidth]{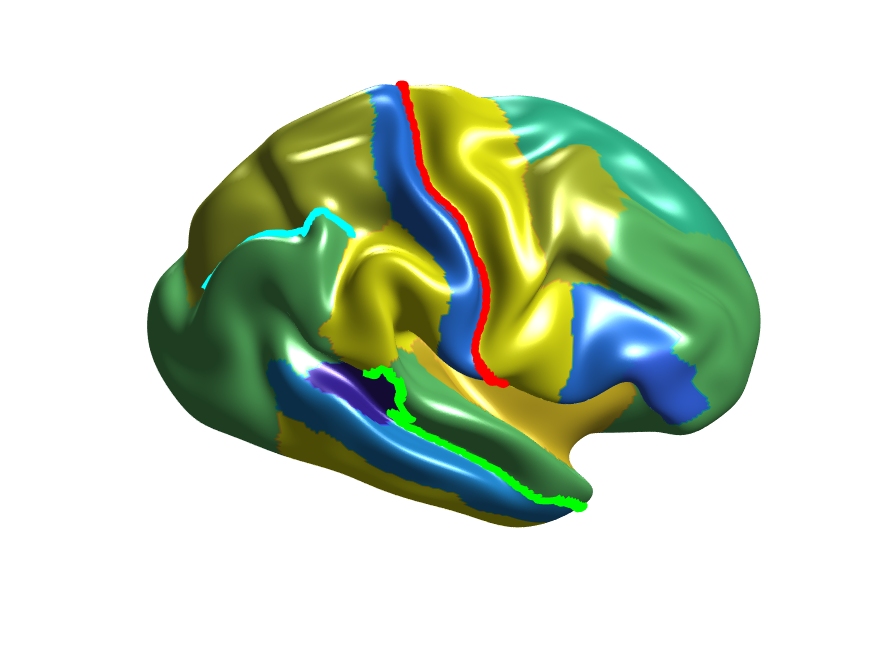} \\
        \includegraphics[clip,trim = {3cm 1cm 2.5cm 0.5cm},width = 0.18\textwidth]{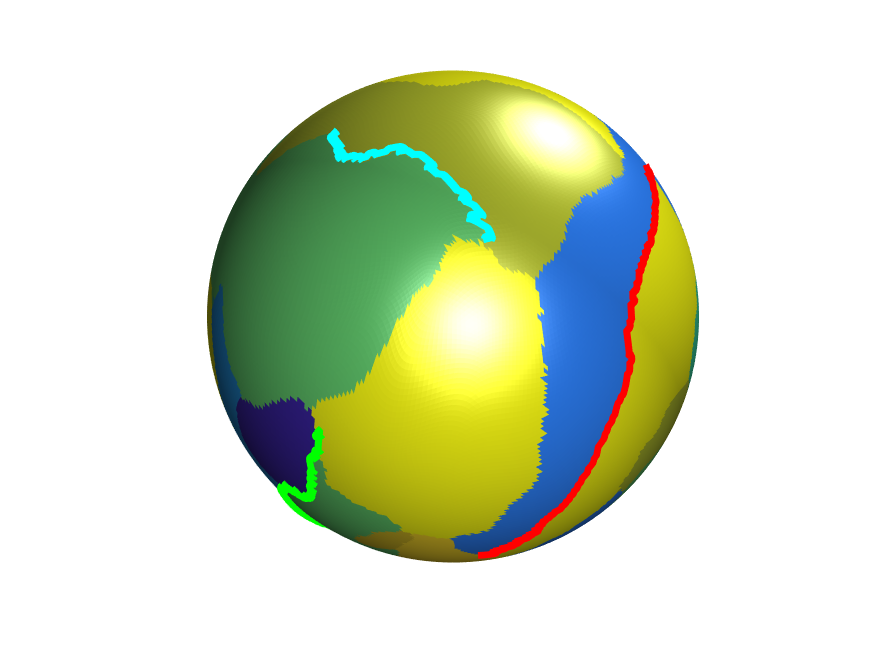} &
        \includegraphics[clip,trim = {3cm 1cm 2.5cm 0.5cm},width = 0.18\textwidth]{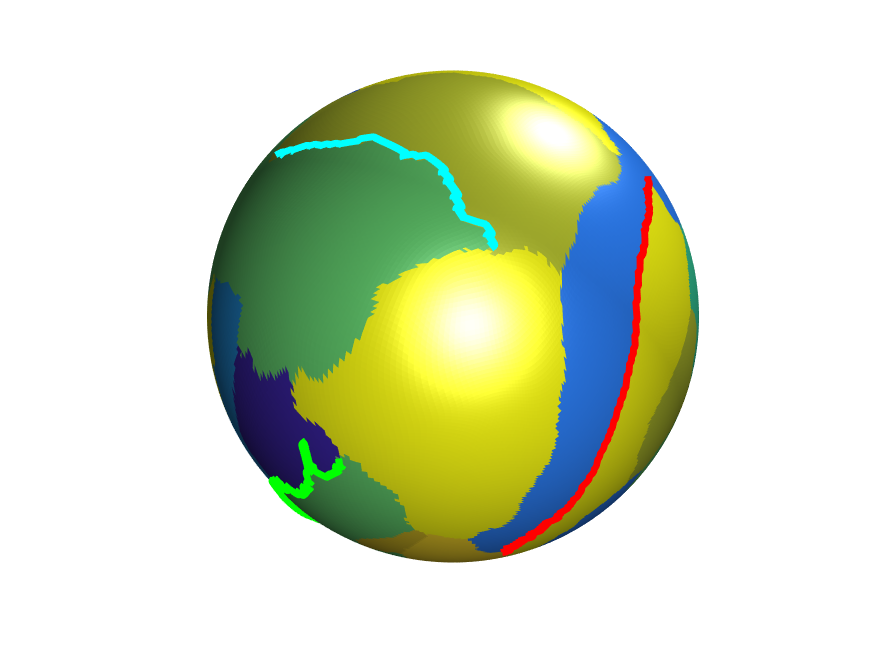} &
        \includegraphics[clip,trim = {3cm 1cm 2.5cm 0.5cm},width = 0.18\textwidth]{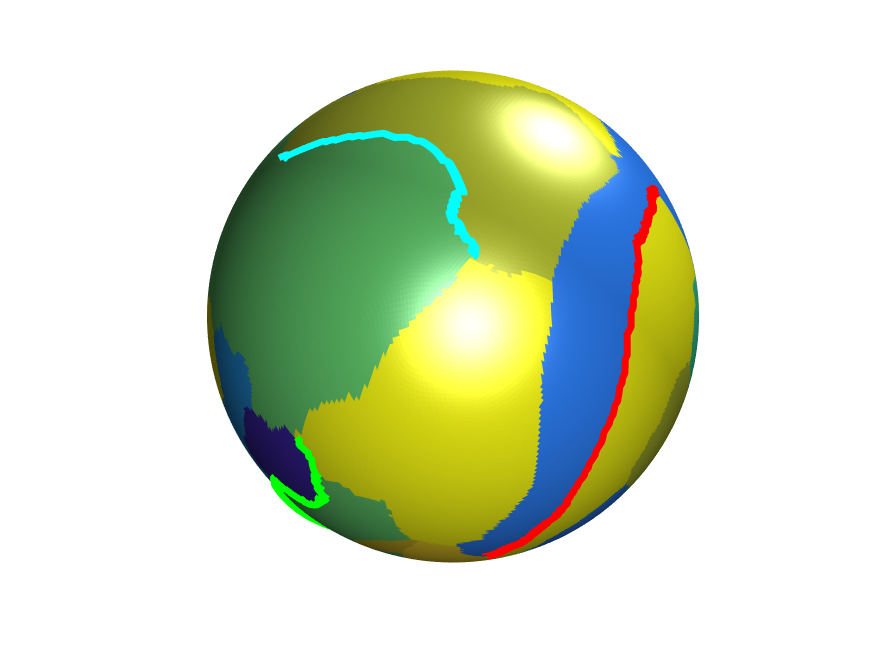} &
        \includegraphics[clip,trim = {3cm 1cm 2.5cm 0.5cm},width = 0.18\textwidth]{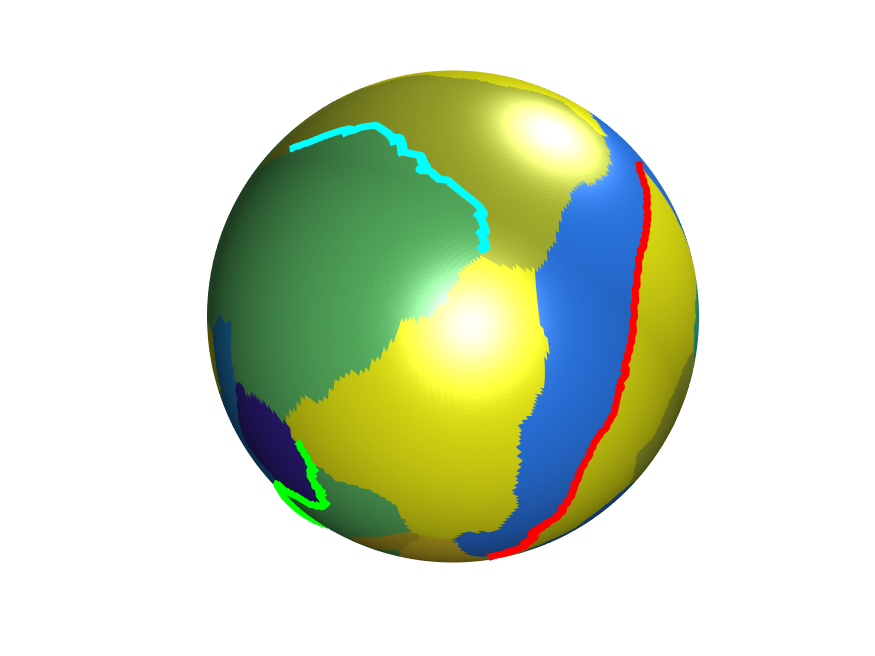} &
        \includegraphics[clip,trim = {3cm 1cm 2.5cm 0.5cm},width = 0.18\textwidth]{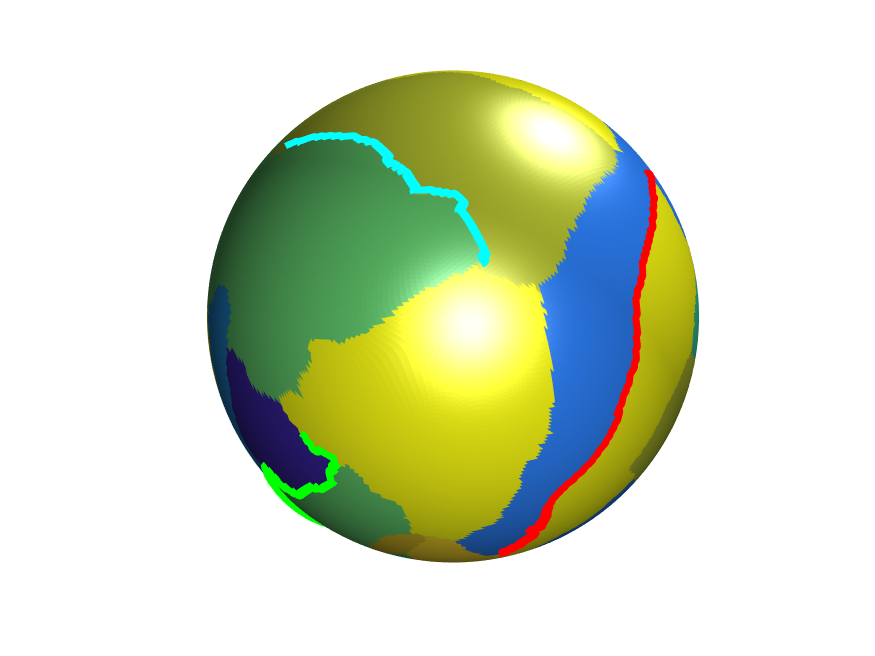}\\
        \includegraphics[clip,trim = {2.2cm 2.0cm 1.8cm 1.4cm},width = 0.18\textwidth]{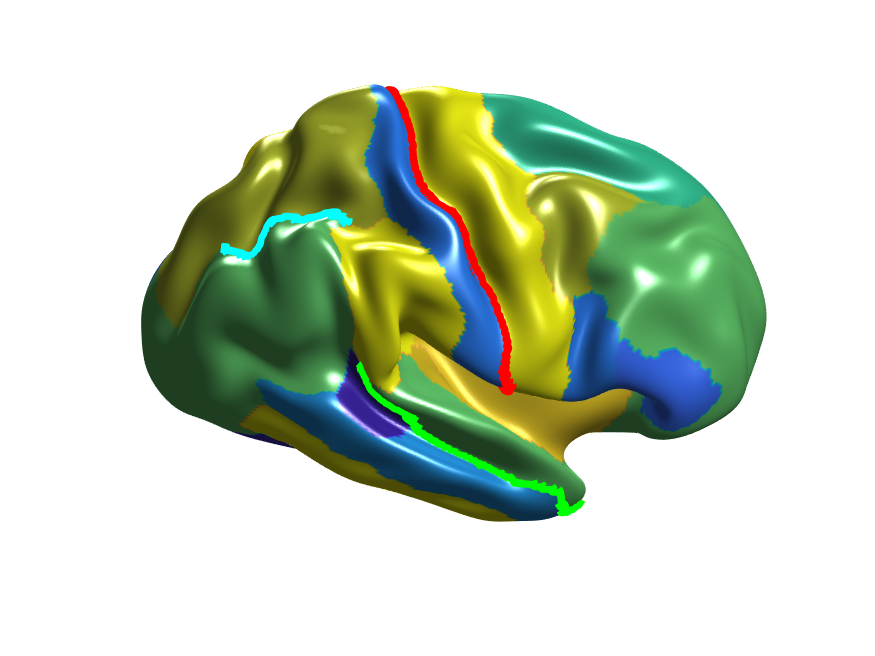} &
        \includegraphics[clip,trim = {2.2cm 2.0cm 1.8cm 1.4cm},width = 0.18\textwidth]{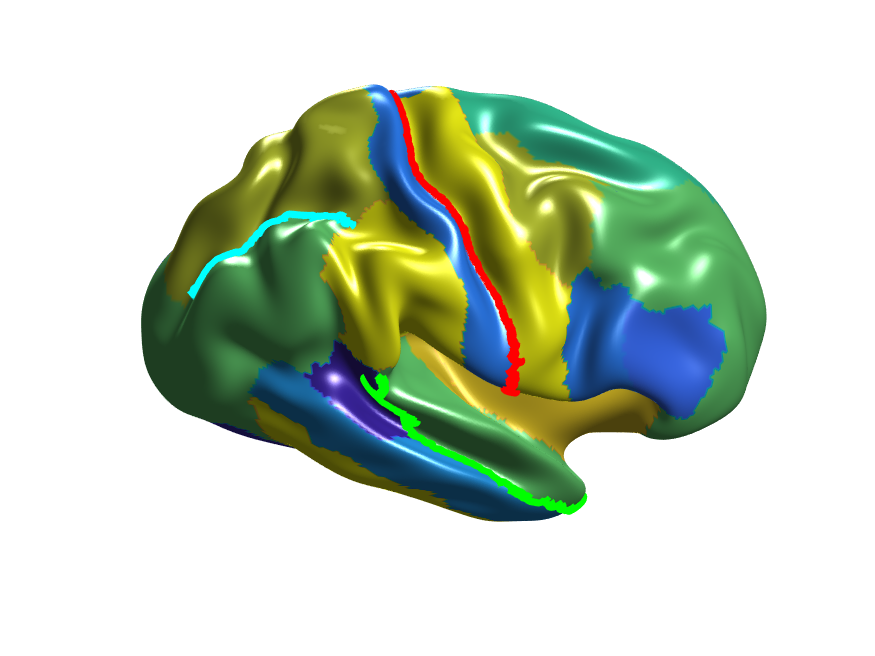} &
        \includegraphics[clip,trim = {2.2cm 2.0cm 1.8cm 1.4cm},width = 0.18\textwidth]{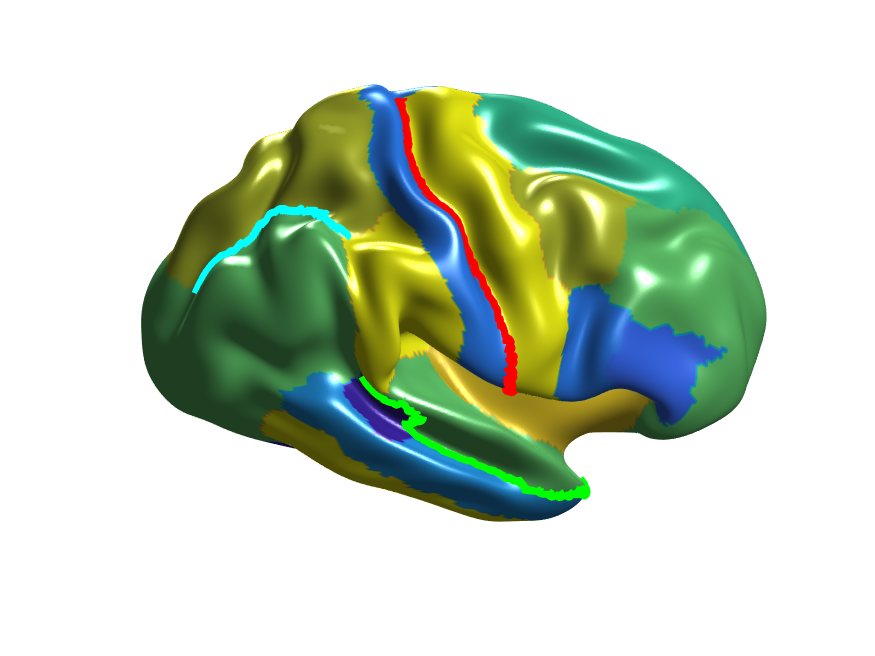} &
        \includegraphics[clip,trim = {2.2cm 2.0cm 1.8cm 1.4cm},width = 0.18\textwidth]{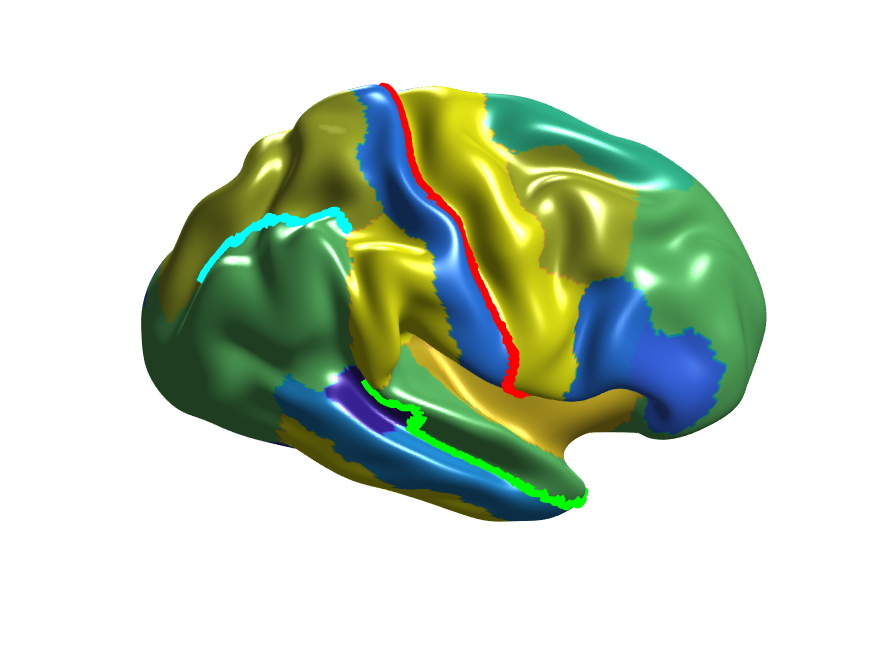} &
        \includegraphics[clip,trim = {2.2cm 2.0cm 1.8cm 1.4cm},width = 0.18\textwidth]{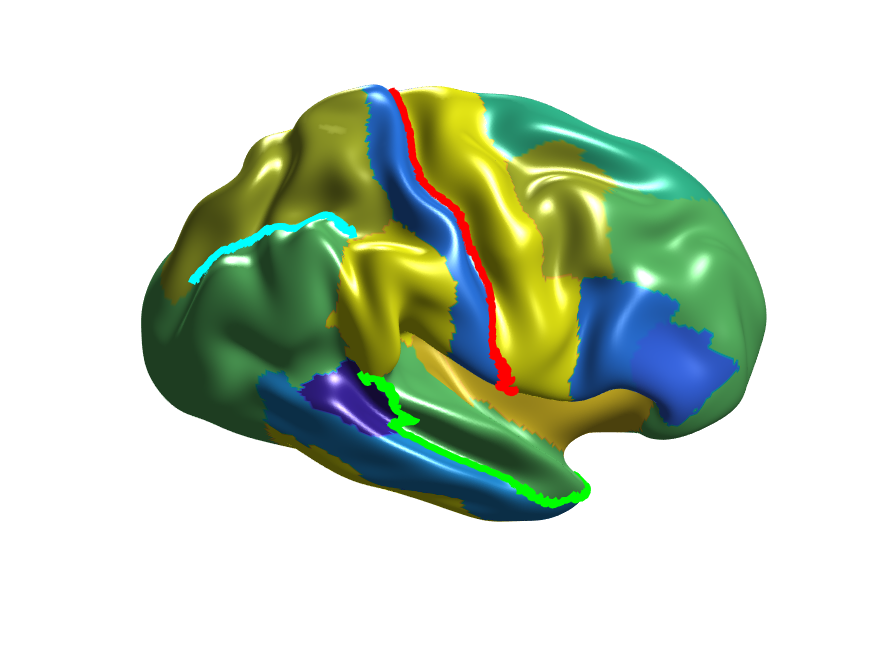} \\
RBrain0 (fixed) & RBrain1 & RBrain2 & RBrain3 & RBrain4
\end{tabular}
\caption{The right brains for registration, the registered spheres and the registered brains from top to bottom with $\lambda = 5$. The curves in red, green and cyan are the landmark curves.}
    \label{fig:Brain4Reg}
\end{figure}

\begin{table}[h]
    \centering
    \begin{tabular}{c|c|c|c|c|c}
        \hline
        \multirow{2}{*}{Mesh} & \multirow{2}{*}{$\lambda$} &
        Conformal &
        \multicolumn{2}{c|}{Angle distortion} &
        Registration \\
        \cline{4-5}
        & & energy & Mean & SD & loss \\
        \hline
        \multirow{3}{*}{RBrain1}
        & 1 & 1.31e-03 & 3.56e-01 & 3.13e-01 & 4.95e-03 \\
        & 5 & 3.41e-03 & 5.53e-01 & 5.27e-01 & 3.55e-03 \\
        & 10 & 4.73e-03 & 6.28e-01 & 6.21e-01 & 1.39e-03 \\
        \hline
        \multirow{3}{*}{RBrain2}
        & 1 & 1.18e-03 & 3.50e-01 & 3.20e-01 & 4.35e-03 \\
        & 5 & 3.00e-03 & 5.31e-01 & 5.43e-01 & 3.90e-03 \\
        & 10 & 8.31e-03 & 7.55e-01 & 8.92e-01 & 3.00e-03 \\
        \hline
        \multirow{3}{*}{RBrain3}
        & 1 & 9.98e-04 & 3.28e-01 & 2.89e-01 & 3.87e-03 \\
        & 5 & 2.34e-03 & 4.69e-01 & 4.42e-01 & 2.66e-03 \\
        & 10 & 5.28e-03 & 5.86e-01 & 6.55e-01 & 1.83e-03 \\
        \hline
        \multirow{3}{*}{RBrain4}
        & 1 & 1.05e-03 & 3.35e-01 & 3.02e-01 & 2.67e-03 \\
        & 5 & 2.63e-03 & 4.93e-01 & 4.82e-01 & 3.04e-03 \\
        & 10 & 6.65e-03 & 7.38e-01 & 7.41e-01 & 2.40e-03 \\
        \hline
    \end{tabular}
    \caption{The conformal energies, angle distortions and registration losses of registered brain surfaces with $\lambda = 1,5,10$, respectively.}
    \label{tab:Reg}
\end{table}

\section{Conclusions} \label{sec:conclusion}

In this paper, we employ spherical coordinates and directly solve the spherical CEM problem for the computation of the surface conformal parameterization. Then, we give the explicit derivations of the gradient vector and the Hessian matrix of the discrete conformal energy, which preserves the sparsity as the Laplacian matrix. Due to the sparsity of the Hessian matrix, the robust algorithm, called HBTR, is developed to solve the spherical CEM problem. HBTR sufficiently combines the local quadratic convergence and continuing descent advantages of the gradient and the Newton directions. The numerical experiments actually show the conformality, the robustness and the local quadratic convergence of the HBTR. For the discrete scheme, we also present the quadratic convergence of the discrete conformal energy to the continuous scheme. Since the gradient vector and Hessian matrix of the registration loss have simple representations, we utilize the HBTR to propose a modified version of HBTR for the application to surface registrations. Significantly, the modified algorithm \ref{alg:TRREG} also has quadratic convergence, suggesting its potential for extension to other applications of our method.

\section*{Acknowledgements}
T. Li was supported in parts by the National Natural Science Foundation of China (NSFC) 12371377. W.-W. Lin was partially supported by the Ministry of Science and Technology (MoST 110-2115-M-A49-004), Taiwan. This work was partially supported by National Centre of Theoretical Sciences (NCTS) in Taiwan. We thank Tianhe-2 and the Big Data Computing Center in Southeast University, China, for their support of our use of their computing resources.

\bibliographystyle{plain}

\end{document}

%% file: head.tex
\usepackage[utf8]{inputenc}
\usepackage{amsmath,amssymb,amsthm,bm}
\usepackage{amsfonts}
\usepackage{mathrsfs}
\usepackage{graphicx}
\usepackage{epstopdf}
\usepackage{enumitem}
\usepackage{overpic}
\usepackage{multirow}
\usepackage{nicefrac}
\numberwithin{equation}{section}
\usepackage{float}
\usepackage{caption}
\usepackage{multicol}
\usepackage{dsfont}
\usepackage{bbm}
\usepackage{bm}
\usepackage{verbatim}
\usepackage{algorithm,algorithmic}
\usepackage{xcolor}
\usepackage{arydshln}
\usepackage{subfig}

 \usepackage[colorlinks,linkcolor=blue,
            citecolor=blue]{hyperref}

\usepackage{geometry}
\usepackage{cleveref}

\usepackage{slashed}

\newcommand{\diag}{\mathrm{diag}}

\newcommand{\argmin}{\operatornamewithlimits{argmin}}

\newcommand{\vectx}{{\text{vec}}}

\newtheorem{theorem}{Theorem}

\newtheorem{lemma}{Lemma}

\newtheorem{remark}{Remark}

\providecommand{\keywords}[1]
{
  \small	
  \textbf{\textit{Key words---}} #1
}

\graphicspath{
{images/}
} 